\documentclass[12pt]{amsart}

\usepackage{amsmath,amsthm,amssymb,enumitem,verbatim,stmaryrd,xcolor,microtype,graphicx,aliascnt,needspace,soul}
\usepackage[T1]{fontenc}
\usepackage[utf8]{inputenc}
\usepackage[english]{babel} 
\usepackage[top=3.5cm,bottom=3.5cm,left=3.5cm,right=3.5cm]{geometry}
\usepackage[bookmarksdepth=2,linktoc=page,colorlinks,linkcolor={red!70!black},citecolor={red!80!black},urlcolor={blue!80!black}]{hyperref}
\usepackage[all,cmtip]{xy}
\usepackage{tikz}\usetikzlibrary{matrix,arrows,decorations.markings,shapes,calc,intersections,through}
\usepackage{tikz-cd}
\pgfrealjobname{congruencespaces}
\usepackage[mathcal]{euscript}                               

\usepackage{mathrsfs}                                        
\usepackage{mathptmx}                                        
\usepackage[colorinlistoftodos,bordercolor=orange,backgroundcolor=orange!20,linecolor=orange,textsize=footnotesize]{todonotes} \makeatletter \providecommand \@dotsep{5} \def\listtodoname{List of Todos} \def\listoftodos{\@starttoc{tdo}\listtodoname} \makeatother 

\allowdisplaybreaks 

\usepackage{etoolbox}
\makeatletter
\patchcmd{\@startsection}{\@afterindenttrue}{\@afterindentfalse}{}{}             
\patchcmd{\part}{\bfseries}{\bfseries\LARGE}{}{}
\patchcmd{\section}{\scshape}{\bfseries}{}{}\renewcommand{\@secnumfont}{\bfseries} 
\patchcmd{\@settitle}{\uppercasenonmath\@title}{\large}{}{}
\patchcmd{\@setauthors}{\MakeUppercase}{}{}{}
\addto{\captionsenglish}{} 
\addto{\captionsenglish}{} 
\addto{\captionsenglish}{} 
\makeatother

\usepackage{fancyhdr}

\addto\extrasenglish{}  
\theoremstyle{plain}
\newtheorem{thm}{Theorem}[section] 
\newaliascnt{lemma}{thm}\newtheorem{lemma}[lemma]{Lemma}\aliascntresetthe{lemma}
\newaliascnt{lem}{thm}\newtheorem{lem}[lem]{Lemma}\aliascntresetthe{lem}
\newaliascnt{cor}{thm}\newtheorem{cor}[cor]{Corollary}\aliascntresetthe{cor}
\newaliascnt{prop}{thm}\newtheorem{prop}[prop]{Proposition}\aliascntresetthe{prop}
\newtheorem{thmA}{Theorem} 
\newaliascnt{propA}{thmA}\aliascntresetthe{propA}

\newtheorem*{thm*}{Theorem}
\newtheorem*{lem*}{Lemma}
\newtheorem*{cor*}{Corollary}

\theoremstyle{definition}
\newaliascnt{df}{thm}\newtheorem{df}[df]{Definition}\aliascntresetthe{df}
\newaliascnt{defn}{thm}\newtheorem{defn}[defn]{Definition}\aliascntresetthe{defn}
\newaliascnt{rem}{thm}\newtheorem{rem}[rem]{Remark}\aliascntresetthe{rem}
\newaliascnt{ex}{thm}\newtheorem{ex}[ex]{Example}\aliascntresetthe{ex}

\newtheorem*{df*}{Definition}
\newtheorem*{ex*}{Example}
\newtheorem*{rem*}{Remark}
\newtheorem*{notation*}{Notation}

\theoremstyle{remark}

\DeclareRobustCommand{\gobblefour}[5]{}    
\numberwithin{equation}{section} \DeclareMathSizes{2}{10}{12}{13}

\DeclareFontFamily{OT1}{pzc}{}                                
\DeclareFontShape{OT1}{pzc}{m}{it}{<-> s * [1.10] pzcmi7t}{}
\DeclareMathAlphabet{\mathpzc}{OT1}{pzc}{m}{it}
\DeclareSymbolFont{sfoperators}{OT1}{bch}{m}{n} \DeclareSymbolFontAlphabet{\mathsf}{sfoperators} \makeatletter\def\operator@font{\mathgroup\symsfoperators}\makeatother 
\DeclareSymbolFont{cmletters}{OML}{cmm}{m}{it}              
\DeclareSymbolFont{cmsymbols}{OMS}{cmsy}{m}{n}
\DeclareSymbolFont{cmlargesymbols}{OMX}{cmex}{m}{n}
\DeclareMathSymbol{\myjmath}{\mathord}{cmletters}{"7C}     \let\jmath\myjmath 
\DeclareMathSymbol{\myamalg}{\mathbin}{cmsymbols}{"71}     \let\amalg\myamalg
\DeclareMathSymbol{\mycoprod}{\mathop}{cmlargesymbols}{"60}
\DeclareMathSymbol{\myalpha}{\mathord}{cmletters}{"0B}     \let\alpha\myalpha 
\DeclareMathSymbol{\mybeta}{\mathord}{cmletters}{"0C}      \let\beta\mybeta
\DeclareMathSymbol{\mygamma}{\mathord}{cmletters}{"0D}     \let\gamma\mygamma
\DeclareMathSymbol{\mydelta}{\mathord}{cmletters}{"0E}     \let\delta\mydelta
\DeclareMathSymbol{\myepsilon}{\mathord}{cmletters}{"0F}   \let\epsilon\myepsilon
\DeclareMathSymbol{\myzeta}{\mathord}{cmletters}{"10}      \let\zeta\myzeta
\DeclareMathSymbol{\myeta}{\mathord}{cmletters}{"11}       \let\eta\myeta
\DeclareMathSymbol{\mytheta}{\mathord}{cmletters}{"12}     \let\theta\mytheta
\DeclareMathSymbol{\myiota}{\mathord}{cmletters}{"13}      \let\iota\myiota
\DeclareMathSymbol{\mykappa}{\mathord}{cmletters}{"14}     \let\kappa\mykappa
\DeclareMathSymbol{\mylambda}{\mathord}{cmletters}{"15}    \let\lambda\mylambda
\DeclareMathSymbol{\mymu}{\mathord}{cmletters}{"16}        \let\mu\mymu
\DeclareMathSymbol{\mynu}{\mathord}{cmletters}{"17}        \let\nu\mynu
\DeclareMathSymbol{\myxi}{\mathord}{cmletters}{"18}        \let\xi\myxi
\DeclareMathSymbol{\mypi}{\mathord}{cmletters}{"19}        \let\pi\mypi
\DeclareMathSymbol{\myrho}{\mathord}{cmletters}{"1A}       \let\rho\myrho
\DeclareMathSymbol{\mysigma}{\mathord}{cmletters}{"1B}     \let\sigma\mysigma
\DeclareMathSymbol{\mytau}{\mathord}{cmletters}{"1C}       \let\tau\mytau
\DeclareMathSymbol{\myupsilon}{\mathord}{cmletters}{"1D}   \let\upsilon\myupsilon
\DeclareMathSymbol{\myphi}{\mathord}{cmletters}{"1E}       \let\phi\myphi
\DeclareMathSymbol{\mychi}{\mathord}{cmletters}{"1F}       \let\chi\mychi
\DeclareMathSymbol{\mypsi}{\mathord}{cmletters}{"20}       \let\psi\mypsi
\DeclareMathSymbol{\myomega}{\mathord}{cmletters}{"21}     \let\omega\myomega
\DeclareMathSymbol{\myvarepsilon}{\mathord}{cmletters}{"22}\let\varepsilon\myvarepsilon
\DeclareMathSymbol{\myvartheta}{\mathord}{cmletters}{"23}  \let\vartheta\myvartheta
\DeclareMathSymbol{\myvarpi}{\mathord}{cmletters}{"24}     \let\varpi\myvarpi
\DeclareMathSymbol{\myvarrho}{\mathord}{cmletters}{"25}    \let\varrho\myvarrho
\DeclareMathSymbol{\myvarsigma}{\mathord}{cmletters}{"26}  \let\varsigma\myvarsigma
\DeclareMathSymbol{\myvarphi}{\mathord}{cmletters}{"27}    \let\varphi\myvarphi

\DeclareMathOperator{\Hom}{Hom}

\DeclareMathOperator{\MSch}{MSch}
\DeclareMathOperator{\Mon}{Mon}

\DeclareMathOperator{\colim}{colim\,}

\DeclareMathOperator{\Frac}{Frac}

\DeclareMathOperator{\Spec}{Spec}
\DeclareMathOperator{\MSpec}{MSpec}
\DeclareMathOperator{\Cong}{Cong}
\DeclareMathOperator{\Congker}{Congker}
\DeclareMathOperator{\congker}{congker}
\DeclareMathOperator{\Nil}{Nil}
\DeclareMathOperator{\nil}{\mathfrak{nil}}
\DeclareMathOperator{\Top}{Top}
\DeclareMathOperator{\res}{res}

\newcommand\A{{\mathbb A}}

\newcommand\F{{\mathbb F}}

\newcommand\N{{\mathbb N}}

\newcommand\Z{{\mathbb Z}}

\newcommand\cA{{\mathcal A}}

\newcommand\cO{{\mathcal O}}

\newcommand\cU{{\mathcal U}}

\newcommand\fC{{\mathfrak C}}
\newcommand\fD{{\mathfrak D}}
\newcommand\fc{{\mathfrak c}}
\newcommand\fd{{\mathfrak d}}
\newcommand\fm{{\mathfrak m}}
\newcommand\fn{{\mathfrak n}}
\newcommand\fp{{\mathfrak p}}
\newcommand\fq{{\mathfrak q}}

\newcommand\til[1]{{\widetilde{#1}}}
\renewcommand{\phi}{\varphi}
\renewcommand{\setminus}{-}

\newcommand\Fun{{\F_1}}
\newcommand\id{\textup{id}}
\newcommand\im{\textup{im}}
\newcommand\pr{\textup{pr}}
\renewcommand\max{{\operatorname{max}}}

\renewcommand\geq{\geqslant}
\renewcommand\leq{\leqslant}
\newcommand{\congr}{\textup{cong}}

\newcommand{\triv}{\textup{triv}}
\newcommand{\red}{\textup{red}}
\newcommand{\sred}{\textup{sred}}
\newcommand{\vcl}{\textup{vcl}}
\newcommand{\overint}{{\overline{\textup{int}}}}
\newcommand{\gen}[1]{\langle #1 \rangle}

\title{The topological shadow of \texorpdfstring{$\Fun$}{F1}-geometry: congruence spaces}

\author{Oliver Lorscheid}
\address{\rm Oliver Lorscheid, University of Groningen, the Netherlands, and IMPA, Rio de Janeiro, Brazil}
\email{oliver@impa.br}

\author{Samarpita Ray}
\address{\rm Samarpita Ray, Indian Statistical Institute, Bangalore, India}
\email{ray.samarpita31@gmail.com}

\begin{document}

\begin{abstract} 
 In this paper we introduce congruence spaces, which are topological spaces that are canonically attached to monoid schemes and that reflect closed topological properties. This leads to satisfactory topological characterizations of closed morphisms and closed immersions as well as separated and proper morphisms. We study congruence spaces thoroughly and extend standard results from usual scheme theory to monoid schemes: a closed immersion is the same as an affine morphism for which the pullback of sections is surjective; a morphism is separated if and only if the image of the diagonal is a closed subset of the congruence space; a valuative criterion for separated and proper morphisms.
\end{abstract}

\maketitle 

\begin{small} \tableofcontents \end{small}

\section*{Introduction}

Monoid schemes are sometimes termed the \emph{minimalist approach} to $\Fun$-geometry and appear, in fact, as a special case of nearly every other approach to $\Fun$-geometry (cf.\ \cite{LopezPena-Lorscheid11}). They have been utilized in and connected with various other areas of geometry: most profoundly, monoid schemes are a generalization of toric geometry (\cite{Deitmar08}). They are used to study the $K$-theory of toric varieties in positive characteristic (\cite{CHWW14,CHWW18}), cdh descent for toric varieties (\cite{CHWW15}) and sheaf cohomology of toric vector bundles (\cite{Flores-Lorscheid-Szczesny17}). Stable homotopy groups of spheres and other spaces resurface as the $K$-theory of monoid schemes (\cite{Chu-Lorscheid-Santhanam12,Deitmar07,Eberhardt-Lorscheid-Young22}). They enter as an integral component into the theory of tropical schemes (\cite{Giansiracusa-Giansiracusa16,Giansiracusa-Giansiracusa22,Lorscheid23,Maclagan-Rincon18,Maclagan-Rincon20,Maclagan-Rincon22}). Other works on monoid schemes include 
\cite{
Berkovich12a,
Berkovich12b,
Cohn04, 
Bothmer-Hinsch-Stuhler11, 
Lorscheid-Szczesny18, 
Manin10,
Merida-Angulo-Thas17a,
Merida-Angulo-Thas17b,
Pirashvili12, 
Ray20, 
Szczesny12b,
Thas14,
Thas19}.

Monoid schemes are understood from two equivalent perspectives: as topological spaces together with a structure sheaf in monoids (\cite{Chu-Lorscheid-Santhanam12,Connes-Consani10a,CHWW15,Deitmar05}) and as sheaves on the category of monoids (\cite{Marty07,Toen-Vaquie09,Vezzani12}). The topological perspective gives monoid schemes a concrete shape as spaces of prime ideals. The topos theoretic interpretation justifies that prime ideals capture \emph{open conditions} on monoids. More explicitly, a family $\{U_i\to X\}_{i\in I}$ of open immersions is a covering family in the canonical topology for monoid schemes if and only if the $U_i$ cover $X$ as a topological space. As a consequence, every monoid is isomorphic to the global sections of its spectrum and the Yoneda functor of monoids into sheaves is fully faithful.

However, in contrast to usual scheme theory, prime ideals do not reflect \emph{closed conditions}. A prototypical example is the diagonal embedding $\Delta:\A^1_\Fun\to\A^2_\Fun$, which is not a closed topological embedding. This is illustrated in Figure \ref{fig: diagonal embedding line into plane}: the image of $\Delta$ (shaded in orange) is not among the closed subsets of $\A^2_\Fun$ (shaded in green).

\begin{figure}[htb]
 \[
 \beginpgfgraphicnamed{tikz/fig1}
  \begin{tikzpicture}[x=30pt,y=30pt]
   \draw[fill=green!50!black,fill opacity=0.2,rounded corners=11pt,thick] (4.5,1) -- (6,2.5) -- (7.5,1) -- (6,-0.5) -- cycle;   	
   \draw[fill=green!50!black,fill opacity=0.3,rounded corners=11pt,thick] (5.5,2) -- (6,2.5) -- (7.5,1) -- (7,0.5) -- cycle;   	
   \draw[fill=green!50!black,fill opacity=0.3,rounded corners=11pt,thick] (5.5,0) -- (6,-0.5) -- (7.5,1) -- (7,1.5) -- cycle;   	
   \draw[fill=orange!100!black,fill opacity=0.5,rounded corners=7pt,thick,draw=orange!50!black] (-0.25,0.75) rectangle (2.25,1.25);		
   \draw[fill=orange!100!black,fill opacity=1,rounded corners=7pt,thick,draw=orange!50!black] (1.75,0.75) rectangle (2.25,1.25);		
   \draw[fill=orange!100!black,fill opacity=0.5,rounded corners=7pt,thick,draw=orange!50!black] (4.75,0.75) rectangle (7.25,1.25);		
   \draw[fill=orange!100!black,fill opacity=1,rounded corners=7pt,thick,draw=orange!50!black] (6.75,0.75) rectangle (7.25,1.25);		
   \filldraw (0,1) circle (3pt);
   \filldraw (2,1) circle (3pt);
   \filldraw (5,1) circle (3pt);
   \filldraw (6,0) circle (3pt);
   \filldraw (6,2) circle (3pt);
   \filldraw (7,1) circle (3pt);
   \draw[->,thick] (3,1) -- (4,1);
   \node at (3.5,1.25) {$\Delta$};
  \end{tikzpicture}
 \endpgfgraphicnamed 
 \]
 \caption{The diagonal embedding $\Delta:\A^1_\Fun\to\A^2_\Fun$}
 \label{fig: diagonal embedding line into plane}
\end{figure}
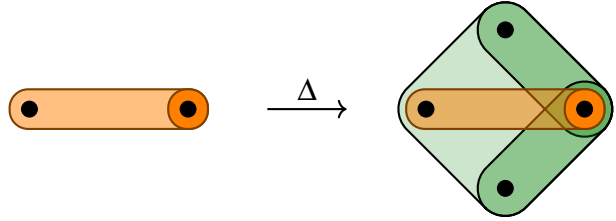

This example shows that the underlying topological space of a monoid scheme does not lead to a satisfactory notion of closed morphisms, which is a prominent concept in usual scheme theory. Several subclasses of closed morphisms allow for equivalent characterizations generalize to monoid schemes. For example:
\begin{itemize}
 \item \emph{Closed immersions} are defined as affine morphisms $\varphi:Y\to X$ for which $\cO_X\to\varphi_\ast\cO_Y$ is surjective (\cite[Definition 2.5]{CHWW15}).
 \item \emph{Separated morphisms} are defined as morphisms $\varphi:Y\to X$ whose diagonal $\Delta_\varphi:Y\to Y\times_XY$ is a closed immersion (\cite[Definition 3.3]{CHWW15}).
 \item \emph{Proper morphisms} are only defined in the finite type case in terms of the valuative criterion for properness (\cite[Definition 8.4]{CHWW15}).
 \item There is no notion of \emph{(universally) closed morphisms} yet.
\end{itemize}
In this paper, we introduce a secondary topological space associated with monoid schemes, which we call the \emph{congruence space} and which captures closed conditions for monoid schemes. This leads to a satisfactory notion of closed morphisms for monoid schemes and a topological characterization of the above mentioned classes of morphisms.

\subsection*{Earlier works on congruences}
The idea to use congruences in the context of monoid schemes is not novel to this paper. Berkovich develops in the unpublished text \cite{Berkovich12a} a scheme theory for monoids that is based on congruences, or quotients, rather than on ideals, or submodules. In fact, this text features many concepts that appear in our text.

Berkovich's approach leads, however, to a quite different notion of $\Fun$-scheme since the space of prime congruences does not reflect the same notion of open coverings as monoid schemes (cf.\ section \ref{subsubsection: Berkovich}). In particular, there are topological coverings in the congruence setting, which are not canonical. This leads to technical ramifications in Berkovich's theory in contrast to monoid schemes in the sense of our text. 

Deitmar develops in \cite{Deitmar13} a variant of congruence schemes for sesquiads, which are algebraic objects that generalize monoids. Also Deitmar's theory deviates from monoid schemes in the sense that it features non-canonical open coverings.

\subsection*{Follow-up work}

Jarra builds up on the theory of congruence spaces as developed in this text: he introduces in \cite{Jarra23a} an intrinsic subspace of the congruence space that exhibits the expected dimension of monoid schemes over algebraically closed monoids. In a second text \cite{Jarra23b}, Jarra shows that this subspace coincide with the spaces in Smirnov's approach to the ABC-conjecture via Hurwitz formulas in $\Fun$-geometry.

\subsection*{Congruence spaces}
We refer the reader to section \ref{section: background} for background on pointed monoids and monoid schemes. Here we restrict ourselves to introduce the novel concepts of this paper.

A \emph{congruence on $A$} is an equivalence relation $\fc$ on $A$ such that $(a,b)\in\fc$ implies $(ac,bc)\in\fc$. The quotient $A/\fc$ is naturally a monoid, which establishes a bijective correspondence between the congruences of $A$ and the monoid quotients of $A$. A congruence $\fp$ is \emph{prime} if $A/\fp$ is \emph{integral}, or multiplicatively cancellative. 

The \emph{congruence space of $A$} is the set $\Cong(A)$ of all prime congruences $\fp$ on $A$ together with the topology generated by open subsets of the form 
\[
 U_{a,b} \ = \ \{ \fp\in \Cong (A) \mid (a,b)\notin\fp \}
\]
for $a,b\in A$. It comes with a continuous surjection $\pi_A:\Cong(A)\to\MSpec(A)$, which sends a prime congruence $\fp$ to its nullideal $I_\fp=\{a\in A\mid (a,0)\in\fp\}$.

The following result endows every monoid scheme with a congruence space; cf. Propositions \ref{prop: functoriality of congruence spaces} and \ref{prop: projection onto the monoid scheme}. Let $\MSch$ be the category of monoid schemes and $\Top$ the category of topological spaces.

\begin{thmA}
 There is a functor $\MSch\to\Top$ that associates with every monoid scheme $X$ a topological space $\til X$, called its \emph{congruence space}, together with a continuous surjection $\pi_X:\til X\to X$, which is characterized by the following properties:
 \begin{enumerate}
  \item If $X=\MSpec(A)$ is affine, then $\til X=\Cong(A)$ and $\pi_X=\pi_A$.
  \item For every morphism $\varphi:Y\to X$ of monoid schemes, the diagram
        \[
         \begin{tikzcd}[column sep=80pt, row sep=20pt]
          \til Y \ar[r,"\tilde\varphi"] \ar[d,"\pi_Y"'] & \til X \ar[d,"\pi_X"] \\
          Y \ar[r,"\varphi"] & X
         \end{tikzcd}
        \] 
        commutes.
  \item Open immersions $\varphi:Y\to X$ induce open topological embeddings $\tilde\varphi:\til Y\to \til X$.
  \item If $X=\bigcup_{i\in I} U_i$ is covered by open subschemes $U_i$, then $\til X=\bigcup_{i\in I} \til U_i$.
 \end{enumerate}
\end{thmA}

The notion of a structure sheaf for congruence spaces is problematic  (cf.\ section \ref{subsection: structure sheaf for the congruence space}); still every point comes with a well-defined notion of a \emph{stalk} and a \emph{residue field}.

\subsection*{Closed properties}

The congruence space turns out to be a suitable space to capture closed properties of (morphisms between) monoid schemes.

We study several morphism classes in this paper and explain our results in the following. The systematic approach is to replace closed morphisms in usual scheme theory by morphisms $\varphi:Y\to X$ of monoid schemes for which $\tilde\varphi:\til Y\to \til X$ is closed. Note that this interpolation between the monoid scheme and its congruence space makes proofs significantly more subtle and difficult than in usual scheme theory. 

\subsection*{Closed points}

As a first insight (Proposition \ref{prop: closed points of the congruence space}), we find canonical bijections
\[
 X(\Fun) \ \stackrel{\chi_X}\longrightarrow \ \big\{\text{closed points of $\til X$}\big\} \ \stackrel{\tilde\pi_X}\longrightarrow \ X
\]
where $X(\Fun)=\Hom(\MSpec\Fun,\, X)$ and where $\tilde\pi_X$ is the restriction of the surjection $\pi_X:\til X\to X$ to the subset of closed points of $\til X$.

\subsection*{Closed subschemes}

A \emph{closed subscheme} of a monoid scheme $X$ is the isomorphism class of a closed immersion $\varphi:Y\to X$. Closed subschemes can be characterized similarly as for usual schemes if concepts are generalized appropriately.

A monoid $A$ is \emph{strongly reduced} if the trivial congruence $\fc_\triv=\{(a,a)\in A\times A\mid a\in A\}$ is a prime congruence. A monoid scheme $X$ is strongly reduced if $\cO_X(U)$ is strongly reduced for every open subset $U$. Every monoid scheme has a strong reduction.

A \emph{vanishing set} of $\til X$ is the image $\tilde\varphi(\til Y)$ for a closed immersion $\varphi:Y\to X$. For affine monoid schemes $X=\MSpec(A)$, a vanishing set is the intersection of closed subsets of the form $V_{a,b}=\{\fp\in \til X\mid (a,b)\in\fp\}$; note that subsets of this form are not closed under finite unions, i.e.\ vanishing sets are closed, but not every closed subset of $\til X$ is a vanishing set.

A \emph{quasi-coherent congruence sheaf} on $X$ is a quasi-coherent subsheaf of $\cO_X\times\cO_X$ that is a congruence on $\cO_X(U)$ for every open subset $U$ of $X$. The congruence kernel of a closed immersion $\varphi:Y\to X$ is the quasi-coherent congruence sheaf defined by $\Congker(\varphi)(U)=\{(a,b)\in \cO_X(U)\times\cO_X(U)\mid \varphi^\#(U)(a)=\varphi^\#(U)(b)\}$.

These relations yield a diagram
\[
 \begin{tikzcd}[column sep=60]
  \big\{\text{vanishing sets of $X$}\big\} \ar[<->,r,"\text{bijective}"] \ar[>->,d,shift right=2] & \Big\{\begin{array}{c}\text{strongly reduced}\\ \text{closed subschemes of $X$}\end{array}\Big\} \ar[>->,d,shift right=2] \\
  \Big\{\begin{array}{c}\text{quasi-coherent} \\ \text{congruence sheaves on $X$}\end{array}\Big\} \ar[->>,u,shift right=2]\ar[<->,r,"\text{bijective}"] & \big\{\text{closed subschemes of $X$}\big\} \ar[->>,u,shift right=2]
 \end{tikzcd}
\]
which commutes for all paths that either start or end in the top row; cf.\ Remark \ref{vanish-c.i-qcoh}.

\subsection*{Closed immersions}

In \cite{CHWW15}, a \emph{closed immersion} of monoid schemes is defined as an affine morphism $\varphi:Y\to X$ for which $\cO_X\to\varphi_\ast\cO_Y$ is surjective. \emph{A posteriori} we gain the following topological characterization, which reflects the usual definition of closed immersions of usual schemes; cf. Theorem \ref{thm: topological characterization of closed immersion}.

\begin{thmA}
 Let $\varphi:Y\to X$ be a morphism of monoid schemes. The following are equivalent:
 \begin{enumerate}
  \item $\varphi:Y\to X$ is a closed immersion.
  \item $\varphi:Y\to X$ is a quasi-compact topological embedding, $\varphi^\#:\cO_X\to\varphi_\ast\cO_Y$ is surjective and the image of $\tilde\varphi:\til Y\to\til X$ is a vanishing set.
 \end{enumerate}
\end{thmA}

\subsection*{Separated morphisms}
In \cite{CHWW15}, a morphism $\varphi:Y\to X$ is defined to be \emph{separated} if the diagonal $\Delta_\varphi:Y\to Y\times_XY$ is a closed immersion (see \cite{CHWW15}). \emph{A posteriori} the congruence space provides a topological characterization. We say that $\varphi:Y\to X$ is \emph{quasi-separated} if $\Delta_\varphi:Y\to Y\times_XY$ is quasi-compact. The following is Theorem \ref{thm: topological characterization of separated morphisms}.

\begin{thmA}
 Let $\varphi:Y\to X$ be a morphism of monoid schemes. Then the following are equivalent:
 \begin{enumerate}
  \item $\varphi$ is separated;
  \item $\varphi$ is quasi-separated and the image of $\tilde\Delta_\varphi$ is a closed subset of the congruence space of $Y\times_XY$.
 \end{enumerate}
\end{thmA}

\subsection*{Closed and proper morphisms}

We say that a morphism $\varphi:Y\to X$ of monoid schemes is 
\begin{itemize}
 \item \emph{closed} if the map $\tilde\varphi:\til Y\to\til X$ between congruence spaces is closed;
 \item \emph{universally closed} if for every morphism $\psi:X'\to X$ the canonical projection $X'\times_XY\to X'$ is closed;
 \item \emph{proper} if it is of finite type, separated and universally closed.
\end{itemize}

\subsection*{Valuative criteria}
A \emph{pointed group} is a pointed monoid $G$ for which every nonzero element is invertible and $0\neq1$. A \emph{valuation monoid} is a submonoid $A$ of a pointed group $G$ such that for every nonzero $a\in G$ either $a\in A$ or $a^{-1}\in A$.

Let $\varphi:Y\to X$ be a morphism of monoid schemes. A \emph{test diagram} for $\varphi$ is a commutative diagram of the form
\[
 \begin{tikzcd}[column sep=80pt, row sep=20pt]
  U=\MSpec(G) \ar[d,"\iota"'] \ar[r,"\eta"] & Y \ar[d,"\varphi"] \\
  T=\MSpec(A) \ar[r,"\nu"]      & X
 \end{tikzcd}
\]
where $A$ is a valuation monoid in the pointed group $G$ and $\iota$ is induced by the inclusion map $A\to G$. A \emph{lift for the test diagram} is a morphism $\hat\nu:T\to Y$ such that 
\[
 \begin{tikzcd}[column sep=80pt, row sep=20pt]
  U=\MSpec (G) \ar[d,"\iota"'] \ar[r,"\eta"] & Y \ar[d,"\varphi"] \\
  T=\MSpec(A) \ar[r,"\nu"] \ar[ur,dashed,"\hat\nu"]     & X
 \end{tikzcd}
\]
commutes. The following summarizes Theorem \ref{uni-closed}, Theorem \ref{thm: valuative criterion of separatedness} and Theorem \ref{prop-intrinsic}.

\begin{thmA}
 A quasi-compact / quasi-separated / finite type morphism $\varphi:Y\to X$ of monoid schemes is universally closed / separated / proper if and only if every test diagram for $\varphi$ has at least / at most / exactly one lift.
\end{thmA}

\subsection*{Content overview}
Section \ref{section: background} contains a short reminder on monoid schemes and introduces the conventions for this paper. In section \ref{section: congruence spaces}, we introduce congruence spaces and study their properties. In section \ref{section: Strongly reduced monoid schemes}, we introduce the strong reduction of monoids and monoid schemes and compare it with the usual reduction. In section \ref{section: Closed immersions and vanishing sets}, we exhibit a topological characterization of closed immersions and study the relation between closed subschemes, vanishing sets and congruence sheaves. In section \ref{section: Valuation monoids}, we introduce valuation monoids and characterize them as maximal objects with respect to dominant morphisms. In section \ref{section: Universally closed morphisms}, we introduce universally closed morphisms and establish the corresponding valuative criterion. In section \ref{section: Separated morphisms}, we exhibit a topological characterization of separated morphisms, which is based on a study of locally closed immersions, and we establish the corresponding valuative criterion. In section \ref{section: Proper morphisms}, we introduce proper morphisms and establish the corresponding valuative criterion.

\subsection*{Acknowledgements}
The first author was supported by the Marie Curie Grant No.\ 101022339 of the European Commission. The second author’s work was partially supported
by the SERB NPDF PDF/2020/000670 grant of Govt.\ of India, partially by the IBS-R003-D1 grant of the IBS-CGP, POSTECH,
South Korea and partially by the DST/INSPIRE/04/2021/002904 grant of Govt.\ of India.

\section{Preliminaries}
\label{section: background}

In this section we recall some basic facts about commutative pointed monoids and monoid schemes in brevity. For further details and a more comprehensive introduction into these concepts, we refer the reader to \cite{Chu-Lorscheid-Santhanam12} and \cite{CHWW15}.

\subsection{Pointed monoids}
\label{subsection: Pointed monoids}

A \emph{(commutative) pointed monoid} is a commutative semigroup $A$ with a zero (absorbing element) $0$ and an identity $1$, satisfying $a.0=0.a=0$ and $a.1=1.a=a$ for all $a\in A$. A \emph{morphism of pointed monoids} is a semigroup homomorphism that preserves the zero and identity. 

An \emph{ideal} of a pointed monoid $A$ is a subset $I$ of $A$ such that $0\in I$ and $IA\subseteq A$. The \emph{quotient monoid} is the set $A/I=(A\setminus I)\cup \{0\}$ together with the multiplication given by $a\cdot b = ab $ if $a,b\in A\setminus I$ and $a \cdot b =0$ otherwise. An ideal $P$ of $A$ is \emph{prime} if $ab\in P$ implies $a\in P$ or $b\in P$ or, in other words, $A/P$ is without zero divisors. The complement of the set of units $A^\times$ of a pointed monoid $A$ forms an ideal and is therefore the unique maximal ideal of $A$ which we denote by $M_A$.

Let $S$ be a \emph{multiplicative subset} of a pointed monoid $A$, i.e.\ $1\in S$ and $ab\in S$ for all $a,b\in S$. The \emph{localization of $A$ at $S$} is the monoid $S^{-1}A=S\times A/\sim$ where $(s,a)\sim(s',a')$ if and only if $tsa'=ts'a$ for some $t\in S$. We denote the equivalence class of $(s,a)$ by $\tfrac as$. The association $a\mapsto\tfrac a1$ defines the \emph{localization map} $\iota_S:A\to S^{-1}A$. Particular examples of localizations are $A[h^{-1}]=S_h^{-1}A$ with $S_h=\{h^i\}_{i\geq0}$ for $h\in A$ and $A_P=S_P^{-1}A$ with $S_P=A\setminus P$ for a prime ideal $P$ of $A$.

The category of pointed monoids is complete and cocomplete. In particular, it comes with a tensor product $A\otimes_CB$ of any pair of morphisms $C\to A$ and $C\to B$. The pointed monoid $\Fun=\{0,1\}$ is its initial object.

Let $A$ be a pointed monoid. The \emph{free monoid in $t_1,\dotsc,t_n$ over $A$} is the pointed monoid
\[
 A[t_1,\dotsc,t_n] \ = \ \big\{ at_1^{e_1}\cdots t_n^{e_n} \, \big| \, e_1,\dotsc,e_n\in\N \big\}
\]
with product 
\[
 (at_1^{e_1}\cdots t_n^{e_n}) \ \cdot \ (bt_1^{f_1}\cdots t_n^{f_n}) \ = \ (ab)t_1^{e_1+f_1}\cdots t_n^{e_n+f_n}
\]
and identification $0\cdot t_1^{e_1}\cdots t_n^{e_n}=0\cdot t_1^{f_1}\cdots t_n^{f_n}$ for all $e_i,f_i\in\N$.

\subsection{Integral and zero divisor free monoids}
\label{subsection: Integral and zero divisor free monoids}

An \emph{integral monoid} is a pointed monoid $A$ such that $ab=ac$ implies $b=c$ for all $a,b, c\in A$ with $a\neq 0$. A \emph{zero divisor free monoid} is a pointed monoid $A$ for which $S=A\setminus \{0\}$ is multiplicative subset. 

Let $A$ be zero divisor free, $S=A\setminus \{0\}$ and $i_S:A\to S^{-1}A$ be the localization map. The \emph{group completion of $A$} is $\Frac A=S^{-1}A$. The \emph{integral quotient of $A$} is $A^{\overint}=i_S(A)$.

Note that $\Frac A$ is a pointed group and $A^\overint$ is an integral monoid. The localization map defines canonical maps $A\to\Frac A$ and $A\to A^\overint$. We mention the following universal properties of $\Frac A$ and $A^\overint$ without proof. 

\begin{lemma}
 Let $A$ be a pointed monoid without zero divisors. Then every morphism $f:A\to B$ into a pointed group $B$ factors uniquely through $\Frac A$ and every morphism $f:A\to B$ into an integral monoid $B$ factors uniquely through $A^{\overint}$.
\end{lemma}

\subsection{Monoid schemes}
\label{subsection: Monoid schemes}

Let $A$ be a pointed monoid. The \emph{prime spectrum of $A$} is the set
\[
 \MSpec(A) \ = \ \{P\subseteq A\mid P \text{~is a prime ideal of } A\}
\]
that carries the topology generated by the open subsets $U_h=\{P\mid h\notin P\}$ for $h\in A$, together with a sheaf $\cO_{X}$ of monoids on $X=\MSpec(A)$ that is characterized by its values $\cO_X(U_h)=A[h^{-1}]$ and its stalks $\cO_{X,P}=A_P$.

A \emph{monoid scheme} is a topological space $X$ together with a sheaf $\cO_X$ of monoids that is locally isomorphic to spectra of pointed monoids. A monoid scheme is \emph{affine} if it isomorphic the spectrum of a pointed monoid. 

A \emph{morphism of monoid schemes $X$ and $Y$} is a continuous map $\varphi:X\to Y$ together with a morphism of sheaves $\varphi^\#:\cO_X\to\varphi_\ast\cO_Y$ of pointed monoids for which the morphism of stalks $\varphi_x:\cO_{X,x}\to\cO_{Y,y}$ are \emph{local} for every $x\in X$ and $y=\varphi(x)$ in the sense that $\varphi_x^{-1}(\cO_{Y,y}^\times)=\cO_{X,x}^\times$. We denote the category of monoid schemes by $\MSch$.

Note that $\MSpec$ extends to an anti-equivalence of $\Mon$ with the category of affine monoid schemes: a morphism $f:A\to B$ of pointed monoids defines the morphism $f^\ast:\MSpec(B)\to\MSpec(A)$ with $f^\ast(P)=f^{-1}(P)$. It has the global section functor $\Gamma:\MSch\to\Mon$, which sends a monoid scheme $X$ to $\Gamma X=\cO_X(X)$, as an adjoint:
\[
 \Hom(X,\MSpec(A)) \ = \ \Hom(A,\Gamma X).
\]

\subsection{Base extensions}
\label{subsection: Base extensions}

The tensor product of pointed monoids extends to a fibre product $X\times_ZY$ of monoid schemes, which defines the base extension $X_B=X\otimes_AB$ of a monoid scheme $X$ over $A$ to a $A$-algebra $B$.

Given a pointed monoid $A$, we denote its \emph{associated ring} by $A^+=\Z[A]/\gen{0}$, which is equal to $\Z[A]\setminus\{0\}$ as an abelian group. Given a monoid scheme $X$ over $A$ and a morphism $f:A\to k$ from a pointed monoid $A$ to (the underlying monoid of) a ring $k$, we obtain a functorial base extension $X_k^+=X\otimes_Ak$ that is characterized by the properties that it commutes with colimits of open immersions and that for affine $X=\MSpec(B)$,
\[
 X_k \ = \ \Spec \big(B^+\otimes_{A^+} k\big).
\]

\subsection{Finite type morphisms}

A morphism $f:A\to B$ of pointed monoids is \emph{of finite type} if there are elements $b_1,\hdots, b_n\in B$ such that every element of $B$ can be expressed as a monomial in the $b_i$ and elements of $f(A)$. 

A morphism $\varphi:X\to Y$ of monoid schemes is \emph{of finite type} if there are affine open coverings $Y=\bigcup_{i\in I} V_i$ and $\varphi^{-1}(V_i)=\bigcup_{j\in J_i} U_{i,j}$ for all $i\in I$ such that $J_i$ is finite for all $i\in I$ and such that $\Gamma{\varphi_{i,j}}:\Gamma{V_i}\to \Gamma{U_{i,j}}$ is of finite type for all $i\in I$ and $j\in J_i$.

\section{Congruence spaces}
\label{section: congruence spaces}

In this section, we introduce congruence spaces for monoid schemes, which are secondary and subordinate topological spaces that translate topological property from usual scheme theory appropriately to the world of monoid schemes.

\subsection{Congruences}
The algebraic backbone for congruence spaces are congruences. We review their definition and first properties. For more details, see \cite[Chapter 3]{Lorscheid18}.

\begin{defn}
 A \emph{congruence} of a pointed monoid $A$ is an equivalence relation $\mathfrak{c}$ on $A$ that is closed under multiplication by $A$, i.e.\ $(a, b)\in \mathfrak{c}$ implies $(ac, bc)\in \mathfrak{c}$ for all $a,b,c \in A$. The \emph{nullideal of $\fc$} is 
 \[
  I_\mathfrak{c}=\{a\in A\mid (a,0)\in\mathfrak{c}\}. 
 \]
 The \emph{quotient of $A$ by a congruence $\fc$} is the set $A/\fc$ of equivalence classes with multiplication given by $[a]\cdot[b]=[ab]$. 
\end{defn}

Note that the product of elements of $A/\fc$ is well-defined on representatives and turns $A/\fc$ into a pointed monoid with zero $[0]$ and identity $[1]$. The quotient map $\pi:A\to A/\fc$ is a pointed monoid morphism. 

Every subset $S$ of $A\times A$ is contained in a smallest congruence 
\[
 \gen S \ = \ \bigcap_{S\subset \fc} \ \fc
\]
that contains $S$, which we call the \emph{congruence generated by $S$}.

\begin{defn}
 Let $f:A\to B$ be a morphism of pointed monoids. The \emph{congruence kernel of $f$} is 
 \[
  \congker(f):=\{(a,a')\in A\times A\mid f(a)=f(a')\}.
 \]
\end{defn}

The congruence kernel of a morphism is indeed a congruence and that every congruence $\fc$ of $A$ is the congruence kernel of the quotient map $A\to A/\fc$. This establishes a bijective correspondence between the quotients of $A$ and the congruences on $A$.

\begin{defn}
 A \emph{prime congruence} of $A$ is a congruence $\mathfrak{c}$ such that $A/\mathfrak{c}$ is integral, i.e., $(ab, ac) \in \mathfrak{c}$ implies $(a,0)\in \mathfrak{c}$ or $(b,c)\in \mathfrak{c}$ for all $a,b,c\in A$. 
\end{defn}

If $\mathfrak{c}$ is a prime congruence, then $I_\mathfrak{c}$ is a prime ideal. The congruence kernel $\congker(f)$ of a morphism $f:A\to B$ into an integral monoid $B$ is a prime congruence. 

\begin{defn}
 Let $f:A\to B$ be a morphism of pointed monoids, $\fc$ a congruence on $A$ and $\fd$ a congruence on $B$. The \emph{push-forward of $\fc$ along $f$} is the congruence 
 \[
  f_\ast(\fc) \ = \ \gen{(f(a),f(b))\in B\times B\mid (a,b)\in\fc\}}
 \]
 of $B$ and the \emph{pullback of $\fd$ along $f$} is the congruence
 \[
  f^\ast(\fd) \ = \ \{(a,b)\in A\times A\mid (f(a),f(b))\in\fd\}
 \]
 on $A$.
\end{defn}

\begin{prop}\label{prop: congruences in quotients}
 Let $A$ be a pointed monoid, $\fc$ a congruence on $A$ and $\pi:A\to A/\fc$ the quotient map. Then the pushforward and pullback define mutually inverse bijections
 \[
  \begin{tikzcd}[column sep=50pt]
   \big\{\text{congruences on $A$ that contain $\fc$}\big\} \ar[r,shift left=2pt,"\pi_\ast"] & \big\{\text{congruences on $A/\fc$}\big\} \ar[l,shift left=2pt,"\pi^\ast"]
  \end{tikzcd}
 \]
 that restrict to bijections between the respective subsets of prime congruences.
\end{prop}

\begin{proof}
 This result follows immediately from the identification of congruences with quotients of $A$ and the fact that the quotients of $A/\fc$ correspond to the quotients of $A$ that factor through $\pi:A\to A/\fc$. The additional claim about prime congruences follows from the characterization of prime congruences $\fp$ in terms of the quotient $A/\fp$, which agrees on both sides of the bijection.
\end{proof}

\subsection{Compatibility with localizations}
\label{subsection: Compatibility with localizations}

\begin{defn}
 Let $A$ be a pointed monoid, $S$ a multiplicative subset and $\mathfrak{c}$ a congruence of $A$. The \emph{localization of $\mathfrak{c}$ at $S$} is 
 \[
  S^{-1}\mathfrak{c} \ = \ \big\{(\tfrac{a}{s},\tfrac{a'}{s'})\in S^{-1}A\times S^{-1}A~\big|~(ts'a,tsa')\in\mathfrak{c}\text{~for some~}t\in S\big\}. 
 \]
\end{defn}

Note that $S^{-1}\mathfrak{c}$ is a congruence of $S^{-1}A$. Localizations of monoids are exact in the following sense.

\begin{lem}\label{lemma: kernels of localization}
 Let $f:A\to B$ be a morphism of pointed monoids. Let $S$ be a multiplicative subset of $A$ and $f_S:S^{-1}A\to f(S)^{-1}B$ the induced morphism. Then $\ker(f_S)=S^{-1}\ker(f)$ and $\congker(f_S)=S^{-1}\congker(f)$. In particular, $f_S$ is injective if $f$ is so.  
\end{lem}

\begin{proof}
 An element $\frac as\in S^{-1}A$ is in the kernel of $f_S$ if and only if $ \frac{f(a)}{f(s)}=\frac{0}{1}$, which means that $f(t)f(1)f(a)=f(t)f(s)f(0)=0$ for some $t\in S$. Equivalently $ta\in \ker(f)$ or, in other words, $\frac{a}{s}= \frac{at}{st}\in S^{-1}(\ker(f))$. This shows $\ker(f_S)=S^{-1}(\ker(f))$.

 A pair $(\frac{a}{s},\frac{a'}{s'})$ is in the congruence kernel of $f_S$ if and only if $f_S(\frac{a}{s})= f_S(\frac{a'}{s'})$, which means that $f(tsa')=f(t)f(s)f(a')=f(t)f(s')f(a)=f(ts'a)$ for some $t\in S$. Equivalently $ (tsa',ts'a)\in \congker(f)$ or, in other words, $ \frac{a}{s}= \frac{at}{st}\in S^{-1}(\ker(f))$. This shows $\congker(f_S)=S^{-1}(\congker(f))$.
	
If $f$ is injective, then $\congker(f)$ is the trivial congruence $\fc_\triv=\{(a,a)\in \mid a\in A\}$. Thus $\congker(f_S)= S^{-1}\fc_\triv$ is trivial and consequently $f_S$ is injective.
\end{proof}

\begin{prop}\label{prop: prime congruences in localizations}
 Let $\iota_S:A\to S^{-1}A$ be a localization of the pointed monoid $A$ at $S$. Then the maps
 \[
  \begin{tikzcd}[column sep=60pt]
   \big\{\begin{array}{c}\text{prime congruences on }S^{-1}A\end{array}\big\} \ar[r,shift left=1,"\iota_S^\ast"] &
   \Big\{\begin{array}{c}\text{prime congruences $\fc$}\\ \text{on $A$ with $I_\fc\cap S=\emptyset$}\end{array}\Big\} \ar[l,shift left=1,"\iota_{S,\ast}"] 
  \end{tikzcd}
 \]
 are mutually inverse bijections.
\end{prop}

\begin{proof}
 We first check that the maps are well defined. Let $\fd$ be a prime congruence on $S^{-1}A$. Then $(ac,bc)\in \iota_S^\ast(\fd)=\{(a,a')\in A\times A~|~(\frac{a}{1},\frac{a'}{1})\in \fd \}$ implies $(\frac{ac}{1},\frac{bc}{1})\in \fd$ which implies $(\frac{c}{1},\frac{0}{1})\in \fd$ or $(\frac{ac}{1},\frac{bc}{1})\in \fd$. Thus we have $(c,0)\in \iota_S^\ast(\fd)$ or $(a,b)\in \iota_S^\ast(\fd)$ which shows that $\iota_S^\ast(\fd)$ is a prime congruence on $A$. Also, if $(\frac{a}{1},\frac{0}{1})\in \fd$ then clearly $\frac{a}{1}\notin (S^{-1}A)^\times$ which implies $a\notin S$. Thus $I_{\iota_S^\ast(\fd)}\cap S=\emptyset$ and this shows that $\iota_S^\ast$ is well defined.
 
 Let $\fc$ be a prime congruence on $A$ such that $I_\fc\cap S=\emptyset$. Consider $(\frac{a}{s}\frac{c}{r},\frac{b}{t}\frac{c}{r})\in \iota_{S,\ast}(\fc)=\langle\{(\frac{a}{1},\frac{b}{1})\in S^{-1}A\times S^{-1}A~|~(a,b)\in \fc\}\rangle$. Since $\frac{a}{s}\frac{c}{r}=\frac{act}{srt}$ and $\frac{b}{t}\frac{c}{r}=\frac{bcs}{trs}$ in $S^{-1}A$, we have $(\frac{act}{srt},\frac{bcs}{trs})\in \iota_{S,\ast}(\fc)$. Therefore $(\frac{act}{1}, \frac{bcs}{1})\in \iota_{S,\ast}(\fc)$ which implies $(actr', bcsr')\in \fc$ for some $r'\in S$. Thus either $(cr',0)\in \fc$ or $(at,bs)\in \fc$. This implies either $(\frac{cr'}{1},\frac{0}{1})\in  \iota_{S,\ast}(\fc)$ or $(\frac{at}{1},\frac{bs}{1})\in  \iota_{S,\ast}(\fc)$. Since $(\frac{cr'}{1},\frac{0}{1})\in  \iota_{S,\ast}(\fc)$ implies $(\frac{cr'}{r'},\frac{0r'}{r'}) = (\frac{c}{1},\frac{0}{1}) \in \iota_{S,\ast}(\fc)$ and $(\frac{at}{1},\frac{bs}{1})\in  \iota_{S,\ast}(\fc)$ implies $(\frac{at}{ts},\frac{bs}{ts})=(\frac{a}{s},\frac{b}{t})\in  \iota_{S,\ast}(\fc)$ it follows that $ \iota_{S,\ast}(\fc)$ is a prime congruence on $S^{-1}A$. Therefore $\iota_{S,\ast}$ is well defined.
 
 Clearly, we have $\fc\subseteq \iota_S^*(\iota_{S,\ast}(\fc))$ and $\fd\subseteq \iota_{S,\ast}(\iota_S^\ast(\fd))$. Consider any $(a,a')\in  \iota_S^*(\iota_{S,\ast}(\fc))$. Then $(\frac{a}{1},\frac{a'}{1})\in \iota_{S,\ast}(\fc)$ which implies $(sa,sa')\in \fc$ for some $s\in S$. Since $I_\fc\cap S=\emptyset$, we have $(s,0)\notin \fc$ and therefore we must have $(a,a')\in \fc$ as desired. Consider any $(\frac{a}{s},\frac{b}{t}) \in \iota_{S,\ast}(\iota_S^\ast(\fd))$. Then $(\frac{ast}{s},\frac{bst}{t})=(\frac{at}{1},\frac{bs}{1})\in \iota_{S,\ast}(\iota_S^\ast(\fd))$. This implies $(atr,bsr)\in \iota_S^\ast(\fd)$ for some $r\in S$. But $(r,0)\notin \iota_S^\ast(\fd)$ since $I_{\iota_S^\ast(\fd)}\cap S =\emptyset$. Therefore $(at,bs)\in \iota_S^\ast(\fd)$ which implies $(\frac{at}{1},\frac{bs}{1})\in \fd$ or equivalently, $(\frac{a}{s},\frac{b}{t})\in \fd$. This shows that $\iota_S^\ast$ and $\iota_{S,\ast}$ are mutually inverse bijections.
\end{proof}

\subsection{The congruence space of a pointed monoid}

\begin{defn}
 Let $A$ be a pointed monoid. The \emph{congruence space of $A$} is the topological space $\Cong(A)$ of all prime congruences on $A$ whose topology is generated by the open subsets
 \[
  U_{a,b} \ = \ \{\fp\in\Cong(A)\mid (a,b)\notin\fp\}
 \]
 where $a$ and $b$ vary through all pairs of elements of $A$. We denote the complement $U_{a,b}$ by $V_{a,b}=\{\fp\in\Cong(A)\mid(a,b)\in\fp\}$.
\end{defn}

\begin{ex}\label{notbasis}
 The collection $\{U_{a,b}\mid a,b \in A\}$ of open subsets is in general not a basis for the topology on $ \Cong(A)$, as in the example of the pointed monoid $A=\mathbb{F}_1[t]$. We have 
 \[
  \Cong(A) \ = \ \big\{\mathfrak{c}_{\triv},\ \langle(t,0)\rangle,\ \langle(t,t^k)\rangle\mid k\geq 0\big\}. 
 \]
 The topology of $\Cong(A)$ is generated by the open subsets of the forms $U_{t^n,0}$ and $U_{t^n,t^m}$ for $n>m\geq 0$. The intersection $U_{t,0}\cap U_{t,1}$ contains the prime congruence $\gen{t,t^2}$, but it does not contain any non-empty open subset of the form $U_{a,b}$. This shows that the family $\{U_{a,b}\mid a,b\in A\}$ does not form a basis for the topology of $\Cong(A)$, but only a subbasis.
\end{ex}

The following result upgrades the congruence space to a functor $\Cong:\Mon\to\Top$.

\begin{lem}\label{cong1} \ 
 Let $f:A\to B$ be a morphism of pointed monoids and $\fq$ a prime congruence of $B$. Then $f^\ast(\fq)$ is a prime congruence of $A$, and the map $f^\ast:\Cong(B)\to\Cong(A)$ is continuous.
\end{lem}

\begin{proof} 
 We begin with the verification that $\fp=f^\ast(\fq)$ is prime. Consider $(ca,cb)\in\fp$, i.e.\ $f(c)f(a)=f(ca)\sim f(cb)=f(c)f(b)$ in $\fq$. Since $\fq$ is prime, this means that either $f(c)\sim0$ or $f(a)\sim f(b)$ in $\fq$. Thus $c\sim 0$ or $a\sim b$ in $\fp$, which shows that $\fp$ is prime.
 
 It suffices to verify the continuity of $f^\ast$ on a subbasis. For $U_{a_1,a_2}\subseteq \Cong(A)$, we have
 \begin{align*}
  (f^\ast)^{-1}(U_{a_1,a_2}) \ &= \ \{\mathfrak{d}\in \Cong(B)\mid (a_1,a_2)\notin f^\ast(\mathfrak{d})\} \\
  &= \ \{\mathfrak{d}\in \Cong(B)\mid (f(a_1),f(a_2))\notin\mathfrak{d})\}= \ U_{f(a_1),f(a_2)}
 \end{align*}
 which shows $f^\ast$ is continuous.
\end{proof}

\subsection{The congruence space of a monoid scheme}

The results from section \ref{subsection: Compatibility with localizations} allow us to extend the definition of the congruence space from pointed monoids, or affine monoid schemes, to arbitrary monoid schemes.

\begin{defn}
 Let $X$ be a monoid scheme and $\mathcal{U}_X$ the category of all affine open subschemes of $X$ together with all open immersions. Let $\cU^\congr_X$ be the category of congruence spaces $U^\congr=\Cong(\Gamma U)$ where $U\in \mathcal{U}_X$ together with the induced maps. The congruence space associated with $X$ is the topological space 
 \begin{equation*}
		X^\congr \ = \ \colim \mathcal{U}^\congr_X.
 \end{equation*}
 It comes with canonical inclusions $\iota_U: \Cong(\Gamma U)\to X^\congr$ for every affine open $U$ of $X$.
\end{defn}

Note that since every affine open covering of an affine monoid scheme $X=\MSpec(A)$ must contain $X$ itself as its largest open subset, the canonical inclusion $\iota_{\MSpec(A)}:\Cong(A)\to (\MSpec(A))^\congr$ is a homeomorphism. The term ``canonical inclusions'' is justified by the following result.

Recall that a morphism $\iota:Y\to X$ of monoid schemes is an \emph{open immersion} if it is an open topological embedding and if $\mathcal{O}_Y$ is the restriction of $\mathcal{O}_X$ to $Y$. 

\begin{prop}\label{prop: functoriality of congruence spaces}
 The assignment $X\mapsto X^\congr$ extends uniquely to a functor $(-)^\congr:\MSch\to\Top$ such that for every morphism $f:A\to B$ of pointed monoids and $\varphi=f^\ast:\MSpec(B)\to \MSpec(A)$, the diagram
 \[
  \begin{tikzcd}
   \Cong(B) \ar[r,"\varphi"] \ar[d,"\iota_{\MSpec(B)}"',"\cong"] & \Cong(A) \ar[d,"\iota_{\MSpec(A)}","\cong"'] \\
   (\MSpec(B))^\congr \ar[r,"\varphi^\congr"]                & (\MSpec(A))^\congr
  \end{tikzcd}
 \]
 commutes. It satisfies that
 \begin{enumerate}
  \item\label{open1} if $\iota:Y\to X$ is an open immersion, then $\iota^\congr:Y^\congr\to X^\congr$ is an open topological embedding;
  \item\label{open2} if $\{U_i\}$ is an open covering of $X$, then $\{U_i^\congr\}$ is an open covering of $X^\congr$.
 \end{enumerate}
\end{prop}

\begin{proof}
  The functor $(-)^\congr:\MSch\to\Top$ is defined on objects, and the commutativity requirement determines $\varphi^\congr$ for morphisms $\varphi:X\to Y$ of affine monoid schemes. If we can extend $(-)^\congr$ to a functor on all monoid schemes, then uniqueness follows from the functoriality of $(-)^\congr$ and the local nature of morphisms; namely, given $\varphi:Y\to X$, there are affine open coverings $\{U_i\}$ of $X$ and $\{V_i\}$ of $Y$  with $\varphi(V_i)\subseteq U_i$, and $\varphi$ is uniquely determined by its restrictions $\varphi_i:V_i\to U_i$ to affine monoid schemes.
 
 We are left with the construction of $\varphi^\congr:Y^\congr\to X^\congr$ for an arbitrary morphism $\varphi:Y\to X$. As a first step, we establish property \eqref{open1} for an open immersion $\iota:Y\to X$ between affine monoid schemes $Y=\MSpec(B)$ and $X=\MSpec(A)$. Since $i:Y\hookrightarrow X$ is an open immersion, it follows from \cite[Lemma 1.3]{CHWW15} and \cite[Lemma 2.4]{CHWW15} that $B\Cong(A)[s^{-1}]$ for some $s\in A$. Next we verify that $i^\congr$ is injective. Let $\mathfrak{c},\mathfrak{d}\in  {Y}^\congr$ be such that $i^\congr(\mathfrak{c})= i^\congr(\mathfrak{d})$, i.e., $\{(a,a')\in A\times A\mid (\frac{a}{1}, \frac{a'}{1})\in \mathfrak{c}\} =\{(a,a')\in A\times A\mid (\frac{a}{1}, \frac{a'}{1})\in \mathfrak{d}\}$. Then we have
	\begin{equation*}
		(\frac{a_1}{s^n}, \frac{a_2}{s^m})\in \mathfrak{c}\iff (\frac{a_1 s^m}{1}, \frac{a_2 s^n}{1}) \in \mathfrak{c} \iff (\frac{a_1s^{m}}{1}, \frac{a_2s^{n}}{1}) \in \mathfrak{d} \iff 	(\frac{a_1}{s^n}, \frac{a_2}{s^m})\in \mathfrak{d}.
	\end{equation*}
That $i^\congr$ is an open embedding follows if we can show that any open subset of $Y^\congr$ of the form $U_{\frac{a_1}{s^n}, \frac{a_2}{s^m}}$ maps to an open subset of $X^\congr$ under $i^\congr$. Note that we have $i^\congr(U_{\frac{a_1}{s^n}, \frac{a_2}{s^m}})\subseteq U_{a_1s^m,a_2s^n}$. Consider any $\mathfrak{d}\in U_{a_1s^m,a_2s^ n}$. We define
\begin{equation*}
	\mathfrak{c}:=\{(\frac{a}{s^k},\frac{a'}{s^r})\in B\times B\mid (as^{r+n_0}, a's^{k+n_0})\in \mathfrak{d}\text{~for some~} n_0>0\}.
\end{equation*}
We leave it to the reader to check that $\mathfrak{c}\in ({\MSpec(B)})^\congr$ and that $i^\congr(\mathfrak{c}) = \mathfrak{d}$. Thus we have $i^\congr(U_{\frac{a_1}{s^n}, \frac{a_2}{s^m}})= U_{a_1s^m,a_2s^n}$.

As the next step of the proof, we establish property \eqref{open2}. We note that it is enough to prove property \eqref{open2}  for an affine open covering of $X$. Consider the family $\mathcal{U}_X=\{U_i\}_{i\in I}$ of all affine open subschemes of $X$. For all $i,j\in I$, consider principal affine opens $U_{ij}\subseteq U_i$ and isomorphisms $\phi_{ji}: U_{ij}\to U_{ji}$. As already shown, this implies $U_{ij}^\congr$ are open subsets of ${U_i}^\congr$ for all $i,j\in I$. Moreover, since $\phi_{ki}=\phi_{kj}\circ\phi_{ji}$ on $U_{ij}\cap U_{jk}$ and $\Cong: \Mon_0 \to \Top$ is a functor by Lemma \ref{cong1}, we have  $\phi_{ki}^\congr=\phi_{kj}^\congr\circ\phi_{ji}^\congr$ on $U_{ij}^\congr\cap U_{jk}^\congr$. Thus the collection $(I, {U_i^\congr}\}_{i\in I}, \{{U_{ij}^\congr}\}_{i,j\in I}, \{\phi_{ij}^\congr\}_{i,j\in I} )$ gives a gluing data of topological spaces (in other words, ${{U}_X^\congr}$ is a monodromy-free diagram in the category of topological spaces in the sense of \cite[\S 1]{Lorscheid17}). The topological space obtained from this gluing data is ${X}^\congr$, by definition. Therefore for any $U_i\in \mathcal{U}_X$, we have ${U_i}^\congr$ is an open subset of ${X}^\congr$. Let $\{X_\alpha\}_{\alpha\in A}$ be an affine open cover of $X$ and consider any $U\in \mathcal{U}_X$. Since every affine open has a unique closed point and since  $U$ is covered by $\{U\cap X_\alpha\}_{\alpha\in A}$, we must have $U\subseteq X_\alpha$ for some $\alpha\in A$. Therefore we have ${U}^\congr\subseteq X_\alpha^\congr$ for some $\alpha\in A$. Since we have shown that $\{{U}^\congr\}_{U\in \mathcal{U}_X}$ is an open cover for ${X}^\congr$, it follows that $\{X_\alpha^\congr\}_{\alpha\in A}$ is an open cover of ${X}^\congr$. 

We are all set to define $\varphi^\congr:Y^\congr\to X^\congr$ for an arbitrary morphism $\varphi:Y\to X$. We choose an affine open cover $\{X_\alpha\}_{\alpha\in A}$ of $X$ and an affine open cover $\{Y_\alpha\}_{\alpha\in A}$ of $Y$ such that $\phi(Y_\alpha)\subseteq X_\alpha$ for each $\alpha \in A$.  We also choose an affine cover $\{Y_{\alpha\beta k}\}_{k\in K}$ for each $Y_{\alpha\beta}=Y_\alpha\cap Y_\beta$ and an affine cover $\{X_{\alpha\beta k}\}_{k\in K}$ for each $X_{\alpha\beta}=X_\alpha\cap X_\beta$ such that $\phi(Y_{\alpha\beta k})\subseteq X_{\alpha\beta k}$ for each $k\in K$. Let $\phi_\alpha := \phi|_{Y_\alpha}:Y_\alpha\to X_\alpha$ and $\phi_{\alpha\beta k} := \phi|_{Y_{\alpha\beta k}}:Y_{\alpha\beta k}\to X_{\alpha\beta k}$. The corresponding morphisms of congruence spaces are given by $\phi_\alpha^\congr:Y_\alpha^\congr\to X_\alpha^\congr$ and $\phi_{\alpha\beta k}^\congr:Y_{\alpha\beta k}^\congr\to X_{\alpha\beta k}^\congr$. We define
\begin{align*}
	{\phi}^\congr:{Y}^\congr=\bigcup_{\alpha\in A}Y_\alpha^\congr&\longrightarrow {X}^\congr=\bigcup_{\alpha\in A}X_\alpha^\congr \quad,\quad
	\mathfrak{q}&\mapsto \phi_\alpha^\congr(\mathfrak{q})\quad\text{ if~}\mathfrak{q}\in Y_\alpha^\congr.
\end{align*}
Assume $\mathfrak{q}\in {Y_\beta}^\congr$ for a different choice of affine open $Y_\beta$  of $Y$.  Since $\{Y_{\alpha\beta k}\}_{k\in K}$ is an affine cover of $Y_{\alpha\beta}$, we know that $\{Y_{\alpha\beta k}^\congr\}_{k\in K}$ is an open cover of $Y_{\alpha\beta}^\congr$. Let $\mathfrak{q}\in Y_{\alpha\beta k}^\congr$ for some $k\in K$. Since $\Cong: \Mon_0 \to \Top$ is a functor by Lemma \ref{cong1}, we have $\phi_\beta^\congr(\mathfrak{q})= \phi_{\alpha\beta k}^\congr(\mathfrak{q})= \phi_\alpha^\congr(\mathfrak{q})$.  This shows that there is no ambiguity in the choice of congruence that represents the point $\mathfrak{q}$ in $Y^\congr$. We leave it to the reader to verify that ${\phi}^\congr$ is also independent of the choice of affine covers considered in its definition by taking suitable common refinements and using similar arguments as above.  This shows that $\varphi^\congr:Y^\congr\to X^\congr$ is well defined as a map. It is clear that $\phi^\congr$ is continuous since each $\phi_\alpha^\congr$ is continuous by Lemma \ref{cong1}.

The only thing left to show is property \eqref{open1} for general $X$ and $Y$. We note that an open immersion $\iota:Y\to X$ is a morphism of monoid schemes which decomposes uniquely into an isomorphism of monoid schemes and the identity inclusion of an open subscheme. We have already shown that for any open subscheme $U$ of $X$, we have $U^\congr$ is an open subspace of $X^\congr$ and we leave it to the reader to verify that any isomorphism of monoid schemes induces a homeomorphism of congruence spaces. It follows that $\iota^\congr:Y^\congr\to X^\congr$ is an open topological embedding.
\end{proof}

\begin{notation*}
 Given a morphism $\varphi:Y\to X$ of monoid schemes, we often write $\widetilde\varphi:\widetilde Y\to\widetilde X$ for $\varphi^\congr:Y^\congr\to X^\congr$. 
\end{notation*}

\subsection{Projection onto the monoid scheme}
\label{subsection: Projection onto the monoid scheme}

The congruence space $ X^\congr$ of a monoid scheme $X$ comes equipped with a continuous surjection $\pi_X: X^\congr\to X$, which sends a prime congruence $\fp$ to its nullideal $I_\fp$ in the affine case. More precisely, $\pi_X$ is characterized by the following properties.

\begin{prop}\label{prop: projection onto the monoid scheme}
 There are continuous maps $\pi_X:X^\congr\to X$, one for every monoid scheme $X$, which are uniquely determined by the following properties:
 \begin{enumerate}
  \item\label{proj1} If $X=\MSpec(A)$ is affine, then $\pi_X(\fp)=I_\fp$ for all prime congruences $\fp\in X^\congr$.
  \item\label{proj2} The maps $\pi_X$ are functorial in $X$, i.e.,
  \begin{equation*}
         \begin{tikzcd}[column sep=60pt]
          Y^\congr \ar[d,"\pi_Y"'] \ar[r,"\varphi^\congr"] & X^\congr \ar[d,"\pi_X"] \\
          Y \ar[r,"\varphi"]                               & X
         \end{tikzcd}
        \end{equation*}
        commutes for every morphism $\varphi:Y\to X$ of monoid schemes.
 \end{enumerate}
\end{prop}

\begin{proof}
 Let $X=\MSpec(A)$ and let $\mathfrak{p}\in X^\congr$ be any prime congruence on $A$. It is clear that $I_{\mathfrak{p}}$ is a prime ideal of $A$ and that $\pi_X$ is a well defined morphism of sets. Also for any basic open $D(a)= \{P\in X\mid a\notin P\}$ of $\MSpec(A)$, we have $ {\pi_X}^{-1}(D(a)) = \{ \mathfrak{p}\in {X}^\congr\mid  a\notin I_{\mathfrak{p}}\}= U_{a,0}$ and so $\pi_X$ is a continuous map of topological spaces. Moreover, for $X= \MSpec(A)$ and $Y= \MSpec(B)$, the diagram in \eqref{proj2} commutes since for any $\fp$ in $(\MSpec(A))^\congr$, we have
\begin{equation*}
\pi_Y(	\phi^\congr(\fp)) =  \ \{b\in B\mid (\Gamma_\varphi(b),0)\in\fp\} =\Gamma_\varphi^{-1}(I_\fp)= \varphi(I_\fp).
\end{equation*}

For a general monoid scheme $X$, every $\fp\in X^\congr$ is contained in $U^\congr$ for an affine open subscheme $U=\MSpec(A)$ of $X$ and corresponds to a prime congruence $\fp$ of $A$. We define $\pi_X(\fp):=I_\fp$. Let $U'=\MSpec(A)[s^{-1}]$ be an affine open of $U=\MSpec(A)$ containing $\fp$. Let $\fp'$ be the prime congruence of $A[s^{-1}]$ which corresponds to the prime congruence $\fp$ of $A$. Since the following diagram 
 \begin{equation*}
	\begin{tikzcd}[column sep=60pt]
		(\MSpec(A)[s^{-1}])^\congr \ar[d,"\pi_{U'}"'] \ar[r,"\varphi^\congr"] & (\MSpec(A))^\congr \ar[d,"\pi_U"] \\
		\MSpec(A)[s^{-1}] \ar[r,"\varphi"]                               & \MSpec(A)
	\end{tikzcd}
\end{equation*}
commutes, it shows that $I_\fp$ and $I_{\fp'}$ represents the same point in $X$. Hence if $\fp$ is contained in $V^\congr$ for some other affine open $V=\MSpec(B)$ of $X$, by taking common affine refinement of $U$ and $V$ and using similar argument as above, we can conclude that $\pi_X:X^\congr\to X$ is a well defined map. For general $X$ and $Y$, the morphism $\varphi^\congr: Y^\congr\to X^\congr$ is defined by taking restriction to suitable affine pieces and therefore \eqref{proj2} follows from the affine case.

To check that $\pi_X:X^\congr\to X$ is continuous consider an affine open cover $\{X_\alpha\}_{\alpha\in A}$  of $X$. Then  $\{X_{\alpha}^\congr\}_{\alpha\in A}$ is an open cover of $X^\congr$ and since each $\pi_{X_\alpha}: X_\alpha^\congr \to X_\alpha$ is continuous it follows that $\pi_X:X^\congr\to X$ is continuous. 
\end{proof}

\begin{ex}\label{ex: congruence space of affine space and the diagonal embedding}
 In this example, we describe the congruence space of the affine $n$-space $\A^n_\Fun=\MSpec(\Fun[t_1,\dotsc,t_n])$. By \cite[Example 1.2]{CHWW15}, the prime ideals of $\Fun[t_1,\dotsc,t_n]$ are of the form $P_I=\gen{t_i}_{i\in I}$ where $I$ is a subset of $\{1,\dotsc,n\}$. Let $J=\{1,\dotsc,n\}-I$. Then the quotient $\Fun[t_1,\dotsc,t_n]/P_I$ is isomorphic to $\Fun[t_j\mid j\in J]$. The quotients of $\Fun[t_j\mid j\in J]$ by prime congruence $\fp$ with trivial nullideal $I_\fp=\gen0$ are of the form
 \[
  \fp_H \ = \ \Big\langle \Big( \ \prod_{j\in J} \ t_j^{a_j}, \ \prod_{j\in J} \ t_j^{b_j} \ \Big) \ \Big| \ (a_j-b_j)_{i\in J}\in H \Big\rangle
 \]
 where $H$ is a subgroup of $\Z^J$. Piecing these two types of quotients together, we see that the prime congruences of $\Fun[t_1,\dotsc,t_n]$ are of the form
 \[
  \fp_{I,H} \ = \ \big\langle \ (t_i,0), (s,s') \ \big| \ i\in I, \ (s,s')\in \fp_H \ \big\rangle
 \]
 for subsets $I\subset\{1,\dotsc,n\}$ and subgroups $H$ of $\Z^J$ where $J=\{1,\dotsc,n\}-I$. 

 The congruence spaces of the affine line $\A^1_\Fun$ and the affine plane $\A^2_\Fun$ over $\Fun$ are illustrated in Figure \ref{fig: diagonal embedding line into plane with congruence spaces}. The green areas indicate closed subsets of the respective underlying topological spaces. While both $\A^1_\Fun$ and $\A^2_\Fun$ are finite, and all points and closed subsets appear in the illustration, both $\til\A^1_\Fun$ and $\til\A^2_\Fun$ are infinite, i.e.\ the illustrations only capture a few of their points and closed subsets.
 
 The horizontal arrow is the diagonal embedding $\Delta:\A^1_\Fun\to\A^2_\Fun$. Remarkable is that the image of $\Delta$ (in orange) is not closed in $\A^2_\Fun$, but that the image of $\tilde\Delta$ (in light green) is closed in $\til\A^2_\Fun$.
\end{ex}

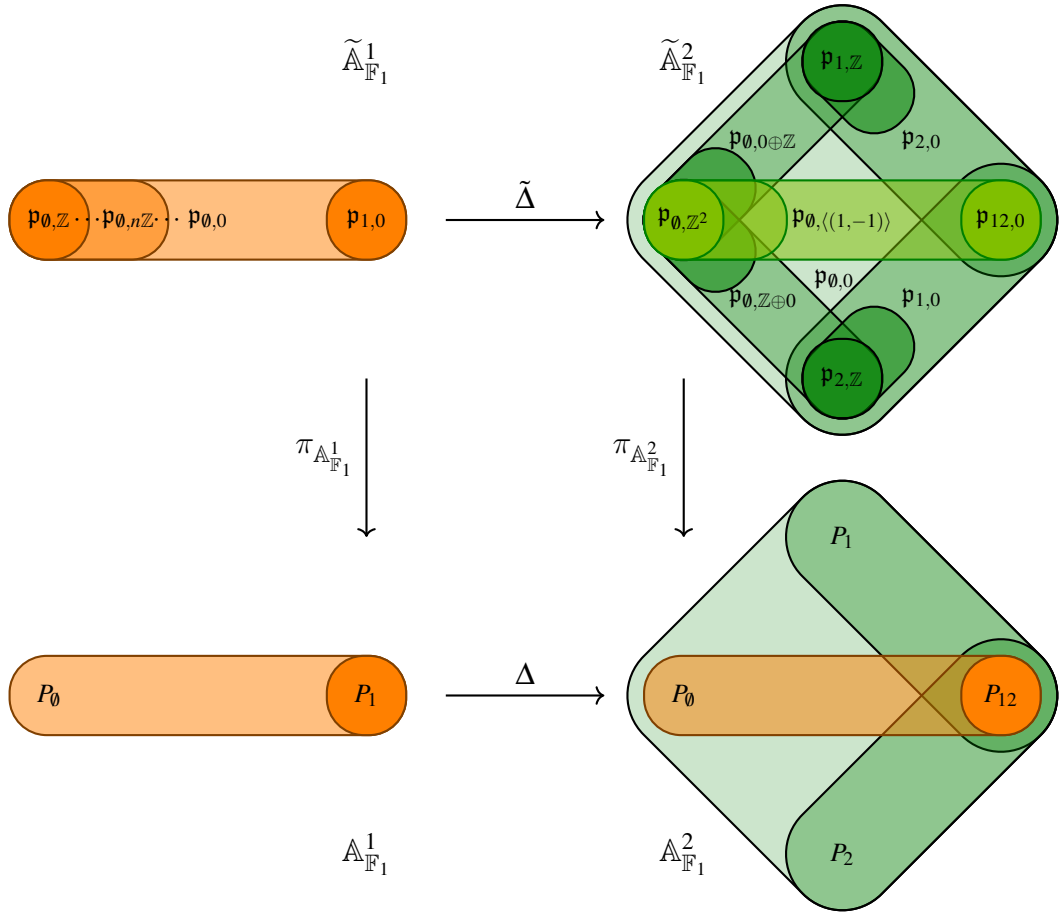
\begin{figure}[tb]
 \[
 \beginpgfgraphicnamed{tikz/fig2}
  \begin{tikzpicture}[x=60pt,y=60pt]
   \draw[fill=green!50!black,fill opacity=0.2,rounded corners=21pt,thick] (4.5,1) -- (6,2.5) -- (7.5,1) -- (6,-0.5) -- cycle;   	
   \draw[fill=green!50!black,fill opacity=0.3,rounded corners=21pt,thick] (5.5,2) -- (6,2.5) -- (7.5,1) -- (7,0.5) -- cycle;   	
   \draw[fill=green!50!black,fill opacity=0.3,rounded corners=21pt,thick] (5.5,0) -- (6,-0.5) -- (7.5,1) -- (7,1.5) -- cycle;   	
   \draw[fill=orange!100!black,fill opacity=0.5,rounded corners=14pt,thick,draw=orange!50!black] (0.75,0.75) rectangle (3.25,1.25);		
   \draw[fill=orange!100!black,fill opacity=1,rounded corners=14pt,thick,draw=orange!50!black] (2.75,0.75) rectangle (3.25,1.25);		
   \draw[fill=orange!100!black,fill opacity=0.5,rounded corners=14pt,thick,draw=orange!50!black] (4.75,0.75) rectangle (7.25,1.25);		
   \draw[fill=orange!100!black,fill opacity=1,rounded corners=14pt,thick,draw=orange!50!black] (6.75,0.75) rectangle (7.25,1.25);		
   \node at (1,1) {\footnotesize $P_\emptyset$};
   \node at (3,1) {\footnotesize $P_1$};
   \node at (5,1) {\footnotesize $P_\emptyset$};
   \node at (6,0) {\footnotesize $P_2$};
   \node at (6,2) {\footnotesize $P_1$};
   \node at (7,1) {\footnotesize $P_{12}$};
   \draw[->,thick] (3.5,1) -- node[above] {$\Delta$} (4.5,1);
   \draw[fill=green!50!black,fill opacity=0.2,rounded corners=21pt,thick] (4.5,4) -- (6,5.5) -- (7.5,4) -- (6,2.5) -- cycle;   
   \draw[fill=green!50!black,fill opacity=0.3,rounded corners=21pt,thick] (5.5,5) -- (6,5.5) -- (7.5,4) -- (7,3.5) -- cycle;   
   \draw[fill=green!50!black,fill opacity=0.3,rounded corners=21pt,thick] (5.5,3) -- (6,2.5) -- (7.5,4) -- (7,4.5) -- cycle;   
   \draw[fill=orange!100!black,fill opacity=0.5,rounded corners=14pt,thick,draw=orange!50!black] (0.75,3.75) rectangle (3.25,4.25);
   \draw[fill=orange!100!black,fill opacity=1,rounded corners=14pt,thick,draw=orange!50!black] (2.75,3.75) rectangle (3.25,4.25);
   \draw[fill=orange!100!black,fill opacity=0.5,rounded corners=14pt,thick,draw=orange!50!black] (0.75,3.75) rectangle (1.75,4.25);
   \draw[fill=orange!100!black,fill opacity=1,rounded corners=14pt,thick,draw=orange!50!black] (0.75,3.75) rectangle (1.25,4.25);
   \draw[fill=green!50!black,fill opacity=0.3,rounded corners=14pt,thick] (6.35,5) -- (6,5.35) -- (4.65,4) -- (5,3.65) -- cycle;   
   \draw[fill=green!50!black,fill opacity=0.3,rounded corners=14pt,thick] (6.35,3) -- (6,2.65) -- (4.65,4) -- (5,4.35) -- cycle;   
   \draw[fill=green!50!black,fill opacity=0.5,rounded corners=14pt,thick] (5.55,4.2) -- (5.2,4.55) -- (4.65,4) -- (5,3.65) -- cycle;   
   \draw[fill=green!50!black,fill opacity=0.5,rounded corners=14pt,thick] (5.55,3.8) -- (5.2,3.45) -- (4.65,4) -- (5,4.35) -- cycle;   
   \draw[fill=green!50!black,fill opacity=0.5,rounded corners=14pt,thick] (6.55,3.2) -- (6.2,3.55) -- (5.65,3) -- (6,2.65) -- cycle;   
   \draw[fill=green!50!black,fill opacity=0.5,rounded corners=14pt,thick] (6.55,4.8) -- (6.2,4.45) -- (5.65,5) -- (6,5.35) -- cycle;   
   \draw[fill=green!50!black,fill opacity=0.5,rounded corners=14pt,thick,draw=green!0!black] (5.75,2.75) rectangle (6.25,3.25);
   \draw[fill=green!50!black,fill opacity=0.5,rounded corners=14pt,thick,draw=green!0!black] (5.75,4.75) rectangle (6.25,5.25);
   \draw[fill=green!50!orange,fill opacity=0.5,rounded corners=14pt,thick,draw=green!50!black] (4.75,3.75) rectangle (7.25,4.25);
   \draw[fill=green!50!orange,fill opacity=0.5,rounded corners=14pt,thick,draw=green!50!black] (4.75,3.75) rectangle (5.65,4.25);
   \draw[fill=green!50!orange,fill opacity=1,rounded corners=14pt,thick,draw=green!50!black] (4.75,3.75) rectangle (5.25,4.25);
   \draw[fill=green!50!orange,fill opacity=1,rounded corners=14pt,thick,draw=green!50!black] (6.75,3.75) rectangle (7.25,4.25);
   \node at (1,4) {\footnotesize $\fp_{\emptyset,\Z}$};
   \node at (1.25,4) {\footnotesize $\dotsb$};
   \node at (1.5,4) {\footnotesize $\fp_{\emptyset,n\Z}$};
   \node at (1.75,4) {\footnotesize $\dotsb$};
   \node at (2,4) {\footnotesize $\fp_{\emptyset,0}$};
   \node at (3,4) {\footnotesize $\fp_{1,0}$};
   \node at (5,4) {\footnotesize $\fp_{\emptyset,\Z^2}$};
   \node at (6,3) {\footnotesize $\fp_{2,\Z}$};
   \node at (6,5) {\footnotesize $\fp_{1,\Z}$};
   \node at (5.5,3.5) {\footnotesize $\fp_{\emptyset,\Z\oplus0}$};
   \node at (5.5,4.5) {\footnotesize $\fp_{\emptyset,0\oplus\Z}$};
   \node at (7,4) {\footnotesize $\fp_{12,0}$};
   \node at (6.5,4.5) {\footnotesize $\fp_{2,0}$};
   \node at (6.5,3.5) {\footnotesize $\fp_{1,0}$};
   \node at (6,4) {\footnotesize $\fp_{\emptyset,\gen{(1,-1)}}$};
   \node at (5.95,3.62) {\footnotesize $\fp_{\emptyset,0}$};
   \draw[->,thick] (3.5,4) -- node[above] {$\tilde\Delta$} (4.5,4);
   \draw[->,thick] (3,3) -- node[left] {$\pi_{\A^1_\Fun}$} (3,2);
   \draw[->,thick] (5,3) -- node[left] {$\pi_{\A^2_\Fun}$} (5,2);
   \node at (3,0) {$\A^1_\Fun$};
   \node at (3,5) {$\til\A^1_\Fun$};
   \node at (5,0) {$\A^2_\Fun$};
   \node at (5,5) {$\til\A^2_\Fun$};
  \end{tikzpicture}
 \endpgfgraphicnamed 
 \]
 \caption{The diagonal embeddings $\Delta:\A^1_\Fun\to\A^2_\Fun$ and $\tilde\Delta:\til\A^1_\Fun\to\til\A^2_\Fun$}
 \label{fig: diagonal embedding line into plane with congruence spaces}
\end{figure}

\subsection{Stalks and residue fields}
\label{subsection: Stalks and residue fields}

The points of a monoid scheme $X$ come with the notion of a stalk and a residue field, which are as follows: let $\cU_{x}$ be the system of all affine open neighbourhood of the point $x$ in $X$. The stalk at $x$ is the pointed monoid
\[
  \cO_{X,x} \ = \ \underset{U\in\cU_{x}}\colim \ \Gamma U
\]
and the residue field at $x$ is the pointed group
 \[
  k(x) \ = \ \cO_{X,x} \, / \, M_x
 \]
where $M_x$ is the maximal ideal of $\cO_{X,x}$. Stalks and residue fields come with canonical morphisms
\[
 \begin{tikzcd}
  \omega_x: \MSpec(\cO_{X,x}) \ \longrightarrow \ X \qquad \text{and} \qquad \kappa_x: \MSpec( k(x)) \ \longrightarrow \ X
 \end{tikzcd}
\]
that commute with the morphism $\MSpec (k(x))\to\MSpec(\cO_{X,x})$ that is induced by the quotient map $\cO_{X,x}\to k(x)$.

Both stalks and residue fields generalize to invariants for points of the congruence space of $X$ as follows.

\begin{defn}\label{stalk}
 Let $X$ be a monoid scheme with congruence space $\til X$ and projection $\pi_X:\til X\to X$. Let $\tilde x\in\til X$ and $x=\pi_X(\tilde x)$. The \emph{stalk at $\tilde x$} is the pointed monoid
 \[
  \cO_{\til X,\tilde x} \ = \ \underset{U\in\cU_{x}}\colim \ \Gamma U \ = \ \cO_{X,x}.
 \]
 where the colimit is a filtered colimit with respect to the restriction maps $\res_{U,V}:\Gamma U\to\Gamma V$ if $V\subseteq U$.
 
 For $U\in\cU_x$, let $i_U:\Gamma U \to\cO_{X,x}=\cO_{\til X,\tilde x}$ the canonical map into the colimit. Let $\fp_{\tilde x,U}$ be the prime congruence of $\Gamma U$ that corresponds to $\tilde x$ and define the congruence
 \[
  \fp_{\tilde x} \ = \ \big\langle \bigcup_{U\in\cU_x} i_{U,\ast}(\fp_{\tilde x,U})\big\rangle
 \]
 of $\cO_{\til X,\tilde x}$. The \emph{residue field at $\tilde x$} is the pointed group
 \[
  k(\tilde x) \ = \ \cO_{\til X,\tilde x} \, / \, \fp_{\tilde x}.
 \]
\end{defn}

\begin{lemma}
 Let $X$ be a monoid scheme with congruence space $\til X$. Let $\tilde x\in \til X$ and $x=\pi_X(\tilde x)$. Then $\fp_{\tilde x}$ is a prime congruence of $\cO_{\til X,\tilde x}$ and $k(\tilde x)$ is a pointed group. More precisely, $\fp_{\tilde x}=i_{U,\ast}(\fp_{\tilde x,U})$ for any $U\in\cU_x$ and $k(\tilde x)$ is isomorphic to $\Frac(\Gamma U/\fp_{\tilde x,U})$. 
\end{lemma}

\begin{proof}
 For any $U, V\in \mathcal{U}_x$, we want to show that $i_{U,\ast}(\fp_{\tilde x,U})= i_{V,\ast}(\fp_{\tilde x,V})$. In fact, it is enough to check this for any basic affine open $V$ of $U$. Since we have $i_V\circ \res_{U,V} = i_U$ it follows that $i_{V,\ast}(\fp_{\tilde x,V}) = i_{U,\ast}(\fp_{\tilde x,U})$. Thus $\fp_{\tilde x}=i_{U,\ast}(\fp_{\tilde x,U})$. Also, we have $k(\tilde{x})=\ \cO_{\til X,\tilde x} \, / \, \fp_{\tilde x} =\ \cO_{\til U,\tilde x} \, / \, i_{U,\ast}(\fp_{\tilde x,U})=S^{-1}A \, / \, S^{-1}(\fp_{\tilde x,U})$ for any $U=\MSpec(A)\in \mathcal{U}_x$ and $S=A\setminus \pi_U(\fp_{\tilde x,U})=\{a\in A~|~(a,0)\notin \fp_{\tilde x,U} \}$. Let $\phi: A\to A/\fp_{\tilde x,U}$ denote the projection map. Since localizations of monoids are exact (in the sense of Lemma \ref{lemma: kernels of localization}), we have $S^{-1}A \, / \, S^{-1}(\fp_{\tilde x,U}) = \phi(S)^{-1}(A \, / \, \fp_{\tilde x,U})= \Frac (A \, / \, \fp_{\tilde x,U})$. 
\end{proof}

Note that a morphism $\phi:X\to Y$ of monoid schemes induces morphisms 
\begin{equation*}
{{\tilde\phi}}_{\mathfrak{p}}:\mathcal{O}_{\widetilde{Y},\mathfrak{q}}\longrightarrow \mathcal{O}_{\widetilde{X},\mathfrak{p}}\qquad \text{and} \qquad k(\mathfrak{q})\longrightarrow k(\mathfrak{p})
\end{equation*}
between stalks and residue fields for every $\mathfrak{p}\in \widetilde{X}$ and ${\tilde\phi}(\mathfrak{p}) =\mathfrak{q}\in \widetilde{Y}$. Note further that the morphism $k(\mathfrak q)\to k(\mathfrak p)$ is injective in contrast to the induced maps between residue fields associated with prime ideals.

\subsection{Fibres over the monoid scheme}
\label{subsection: Fibres over the monoid scheme}

Let $X$ be a monoid scheme with congruence space $\til X$ and projection $\pi_X:\til X\to X$. In this section, we study the fibres $\pi^{-1}_X(x)$ over the points $x$ of $X$. 

\begin{prop}\label{prop: fibre of pi_X}
 Let $x\in X$ and $k(x)$ be the residue field at $x$. The canonical morphism $\kappa_x:\MSpec(k(x))\to X$ induces a homeomorphism $\tilde\kappa_x:\Cong (k(x)) \to \pi_X^{-1}(x)$.
\end{prop}

\begin{proof}
 Since $\im(\kappa_x)=\{x\}$ and the diagram 
   \begin{equation*}
 	\begin{tikzcd}[column sep=60pt]
 		\Cong(k(x)) \ar[d,"\pi_{\MSpec(k(x))}"'] \ar[r,"\til{\kappa_x}"] & \til X \ar[d,"\pi_X"] \\
 		\MSpec(k(x)) \ar[r,"\kappa_x"]                               & X
 	\end{tikzcd}
 \end{equation*}
 commutes, we have $im(\til {\kappa_x})\subseteq \pi_X^{-1}(x)$. Moreover, we can assume that $X=\MSpec(A)$ is affine and $x$ corresponds to a prime ideal $P$ of $A$.
 
 Since $A_P\rightarrow A_P/ {PA_P}$ is a surjective morphism, $ \MSpec (A_P/ {PA_P})\rightarrow \MSpec(A_P)$ is a closed immersion and therefore $ \Cong (A_P/ {PA_P})\rightarrow \Cong(A_P)$ is injective. By Proposition \ref{prop: prime congruences in localizations}, we know that $\Cong(A_P) \rightarrow \Cong(A)$ is an injection. Thus $\til{\kappa_x}: \Cong(k(x)) \cong \Cong (A_P/{PA_P})\rightarrow \pi_{X}^{-1}(P) \subseteq  \pi_X^{-1}(x)$ is an injection. 
 
 In order to show surjectivity, consider any $\mathfrak{c}\in  \pi_{X}^{-1}(P) $. Since $ \pi_{X}(\mathfrak{c})=I_\mathfrak{c} =P$, we have $I_\mathfrak{c}\cap {(A\setminus P)} = \emptyset$ and therefore $\mathfrak{c}\in \Cong(A)_P$ by Proposition \ref{prop: prime congruences in localizations}. Moreover, since $(a,0)\in \mathfrak{c}$ for all $a\in P$, we have $\mathfrak{c}\in \Cong (A_P/PA_P)$ and this shows $\til \kappa_x$ is surjective. 
 
 Let $\varphi: A\to A_P\to A_P/PA_P$ be the natural map. Since $\til \kappa_X(U_{\varphi(a),\varphi(b)})= U_{a,b}\cap  \pi_{X}^{-1}(P) $ and $\varphi$ is surjective, it follows that $\til \kappa_x$ is a homeomorphism.
\end{proof}

As a consequence of Proposition \ref{prop: fibre of pi_X}, we find two natural section $\sigma_X:X\to\til X$ and $\tau_X:X\to\til X$ of $\pi_X$: the former section $\sigma_X$ send a point $x$ of $X$ to the image $\kappa_x(\fc_\triv)$ of the trivial congruence $\fc_\triv$ on $k(x)$, which is prime since $k(x)$ is integral; the latter section $\tau_X$ sends $x$ to the image $\til{\kappa}_x(\fm)$ of congruence kernel $\fm$ of the unique morphism $k(x)\to\Fun$, which sends all units of $k(x)$ to $1$.

\begin{rem}
 The sections $\sigma_X$ and $\tau_X$ are less well-behaved than the projection $\pi_X$. Neither of them is continuous in general, as the example $X=\MSpec(A)$ for $A=\{0,e,1\}$ with idempotent $e=e^2$ shows, which is illustrated in Figure \ref{fig: bijection pi_X that is not a homeomorphism}. The map $\pi_X:\til X\to X$ is a bijection, which fails to be a homeomorphism since $\til X$ is discrete, but $X$ is not. Thus both $\sigma_X$ and $\tau_X$ are equal to the inverse bijection to $\pi_X$ and fail to be continuous.
 
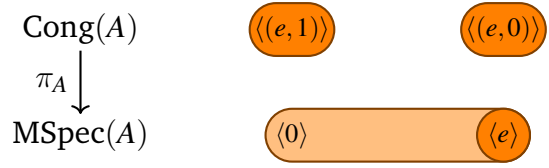
\begin{figure}[tb]
 \[
 \beginpgfgraphicnamed{tikz/fig3}
  \begin{tikzpicture}[x=40pt,y=40pt]
   \draw[fill=orange!100!black,fill opacity=0.5,rounded corners=10pt,thick,draw=orange!50!black] (0.75,0.75) rectangle (3.25,1.25);		
   \draw[fill=orange!100!black,fill opacity=1,rounded corners=10pt,thick,draw=orange!50!black] (2.75,0.75) rectangle (3.25,1.25);		
   \node at (1,1) {\footnotesize ${\gen 0}$};
   \node at (3,1) {\footnotesize ${\gen e}$};
   \draw[fill=orange!100!black,fill opacity=1,rounded corners=10pt,thick,draw=orange!50!black] (2.60,1.75) rectangle (3.40,2.25);
   \draw[fill=orange!100!black,fill opacity=1,rounded corners=10pt,thick,draw=orange!50!black] (0.60,1.75) rectangle (1.40,2.25);
   \node at (1,2) {\footnotesize ${\gen{(e,1)}}$};
   \node at (3,2) {\footnotesize ${\gen{(e,0)}}$};
   \node at (-1,1) {$\MSpec(A)$};
   \node at (-1,2) {$\Cong(A)$};
   \draw[->,thick] (-1,1.8) -- node[left] {$\pi_A$} (-1,1.2);
  \end{tikzpicture}
 \endpgfgraphicnamed 
 \]
 \caption{The projection $\pi_A:\Cong(A)\to\MSpec(A)$ for $A=\{0,e,1\}$}
 \label{fig: bijection pi_X that is not a homeomorphism}
\end{figure}

 It follows from Proposition \ref{prop: closed points of the congruence space} that $\tau_X$ is functorial in $X$. The section $\sigma_X$ is functorial in open immersions, but fails to be functorial in general, as the example of the diagonal embedding $\Delta:\A^1_\Fun\to\A^2_\Fun$ shows:
 \[
  \til\Delta\circ\sigma_{\A^1_\Fun}\big(\gen 0\big) = \til\Delta\big(\gen{(0,0)}\big) = \gen{(t_1,t_2)} \ \neq \ \gen{(0,0)} = \sigma_{A^2_\Fun}\big(\gen 0\big) = \sigma_{\A^2_\Fun}\circ\Delta\big(\gen 0\big). 
 \]
\end{rem}

Let $A$ be a pointed monoid. Taking congruence kernels defines an injection $\chi_A:\Hom(A,\Fun)\to \Cong(A)$. Let $X$ be a monoid scheme with congruence space $\til X$ and $X(\Fun)=\Hom(\MSpec(\Fun), X)$. Since the congruence kernel of a composition is the pullback of a congruence kernel, the maps $\chi_A$ are functorial in $A$. Therefore the injections
\[
 \chi_{\Gamma U}: \ U(\Fun) \ = \ \Hom(\Gamma U,\Fun) \ \longrightarrow \ \Cong(\Gamma U) \ = \ \til U
\]
for affine open subschemes $U$ of $X$ glue to a functorial injection
\[
 \chi_X: \ X(\Fun) \ \longrightarrow \ \til X.
\]

\begin{prop}\label{prop: closed points of the congruence space}
 Let $X$ be a monoid scheme with congruence space $\til X$.
 \begin{enumerate}
  \item\label{closedpoint1} The image of $\chi_X:X(\Fun)\to\til X$ is the set of closed points of $\til X$.
  \item\label{closedpoint2} The composition $\pi_X\circ\chi_X:X(\Fun)\to X$ is a bijection and $\tau_X$ is the composition of its inverse bijection with $\chi_X$. In particular, $\tau_X$ is functorial in $X$ and gives a bijection between $X$ and the set of closed points of $\til X$.
 \end{enumerate}
\end{prop}

\begin{proof}
 All assertions are local in nature, so we can assume that $X=\MSpec(A)$ is affine. 
 
 We begin with \eqref{closedpoint1}. Consider a congruence $\mathfrak{c}$ in the image of $\chi_A: \Hom(A, \mathbb{F}_1)\to \Cong(A)$, i.e.\ $\mathfrak{c} = \congker (f)$ for some $f\in\Hom(A,\mathbb{F}_1)$.  Assume that $\mathfrak{c}$ is not closed in $\til X$. Then there exists a prime congruence $\mathfrak{c'}$ that properly contains $\mathfrak{c}$. Thus there is an $(a_1, a_2)\in \mathfrak{c'}$ with $f(a_1)\neq f(a_2)$. Without loss of generality, let $f(a_1)= 0$ and $f(a_2)= 1$, i.e.\ $(0, a_1), (a_2,1)\in \mathfrak{c}\subseteq \mathfrak{c'}$. By transitivity, $(0,1)\in \mathfrak{c'}$ which implies $\mathfrak{c'}= A\times A$ which yields the desired contradiction. Hence $\mathfrak{c}$ is a closed point in $\Cong(A)$. 
 
 Conversely, let $\mathfrak{c}$ be any closed point in $\Cong(A)$. It follows from the maximality of $\mathfrak{c}$ that $A/\mathfrak{c}\cong \mathbb{F}_1$ and thus $\mathfrak{c} = \congker (A\to A/\mathfrak{c}\cong \mathbb{F}_1)$. This proves \eqref{closedpoint1}. 
 
 We continue with \eqref{closedpoint2}. Let $f, g \in X(\mathbb{F}_1)$ satisfy that $\pi_X\circ\chi_X(f) = \pi_X\circ\chi_X(g)$.  This implies $f^{-1}(0)=g^{-1}(0)$. Since $\im f=\im g=\{0,1\}$ has only $2$ elements, this implies that $f=g$. Thus $\pi_X\circ\chi_X$ is injective. For $P\in X$, we define the map $f_P: A\to \mathbb{F}_1$ by $f_P(a)= 0$ if $a\in P$ and $f_P(a)= 1 $ if $a\notin P$. Then $\pi_X\circ\chi_X(f_P) = P$, which shows that $\pi_X\circ\chi_X$ is surjective. This shows that  $\pi_X\circ\chi_X:X(\Fun)\to X$ is a bijection. 
 
 Next we show that $\tau_X=\chi_X\circ(\pi_X\circ\chi_X)^{-1}$. Let $f\in X(\mathbb{F}_1)$ and $P=\ker(f)$. Then $\tau_X\circ\pi_X\circ\chi_X(f) = \tau_X(P)$. Let $\varphi: A\to A_P\to A_P/PA_P\cong k(P)$ be the natural map and let $\fm$ denote the congruence kernel of the unique morphism $k(P)\to\Fun$, which sends all units of $k(P)$ to $1$. Then 
 \begin{multline*}
 \tau_X(P) \ = \ \til {\kappa}_P(\fm)=\{(a,a')\in A\times A~|~a,a'\in P\text{~or~}a,a'\in A\setminus P\}\\ 
 =\{(a,a')\in A\times A~|~f(a)=f(a')\} \ = \ \congker(f) \ = \ \chi_X(f).
 \end{multline*}
 Thus $\tau_X\circ\pi_X\circ\chi_X=\chi_X$ as desired. 
 
 It follows that $\tau_X$ gives a bijection between $X$ and the set of closed points of $\til X$. Since both $\chi_X$ and $\pi_X$ are functorial in $X$, so is $\tau_X$.
\end{proof}

\begin{rem}
 The natural maps $\chi_X:X(\Fun)\to \til X$ and $\pi_X\circ\chi_X:X(\Fun)\to X$ endow $X(\Fun)$ with two distinct initial topologies, which differ in general. Both topologies are, in fact, fine topologies that are induced by topologies for $\Fun$. 
 
 The concept of fine topologies was introduced in \cite{Lorscheid-Salgado16} in the case of rational point sets for usual schemes and later applied for semiring schemes and ordered blue schemes in \cite{Lorscheid23}. We recall the definition tailored to monoid schemes.
 
 Let $A$ be a monoid together with a topology and $X$ a monoid scheme. The \emph{fine topology for $X(A)$} is the finest topology such that for every morphism $U\to X$ from an affine monoid scheme $U$ into $X$, the induced map $U(A)\to X(A)$ is continuous where $U(A)=\Hom(\Gamma U,A)$ carries the compact open topology with respect to the discrete topology for $\Gamma U$. 
 
 The initial topology for $X(\Fun)$ with respect to $\chi_X:X(\Fun)\to \til X$ coincides with the fine topology for $X(\Fun)$ that is induced by the discrete topology for $\Fun$.
 
 The morphism $\pi_X\circ\chi_X:X(\Fun)\to X$ is a bijection and thus the initial topology for $X(\Fun)$ corresponds to the Zariski topology of $X$. This equals the fine topology for $X(\Fun)$ that is induced by the topology for $\Fun$ for which $\{0\}$ is closed, but $\{1\}$ is not. 
\end{rem}

\section{Strongly reduced monoid schemes}
\label{section: Strongly reduced monoid schemes}

In this section, we develop the theory of (strong) reductions of monoids and monoid schemes. Namely, as a consequence of the discrepancy between ideals and congruences, we obtain to distinct concepts of reduction of a monoid: the (weak) reduction quotients out the intersection of all prime ideals, the strong reduction quotients out the intersection of all prime congruences. These concepts carry over to the monoid schemes.

\subsection{Radicals}
\label{subsection: Radicals}

\begin{defn}
 Let $A$ be a pointed monoid and $I$ an ideal of $A$. The \emph{radical of $I$} is the ideal
 \[
  \sqrt{I} \ = \ \{ a\in A \mid a^n\in I \text{ for some }n\geq0\}.
 \]
 A \emph{radical ideal of $A$} is an ideal $I$ of $A$ with $\sqrt{I}=I$. The \emph{nilradical of $A$} is the radical ideal $\Nil(A)=\sqrt{\{0\}}$.

 Let $\fc$ be a congruence on $A$. The \emph{radical of $\fc$} is the congruence 
 \[
  \sqrt{\fc} \ = \ \big\{(a,b)\in A\times A \, \big| \, (aa^n,ba^n)\in\fc \text{ and }(ab^n,bb^n)\in\fc \text{ for some }n\geq 0 \big\}.
 \]
 A \emph{radical congruence} is a congruence $\fc$ with $\sqrt{\fc}=\fc$. The \emph{strong nilradical of $A$} is the radical $\nil(A)=\sqrt{\fc_\triv}$ of the trivial congruence $\fc_\triv$.
\end{defn}

\begin{prop}\label{prop: radicals}
 Let $A$ be a pointed monoid, $I$ an ideal of $A$ and $\fc$ a congruence on $A$. Then
 \[
  \sqrt{I} \ = \ \bigcap_{\substack{I\subset P,\\ P \text{ prime}}} P, \qquad \sqrt{\fc} \ = \ \bigcap_{\substack{\fc\subset \fp,\\ \fp \text{ prime}}} \fp \qquad \text{and} \qquad I_{\sqrt{\fc}} \ = \ \sqrt{I_\fc}.
 \]
\end{prop}

\begin{proof}
 As a first step, we prove the claims for $I=\{0\}$ and $\fc=\fc_\triv$, for which we have $\sqrt{I}=\Nil(A)$ and $\sqrt{\fc}=\nil(A)$. Let $X=\MSpec(A)$ and $\til X=\Cong(A)$. The inclusions $\Nil(A)\subseteq\bigcap_{P\in X} P$ and $\nil(A)\subseteq\bigcap_{\fp\in \til X} \fp$ follow directly from the definitions. 
 
 Conversely consider $a\in A\setminus \Nil(A)$. Then $0\notin S_a=\{a^i\}_{i\geq 0}$ and $A[a^{-1}]= S_a^{-1}A\neq \{0\}$. Let $q$ denote the unique maximal ideal $M_{A[a^{-1}]}=A[a^{-1}]\setminus A[a^{-1}]^\times$ of $A[a^{-1}]$ and let $P=i_a^{-1}(Q)$ under the natural map $i_a: A\to A[a^{-1}]$. Then $P$ is a prime ideal of $A$ and $a\notin P$. This proves that $\Nil(A)=\bigcap_{P\in X}  P$.
 
 Next we consider $a,b\in A$ such that $(a,b)\notin \mathfrak{nil}(A)$. Without loss of generality, we can assume $aa^k\neq ba^k$ for all integers $k\geq 0$. Then $\frac{a}{1}\neq \frac{b}{1}$ in $A[a^{-1}]=S_a^{-1}(A)$ where $S_a=\{a^i\}_{i\geq 0}$. We consider the quotient map $\pi:A[a^{-1}]^\times \to A[a^{-1}]/M_{A[a^{-1}]}$ which is clearly injective. Thus $\pi(\frac{a}{1})\neq \pi(\frac{b}{1})$ in $A[a^{-1}]/M_{A[a^{-1}]}$. Since $A[a^{-1}]/M_{A[a^{-1}]}$ is a pointed group (and therefore integral), the congruence $\mathfrak{c}=\congker(A\to A[a^{-1}]/M_{A[a^{-1}]})$ is prime and $(a,b)\notin \mathfrak{p}$. This proves that $\mathfrak{nil}(A)= \bigcap_{\mathfrak{p}\in \widetilde{X}}\mathfrak{p}$.
	 
 Consider any $ a\in I_{\mathfrak{nil}(A)}$. Then $	(a,0)\in \mathfrak{nil}(A)$ which implies $a.a^n=0.a^n =0 $ for some integer $n\geq 0$. Therefore $a\in \Nil(A)$. Conversely, consider any $ a\in \Nil(A)$. Then $a^n=0 $ for some $n\geq 0$. This implies $a.a^{n-1}=0=0.a^{n-1}$ and $n.0^{n-1}=0=0.0^{n-1}$. Thus $(a,0)\in \mathfrak{nil}(A)$, i.e.  $a\in I_{\mathfrak{nil}(A)}$. This proves that $I_{\mathfrak{nil}(A)}=\Nil(A)$.
	 
 We turn to the general case and consider an arbitrary ideal $I$ of $A$. The previous arguments show that $\Nil(A/I)=\bigcap_{P\in \MSpec(A/I)}~ P$. Since the prime ideals of $A/I$ can be identified with the prime ideals of $A$ containing $I$ and since $\Nil(A/I)$ can be identified with $\sqrt{I}$ (see Proposition \ref{prop: congruences in quotients}), it follows that 
 \[
  \sqrt{I} \ = \ \bigcap_{\substack{I\subset P,\\ P \text{ prime}}} P.
 \]
 
 Let $\fc$ be an arbitrary congruence of $A$. The previous arguments show that $\mathfrak{nil}(A/I)=\bigcap_{\fp\in \Cong(A)/I}~  \fp$. By Proposition \ref{prop: congruences in quotients}, the prime congruences of $A/\fc$ pull back to the prime congruences of $A$ that contain $\fc$. Thus $\mathfrak{nil}(A/\fc)$ pulls back to $\sqrt{\fc}$. Therefore 
 $$ \sqrt{\fc} \ = \ \bigcap_{\substack{\fc\subset \fp,\\ \fp \text{ prime}}} \fp. $$ 
 
 To conclude the proof, note that $a\in I_{\sqrt{\fc}}$ if and only if $ (a,0)\in \sqrt{\fc} $, i.e. $(a^n,0)\in \fc$ for some $n>0$. This is equivalent to $a^n\in I_\fc $ or, in other words, $ a\in \sqrt{I_\fc}$.
\end{proof}

\begin{defn}
 Let $A$ be a pointed monoid. The \emph{reduction of $A$} is $A^\red= A/\Nil(A)$. The \emph{strong reduction of $A$} is $A^\sred= A/\mathfrak{nil}(A)$. We say that $A$ is \emph{reduced} if $A=A^\red$ and that $A$ is \emph{strongly reduced} if $A=A^\sred$.
\end{defn}

\begin{lem}\label{induced-sred}
 Let $f:A\to B$ be a morphism of pointed monoids. If $B$ is reduced, then $f$ factors uniquely through $A\to A^\red$. If $B$ is strongly reduced, then  $f$ factors uniquely through $A\to A^\sred$. This defines reflections 
 \[
  (-)^\red: \ \Mon_0 \ \longrightarrow \ \Mon_0 \qquad \text{and} \qquad (-)^\sred: \ \Mon_0 \ \longrightarrow \ \Mon_0 
 \]
 onto the respective subcategories of reduced and strongly reduced pointed monoids.
\end{lem}

\begin{proof}
 Let $f:A\to B$ be a morphism of pointed monoids. If $a^n=0$ for some $a\in A$ and $n\geq 0$, then $f(a)^n=0$ in $B$. Thus if $B$ is reduced, then $f(a)=0$ and $f$ factors (necessarily uniquely) through $A^\red$. If $aa^n=ba^n$ and $ab^n=bb^n$ for some $a,b\in A$ and $n\geq 0$, then $f(a)f(a)^n=f(b)f(a)^n$ and $f(a)f(b)^n=f(b)f(b)^n$ in $B$. Thus if $B$ is strictly reduced, then $f(a)=f(b)$ in $B$ and $f$ factors through (necessarily uniquely) $A^\sred$. 
 
 These two universal properties show that the embeddings of the subcategories of reduced and strongly reduced monoids into $\Mon$ have respective left adjoints, which gives rise to the reflections $(-)^\red:\Mon_0\to \Mon_0$ and  $(-)^\sred:\Mon_0\to \Mon_0$.
\end{proof}

\begin{rem}
 The difference between the nilradical and the strong nilradical extends the difference of zero divisor free and integral monoids in the following sense. If $A$ is a zero divisor free monoid, then $A$ is reduced and $\Nil(A)=\{0\}$. On the other hand, $\nil(A)=\congker(A\to\Frac A)$ might be non-trivial, and $A^\sred=A^\overint$.
\end{rem}

\subsection{Compatibility with localizations}

\begin{lem}\label{red-comm}
	Let $A$ be a pointed monoid and $\pi:A \to A^\sred$ the quotient map. Let $S$ be a multiplicative subset of $A$. Then $(S^{-1}A)^\sred= \pi(S)^{-1}A^\sred$.
\end{lem}

\begin{proof}
	We denote the multiplicative set $\pi(S)$ of $A^\sred$ by $\bar{S}$.
	We claim that the canonical morphisms 
	\[
	\begin{array}{cccc}
		f: & A & \longrightarrow & \bar{S}^{-1}A^\sred\\
		& a & \longmapsto     & \frac{[a]}{[1] } 
	\end{array}
	\qquad\qquad 
	\begin{array}{cccc}
		g: & A & \longrightarrow &(S^{-1}A)^\sred\\
		& a & \longmapsto     & [\frac{a}{1}]
	\end{array}
	\]
	descend to morphisms 
	\begin{equation*}
		\bar{f_s}:(S^{-1}A)^\sred\longrightarrow  \bar{S}^{-1}A^\sred \qquad\qquad \bar{g}_{\bar{s}}: \bar{S}^{-1}A^\sred\longrightarrow (S^{-1}A)^\sred
	\end{equation*}
	which are mutually inverse to each other since both are induced from the identity map $A\xrightarrow{id} A$. We proceed to show how to obtain the above maps. Since $f(s)=\frac{[s]}{[1]}$ is invertible in  $\bar{S}^{-1}A^\sred$, there is a unique morphism $f_s: S^{-1}A\to \bar{S}^{-1}A^\sred$ such that $f$ factors as follows
	\begin{equation*}
		\xymatrix{
			A\ar[r]^{f} \ar@<-2pt>[d]_{} &  \bar{S}^{-1}A^\sred\\
			S^{-1}A\ar[ur]^{f_s} &
		}
	\end{equation*}
	The morphism $f_s: S^{-1}A\to \bar{S}^{-1}A^\sred$ is explicitly given by $f_s(\frac{a}{s})=\frac{\pi{(a)}}{\pi{(s)}}$. Consider an element $(\frac{a}{s^i},\frac{b}{s^j})\in \mathfrak{nil}(S^{-1}A)$. Since $\frac{a}{s^i}=\frac{as^j}{s^{i+j}}$ and $\frac{b}{s^j}=\frac{bs^i}{s^{i+j}}$, we can assume that any element of $ \mathfrak{nil}(S^{-1}A)$ is of the form $(\frac{a}{s^i},\frac{b}{s^i})$. Thus $(\frac{a}{s^i},\frac{b}{s^i})\in \mathfrak{nil}(S^{-1}A)$ means that $\frac{a}{s^i}(\frac{c}{s^i})^N = \frac{b}{s^i}(\frac{c}{s^i})^N$ in $S^{-1}A$ for $c=a,b$ and for some integer $N\gg 0$. This implies $s^ks^ia(s^ic)^N=s^ks^ib(s^ic)^N$ in $A$ for some integer $k\geq 0$, $c=a,b$ and $N\gg0$. In fact, we can assume $k=0$ if $N$ is large enough, i.e., $s^ia(s^ic)^N=s^ib(s^ic)^N$ for $c=a,b$ and $N\gg0$. Then $ (s^ia,s^ib)\in \mathfrak{nil}(A)$ and therefore ${\pi(s)}^i\pi(a)={\pi(s)}^i\pi(b)\text{~in~} A^\sred$ which implies $ \frac{\pi(a)}{{\pi(s)}^i}= \frac{\pi(b)}{{\pi(s)}^i}\text{~in~} \bar{S}^{-1}A^\sred$.
	Thus there is a unique morphism $\bar{f_s}:(S^{-1}A)^{\sred}\to \bar{S}^{-1}A^\sred$ such that $f_s$ factors as follows 
	\begin{equation*}
		\xymatrix{
			S^{-1}A\ar[r]^{f_s} \ar@<-2pt>[d]_{} & \bar{S}^{-1}A^\sred\\
			(S^{-1}A)^{\sred}\ar[ur]^{\bar{f_s}} &
		}
	\end{equation*}
	The morphism $\bar{f_s}:(S^{-1}A)^{\sred}\to  \bar{S}^{-1}A^\sred$ is explicitly given by $\bar{f_s}([\frac{a}{s}])=\frac{\bar{a}}{\bar{s}}$. Next, we consider $(a,b)\in \mathfrak{nil}(A)$, i.e., $ac^N=bc^N$ for $c=a,b$ and an integer $N\gg 0$. This implies $ \frac{a}{1}(\frac{c}{1})^N= \frac{b}{1}(\frac{c}{1})^N \text{~in~} S^{-1}A$ and hence $[\frac{a}{1}]=[\frac{b}{1}] \text{~in~} (S^{-1}A)^\sred$. Thus there is a unique morphism $\bar{g}:A^\sred\to (S^{-1}A)^\sred$ such that $g$ factors as follows
	\begin{equation*}
		\xymatrix{
			A\ar[r]^{g} \ar[d]_{} &  (S^{-1}A)^\sred \\
			A^\sred\ar[ur]^{\bar{g}} &
		}
	\end{equation*}
	where the morphism $\bar{g}:A^\sred\to (S^{-1}A)^\sred$ is explicitly given by $\bar{g}(\bar{a})=[\frac{a}{1}]$. Since $\bar{g}(\bar{s})= [\frac{s}{1}]$ is invertible in $ (S^{-1}A)^\sred$, there is a unique morphism $\bar{g}_{\bar{s}}: {\bar{S}}^{-1}A^\sred\to (S^{-1}A)^\sred$ such that $\bar{g}$ factors as follows
	\begin{equation*}
		\xymatrix{
			A^\sred\ar[r]^{\bar{g}} \ar[d]_{} &  (S^{-1}A)^\sred\\
			{\bar{S}}^{-1}A^\sred\ar[ur]^{{\bar{g}}_{\bar{s}}} &
		}
	\end{equation*}
	The morphism $\bar{g}_{\bar{s}}: {\bar{S}^{-1}}A^\sred\to(S^{-1}A)^\sred$ is explicitly given by $\bar{g}_{\bar{s}}(\frac{\bar{a}}{\bar{s}^i})=[\frac{a}{s^i}]$.
\end{proof}

Proposition \ref{prop: prime congruences in localizations} extends to radical congruences in the following way.

\begin{prop}\label{prop: radical congruences in localizations}
 Let $\iota_S:A\to S^{-1}A$ be a localization of the pointed monoid $A$ at $S$. Then the maps
 \[
  \begin{tikzcd}[column sep=60pt]
   \Big\{\begin{array}{c}\text{radical}\\ \text{congruences on }S^{-1}A\end{array}\Big\} \ar[r,shift left=1,"\iota_S^\ast"] &
   \Big\{\begin{array}{c}\text{radical congruences}\\ \text{$\fc$ on $A$ with $I_\fc\cap S=\emptyset$}\end{array}\Big\} \ar[l,shift left=1,"\iota_{S,\ast}"] 
  \end{tikzcd}
 \]
 are mutually inverse bijections.
\end{prop}

\begin{proof}
 Let $\fd$ be a radical congruence on $S^{-1}A$. Then by Proposition \ref{prop: radicals}, $\fd =\bigcap_{\substack{\fd\subset \fp,\\ \fp \text{ prime}}} \fp$. Thus we have
 \begin{equation*}
 \iota_S^\ast(\fd)=(\iota_S,\iota_S)^{-1}( \underset{\substack{\fd\subset \fp,\\ \fp \text{ prime}}} \bigcap \fp) =  \underset{\substack{\fd\subset \fp,\\ \fp \text{ prime}}} \bigcap (\iota_S,\iota_S)^{-1}(\fp)= \underset{\substack{\fd\subset \fp,\\ \fp \text{ prime}}} \bigcap \iota_S^\ast(\fp)= \underset{\substack{ \iota_S^\ast(\fd)\subset \fq,\\ \fq \text{ prime}}} \bigcap \fq
 \end{equation*}
  where the last equality follows from Proposition \ref{prop: prime congruences in localizations}. This shows that $ \iota_S^\ast(\fd)$ is a radical congruence on $A$.
  
  Conversely, let $\fc$ be a radical congruence on $A$ such that $I_\fc\cap S=\emptyset$. Consider any $(\frac{a}{s}, \frac{a'}{s'}) \in \sqrt{\iota_{S,*}(\fc)}$. Since $\frac{a}{s}$ is equivalent to $\frac{as'}{ss'}$ and $\frac{a'}{s'}$ is equivalent to $\frac{a's}{ss'}$ in $S^{-1}A$, we have $(\frac{a}{s}, \frac{a'}{s'}) \in \sqrt{\iota_{S,*}(\fc)}$ if and only if $(\frac{as'}{1}, \frac{a's}{1}) \in \sqrt{\iota_{S,*}(\fc)}$. Therefore we have $(\frac{as'}{1}(\frac{as'}{1})^n, \frac{a's}{1}(\frac{as'}{1})^n) \in {\iota_{S,*}(\fc)}$ and $(\frac{as'}{1}(\frac{a's}{1})^n, \frac{a's}{1}(\frac{a's}{1})^n) \in {\iota_{S,*}(\fc)}$ for some $n\gg 0$. This implies there exists some $r\in S$ such that $(as'a^ns'^n r, a'sa^ns'^nr)\in \fc$ and $(as'a'^ns^nr, a'sa'^ns^nr)\in \fc$. Hence we have $(as'r(as'r)^n, a'sr(as'r)^n)\in \fc$ and $(as'r(a'sr)^n,a'sr(a'sr)^n)\in \fc$. Since $\fc$ is radical it follows that $(as'r,~a'sr)\in \fc$ which implies $(\frac{as'}{1}, \frac{a's}{1}) \in \iota_{S,*}(\fc)$ or equivalently, $(\frac{a}{s}, \frac{a'}{s'}) \in {\iota_{S,*}(\fc)}$. This shows that ${\iota_{S,*}(\fc)}$ is radical.
  
  We have therefore shown that both maps are well defined. The proof that the maps are mutually bijective is similar to Proposition  \ref{prop: prime congruences in localizations}.
\end{proof}

\subsection{Strong reduction of monoid schemes}

Let $X=\MSpec(A)$ be an affine monoid scheme. We define $X^\sred=\MSpec(A^\sred)$. By Lemma \ref{induced-sred}, a morphism $\varphi:\MSpec(A)\to \MSpec(B)$ of affine monoid schemes induces a morphism $\varphi^\sred: \MSpec(B^\sred)\hookrightarrow \MSpec(A^\sred)$ between the respective strong reductions.

\begin{lem}\label{o.istrngred}
 Let $i:\MSpec(B)\hookrightarrow \MSpec(A)$ be an open immersion of affine monoid schemes. Then $i^\sred: \MSpec(B^\sred)\hookrightarrow \MSpec(A^\sred)$ is an open immersion of strongly reduced affine monoid schemes.
\end{lem}

\begin{proof}
 Since $i:\MSpec(B)\hookrightarrow \MSpec(A)$ is an open immersion, it follows from \cite[Lemma 1.3]{CHWW15} and \cite[Lemma 2.4]{CHWW15} that we must have $\MSpec(B)\cong\MSpec(A[s^{-1}])$ for some $s\in A$. Therefore by Lemma \ref{red-comm}, it follows that  $i^\sred: \MSpec(B^\sred)\hookrightarrow \MSpec(A^\sred)$ is an open immersion.
\end{proof}

\begin{prop}
 The functor $(-)^\sred$ on affine monoid schemes extends uniquely to a functor $(-)^\sred:\MSch\to\MSch$, which preserves affine open covers.
\end{prop}

\begin{proof}
 Since morphisms are determined by their restrictions to open affines, any endofunctor on $\MSch$ is determined by its restriction to affine monoid schemes, which shows the uniqueness of $(-)^\sred$.
 
 Next we construct $(-)^\sred$ on objects. Let $X$ be a monoid scheme. Consider the family $\mathcal{U}_X=\{U_i\}_{i\in I}$ of all affine open subschemes of $X$. For all $i,j\in I$, consider principal affine opens $U_{ij}\subseteq U_i$ and isomorphisms $\phi_{ji}: U_{ij}\to U_{ji}$. Then by Lemma \ref{o.istrngred}, we have $U_{ij}^\sred$ are open subsets of $U_i^\sred$ for all $i,j\in I$. By Lemma \ref{induced-sred}, we know that $(-)^\sred: \Mon_0 \to \Mon_0,~ A\mapsto A^\sred$ is a functor and therefore it follows that the collection $(I, \{{U_i}^\sred\}_{i\in I}, \{U_{ij}^\sred\}_{i,j\in I}, \{\phi_{ij}^\sred\}_{i,j\in I} )$ gives a gluing data of monoid schemes (in other words, $\mathcal{U}_X^\sred=\{U_i^\sred\}_{i\in I}$ is a monodromy-free diagram in the category of monoid schemes in the sense of \cite[\S 1]{Lorscheid17}). We define $X^\sred$ as the monoid scheme obtained from this gluing data. 
 
 In particular, every open subscheme $U_i\in \mathcal{U}_X$ defines an open subscheme ${U_i}^\sred$ of ${X}^\sred$ and $X^\sred= \bigcup_{U_i\in \mathcal{U}_X}U_i^\sred$. This shows that $(-)^\sred$ preserves affine open covers.

 We continue with the construction of $(-)^\sred$ on morphisms. Let $f: Y\to X$ be a morphism of monoid schemes. We choose an affine open cover $\{X_\alpha\}_{\alpha\in A}$ of $X$ and an affine open cover $\{Y_\alpha\}_{\alpha\in A}$ of $Y$ such that $f(Y_\alpha)\subseteq X_\alpha$ for each $\alpha \in A$.  We also choose an affine cover $\{Y_{\alpha\beta k}\}_{k\in K}$ for each $Y_{\alpha\beta}=Y_\alpha\cap Y_\beta$ and an affine cover $\{X_{\alpha\beta k}\}_{k\in K}$ for each $X_{\alpha\beta}=X_\alpha\cap X_\beta$ such that $f(Y_{\alpha\beta k})\subseteq X_{\alpha\beta k}$ for each $k\in K$. Let $f_\alpha := f|_{Y_\alpha}:Y_\alpha\to X_\alpha$ and $f_{\alpha\beta k} := f|_{Y_{\alpha\beta k}}:Y_{\alpha\beta k}\to X_{\alpha\beta k}$. The corresponding morphisms of strongly reduced affine monoid schemes are respectively given by $f_\alpha^\sred:Y_\alpha^\sred\to X_\alpha^\sred$ and $f_{\alpha\beta k}^\sred:Y_{\alpha\beta k}^\sred\to X_{\alpha\beta k}^\sred$. We define
	\begin{align*}
		f^\sred:{Y}^\sred=\bigcup_{\alpha\in A}Y_\alpha^\sred&\longrightarrow \bigcup_{\alpha\in A}X_\alpha^\sred = {X}^\sred \\
		\mathfrak{q}&\mapsto f_\alpha^\sred(\mathfrak{q})~~~~~~~~~~~~ \text{ if~}\mathfrak{q}\in Y_\alpha^\sred.
	\end{align*}
	To show that $f^\sred$ is well defined, it is enough to check that $\mathfrak{q}\in Y_{\alpha\beta k}^\sred\subseteq Y_\alpha^\sred$ implies $ f_{\alpha\beta k}^\sred(\mathfrak{q})=f_\alpha^\sred(\mathfrak{q})$ which is clear from the functoriality described in Lemma \ref{induced-sred}. It is also clear that $f^\sred$ is continuous. 
\end{proof}

For any affine monoid scheme $U=\MSpec(A)$, the morphism $U^\sred\to U$ is a closed immersion since $A\to A^\sred= A/\mathfrak{nil}(A)$ is surjective. Since being a closed immersion is a local property, it follows from the definition of $X^\sred$ that the natural morphism $i^\sred: X^\sred\to X$ is a closed immersion for any general monoid scheme $X$. 

We call a monoid scheme $X$ \emph{strongly reduced} if $i^\sred: X^\sred\to X$ is an isomorphism of monoid schemes. 

\begin{prop}\label{sred-c.i}
 Let $X$ be a monoid scheme. Then the following hold.
 \begin{enumerate}
  \item\label{strongred1} The maps $i^\sred: X^\sred\to X $ and $\widetilde{i^\sred}: \widetilde{X^\sred}\to \widetilde{X}$ are homeomorphisms.
  \item\label{strongred2} Given a closed immersion $\phi: Z\to X$ such that ${\tilde\phi}: \widetilde{Z}\to\widetilde{X}$ is a homeomorphism, there exists $\psi: X^\sred\to Z$ such that the following diagram commutes
  \begin{equation*}
   \xymatrix{
			Z\ar[r]^{\phi}& X\\
			X^\sred \ar@<-2pt>[u]^{\psi}  \ar[ur]_{i^\sred} &
		}
  \end{equation*}
  \item\label{strongred3} If $X=X^\sred= \MSpec(A)$ and $\widetilde{V_{a,b}}=\widetilde{X}$, then $a=b$.
 \end{enumerate}
\end{prop}

\begin{proof}
 We begin with \eqref{strongred1}. Since $i^\sred: X^\sred\to X $ is a closed immersion, we know that by definition $i^\sred$ is a homeomorphism onto its image. Also $\widetilde{i^\sred}: \widetilde{X^\sred}\to \widetilde{X}$ is a homeomorphism onto its image by Proposition \ref{imm-emb}. By Proposition \ref{prop: projection onto the monoid scheme}, the diagram 
		\begin{equation*}
		\xymatrix{
			\widetilde{X^\sred} \ar[r]^{\widetilde{i^\sred}} \ar@<-2pt>[d]_{\pi_{X^\sred}} & 	\widetilde{X} \ar@<-2pt>[d]^{\pi_X} \\
			X^\sred\ar[r]_{i^\sred} & X
		}
	\end{equation*}
	commutes. Since $\pi_{X^\sred}$ and $\pi_X$ are both surjective, it is enough to show that $\widetilde{i^\sred}$ is surjective to prove $i^\sred$ is surjective. Since the question is local in $X$, we can assume that $X=\MSpec(A)$. Let $p:A\to A^\sred$ be the quotient map. Since every $\mathfrak{c}\in\widetilde{X}$ contains $\mathfrak{nil}(A)$, Proposition \ref{prop: congruences in quotients} implies that the map $\tilde p^\ast:\til X^\sred\to \til X$ is surjective. This shows \eqref{strongred1}.

 We continue with \eqref{strongred2}. Let $\phi: Z\to X$ be a closed immersion such that ${\tilde\phi}: \widetilde{Z}\to\widetilde{X}$ is a homeomorphism. Since the question is local in $X$ and $\phi$ is a closed immersion, we can assume that $X=\MSpec(A)$ and $Z=\MSpec(A/\mathfrak{d})$ for some $\mathfrak{d}\in \Cong(A)$. By assumption, ${\tilde\phi}: \widetilde{Z}\to\widetilde{X}$ is surjective and therefore every $\mathfrak{c}\in \widetilde{X}$ contains $\mathfrak{d}$. Thus $\mathfrak{d}\subseteq\bigcap_{\mathfrak{c}\in \widetilde{X}}\mathfrak{c}= \mathfrak{nil}(A)$ and so there exists $p:A/\mathfrak{d}\to A/\mathfrak{nil}(A)$ such that the following diagram commutes
	\begin{equation*}
	\xymatrix{
		A\ar[r]^{}\ar[dr]^{}& A/\mathfrak{d}\ar[d]_{p} \\
		&A/\mathfrak{nil}(A)
	}
 \end{equation*}
 which yields \eqref{strongred2}.

We continue with \eqref{strongred3}. If $\widetilde{V_{a,b}}=\widetilde{X}$ for $X=X^\sred= \MSpec(A)$ , then $\mathfrak{nil}(A)= \mathfrak{c}_\triv$ and $(a,b)\in \mathfrak{c}$ for all $\mathfrak{c}\in \widetilde{X}$. Thus $(a,b)\in \bigcap_{\mathfrak{c}\in \widetilde{X}}\mathfrak{c}= \mathfrak{nil}(A)=\mathfrak{c}_\triv$ which implies $a=b$. This shows \eqref{strongred3} and concludes the proof.
\end{proof}

\section{Closed immersions and vanishing sets}
\label{section: Closed immersions and vanishing sets}

In this section, we investigate the supports of closed subschemes of a monoid scheme $X$ as subsets of its congruence space $\til X$. A first insight is that these are closed subsets of $\til X$. However, not every closed subset of $\til X$ is the support of a closed subscheme. The aim of this section is to characterize supports of closed subschemes in terms of the vanishing sets of quasi-coherent congruence sheaves. This study leads eventually to a topological characterization of closed immersions.

\subsection{Closed subschemes}

Closed immersions of monoid schemes were introduced in \cite{CHWW15}. We recall the definition. 

\begin{defn}
 A morphism $\phi:Y\to X$ of monoid schemes is \emph{affine} if the inverse image $\varphi^{-1}(U)$ of every affine open subset $U$ of $X$ is affine. A \emph{closed immersion} is an affine morphism $\phi:Y\to X$ for which the pullback $\Gamma U\to\Gamma V$ of sections is surjective for every affine open $U$ of $X$ and $V=\phi^{-1}(U)$. A \emph{closed subscheme of $X$} is an isomorphism class of closed immersions into $X$. The \emph{support of a closed subscheme} represented by a closed immersion $\varphi:Y\to X$ is the image of $\til\varphi$ in $\til X$.
\end{defn}

Note that isomorphic closed immersions have the same image in $\til X$. Therefore the support of a closed subscheme does not depend on the choice of representative. In fact, every closed subscheme of $X$ has a canonical representative $\varphi:Z\to X$, which consists of the inclusion of a topological subspace $Z$ of $X$ together with a quotient sheaf of $(\cO_X)\vert_Z$. In this case, $\til Z$ appears as a subset of $\til X$ and is thus equal to the support of the closed subscheme $Z$.

The relevance of congruence spaces is that they have enough closed subsets to include the supports of closed subschemes. This is made precise in the following result.

\begin{prop}\label{prop: closed immersions of congruence spaces}\label{imm-emb}
 Let $\varphi:Y\to X$ be a morphism of monoid schemes and $\widetilde\varphi:\widetilde Y\to\widetilde X$ the associated map of congruence spaces. If $\varphi$ is a closed immersion, then $\widetilde\varphi$ is a closed topological embedding. In particular, the support of a closed subscheme of $X$ is closed in $\til X$.
\end{prop}

\begin{proof}
 The claim of the proposition is local in $X$, so we can assume that $X= \MSpec(A)$ and $Y= \MSpec(A/ \mathfrak{c})$ for some congruence $\mathfrak{c}$ on $A$. We consider the morphism $\Gamma_\iota: A\to A/\mathfrak{c}$ of pointed monoids corresponding to the morphism $\iota:\MSpec(A/ \mathfrak{c})\to \MSpec(A)$. We have $\til \iota(\mathfrak{d})= \{(a,b)\in A\times A\mid (\Gamma_\iota(a),\Gamma_\iota(b))\in\mathfrak{d}\}$ for $\mathfrak{d}\in \til{Y}$. If $\til \iota(\mathfrak{d})= \til \iota(\mathfrak{d'})$ for $\mathfrak{d}, \mathfrak{d'}\in \til{X}$, then $ (\Gamma_\iota(a), \Gamma_\iota(b))\in \mathfrak{d}$ if and only if $(\Gamma_\iota(a), \Gamma_\iota(b))\in \mathfrak{d'}$ for all $(a,b)\in A\times A$.
This implies $\mathfrak{d}=\mathfrak{d'}$ since $\Gamma_\iota$ is surjective. Hence $\til\iota$ is injective. To prove our claim it is enough to check that $\til\iota$ is closed. Consider $\bar{a},\bar{b}\in A/\mathfrak{c}$ where $\bar{a}=\Gamma_\iota(a),\bar{b}=\Gamma_\iota(b)$. Then
	\begin{align*}
		\til\iota(Z_{\bar{a},\bar{b}}) \ 
		&= \ \{\mathfrak{e}\in \til{Y}\mid \mathfrak{e}= \til\iota(\mathfrak{d})\text{~with~}(\bar{a},\bar{b})\in \mathfrak{d}\text{~for some~} \mathfrak{d}\in \til{X}\}\\
		&= \ \{\mathfrak{e}\in \til {Y}\mid \mathfrak{c}\subseteq \mathfrak{e}\text{~and~}(a,b)\in \mathfrak{e}\} = \ Z_{a,b}\cap (\bigcap_{(c,d)\in \mathfrak{c}}Z_{c,d}).	\qedhere
	\end{align*}
\end{proof}

\begin{rem}
 Note that the definition of closed immersion $\varphi:Y\to X$ in \cite{CHWW15} asks in addition that $\varphi$ is a homeomorphism onto its image. We can omit this condition since it follows from the other properties of a closed immersion. 
 
 Indeed, since the question is local in $X$, we can assume that $X=\MSpec(A)$ and its inverse image $Y=\MSpec(B)$ are affine. By the definition of a closed immersion, $B$ is a quotient of $A$. Thus pulling back prime ideals from $B$ to $A$ is injective. It is a homeomorphism onto its image since $\varphi(U_{[h]})=U_h\cap\im\varphi$ for every $h\in A$ with class $[h]$ in $B$.
\end{rem}

\subsection{Congruence sheaves}

We recall some basic notions around modules over monoids from \cite{Chu-Lorscheid-Santhanam12}. A pointed monoid $A$ acts naturally on $A\times A$ via the diagonal multiplication. A \emph{submodule of $A\times A$} is a nonempty subset $M$ of $A\times A$ such that $(ca,cb)\in M$ for all $c\in A$ and $(a,b)\in M$. Given a morphism $f:A\to B$, the \emph{base extension of $M$ along $f$} is the submodule
\[
 M\otimes_AB \ = \ \big\{ (a,b)\in B\times B~\big|~a=cf(a')\text{ and }b=cf(b') \text{ for some }a',b'\in A,\, c\in B \big\}
\]
of $B\times B$. The \emph{quasi-coherent $\cO_X$-module on $X=\MSpec(A)$ associated with $M$} is the sheaf of monoids $ M^\sim$ with
\[
 M^\sim(U) \ = \ M\otimes_A\Gamma U
\]
for all affine opens $U$ of $X$ with respect to the restriction map $\res_{X,U}\colon A=\Gamma X\to\Gamma U$. 

\begin{defn}
 Let $X$ be a monoid scheme. A \emph{congruence sheaf on $X$} is a subsheaf $\fC$ of $\mathcal{O}_X\times \mathcal{O}_X$ such that $\mathfrak{C}(U)$ is a submodule of $\mathcal{O}_X(U)\times \mathcal{O}_X(U)$ and an equivalence relation for all open subsets $U$ of $X$. A congruence sheaf $\fC$ is \emph{quasi-coherent} if $\fC\vert_U\simeq{\fC(U)}^\sim$ for every affine open subset $U$ of $X$.
 
 Let $\varphi:Y\to X$ be a morphism of monoid schemes. The \emph{congruence kernel of $\varphi$} is the congruence sheaf $\Congker(\varphi)$ on $X$ given by 
 \[
  \Congker(\varphi)(U) \ = \ \congker\big(\varphi^\#(U): \Gamma U\to \Gamma\varphi^{-1}(U)\big)
 \]
 for every open subset $U$ of $X$.
\end{defn}

Note that $\Congker(\varphi)$ inherits the sheaf property from $\cO_X\times\cO_X$ and is thus indeed a sheaf. It is in general not quasi-coherent. It is, however, for closed immersions, as the next result shows, and more generally for quasi-compact morphisms, as proven in Lemma \ref{lemma: the congrunce kernel of a quasi-compact morphism is quasi-coherent}.

\begin{lemma}\label{qcoh-cong}
 Let $\varphi:Y\to X$ be a closed immersion of monoid schemes. Then $\Congker(\varphi)$ is a quasi-coherent congruence sheaf. Conversely, any quasi-coherent congruence sheaf determines a closed immersion of monoid schemes. 
\end{lemma}

\begin{proof}
 Let $\varphi:Y\to X$ be a closed immersion of monoid schemes. The claim of the proposition is local in $X$, so we can assume that $X= \MSpec(A)$ and $Y= \MSpec(A/ \mathfrak{c})$ for some congruence $\mathfrak{c}$ on $A$. We will show that $\Congker(\varphi)=\fc^\sim$. Since $\Congker(\varphi)$ is a sheaf, it is enough to check that $\Congker(\varphi)(U)=\fc^\sim (U)$ for every affine open $U$ of $X$. Since every affine open of $X$ is a principal open, we can assume that $U=\MSpec((A)[s^{-1}])$ for some $s\in A$. We note that by universal property of pullback it follows $ \Gamma\varphi^{-1}(\MSpec((A)[s^{-1}]))\cong\Gamma\MSpec( (A/\fc)[\varphi(s)^{-1}])\Cong(A)/\fc[\varphi(s)^{-1}]$. Since localization of monoids is exact by Lemma \ref{lemma: kernels of localization}, we have
\begin{align*}
	 \Congker(\varphi)(U) \ &= \ \congker\big(\varphi^\#(U): A[s^{-1}]\to A/\fc[\varphi(s)^{-1}]\big)\\
	 &=S^{-1}\big(\congker(A\to A/\fc)\big) \qquad\text{where~} S=\{1,s,s^2,s^3,\hdots\}\\
	 &=S^{-1}\fc =\fc\otimes_A A[s^{-1}]=  \fc^\sim(U).
\end{align*}
Conversely, let $\fC$ be a quasi-coherent congruence sheaf on $X$. For any affine open $U=\MSpec(A)$ of $X$, we know that the surjection $A\to A/\fC(U)$ determines a closed immersion. Since $\fC$ is quasi-coherent, we obtain a closed immersion $\varphi: Y\to X$ with $\varphi|_{\varphi^{-1}(U)}: \MSpec (A/\fC(U))\to U=\MSpec(A)$.
\end{proof}

\subsection{Vanishing sets}
\label{subsection: Vanishing sets}

Let $A$ be a pointed monoid with spectrum $X=\MSpec(A)$ and congruence space $\til X$. Recall that the topology of closed subsets of $\til X$ is generated by the closed subsets $V_{a,b}=\{\fp\in \til X\mid (a,b)\in\fp\}$ with $a,b\in A$. For a subset $S$ of $A\times A$, we define 
\[
 V_S \ = \ \bigcap_{(a,b)\in S} \ V_{a,b}.
\]
Note that not every closed subset of $\til X$ is of the form $V_S$; cf.\ Example \ref{notbasis}. 

\begin{lemma}\label{lemma: vanishing sets of affine monoid schemes}
 Let $A$ be a pointed monoid, $S$ a subset of $A\times A$ and $\fc=\gen S$ the congruence generated by $S$. Then
 \[
  V_S \ = \ V_\fc \ = \ V_{\sqrt{\fc}}.
 \]
\end{lemma}

\begin{proof}
 The inclusions $ V_{\sqrt{\fc}}\subseteq V_\fc \subseteq V_S$ follow from elementary set theoretic considerations. If $\fp$ is a prime congruence containing $S$, then we also have $\fc = \langle S\rangle \subseteq \fp$ and therefore $ V_S \ \subseteq \ V_\fc $. Since $\sqrt{\fc}= \bigcap_{\fc\subseteq \fp}\fp$ it follows that $V_\fc\subseteq V_{\sqrt{\fc}}$.
\end{proof}

\begin{lem}\label{loc-van}
	Let $\iota_S:A\to S^{-1}A$ be a localization of a pointed monoid $A$ at $S$ and let $\iota_S^*: \Cong(S^{-1}A) \to \Cong(A)$ denotes the corresponding morphism of congruence spaces. Then ${\iota_S^*}^{-1}(V_\fc)=V_{\iota_{S,*}(\fc)}$.
\end{lem}

\begin{proof}
	We have ${\iota_S^*}^{-1}(V_\fc)=\{\fd\in \Cong(B)~|~\fc \subseteq \iota_S^*(\fd)\}$. By Proposition \ref{prop: prime congruences in localizations}, we know that $\iota_{S,*}$ and $\iota_S^*$ are mutually inverse bijections. It follows that $\iota_{S,*}(\iota_S^*(\fd))= \fd$ and $~\fc \subseteq \iota_S^*(\fd)$ if and only if $\iota_{S,*}(\fc) \subseteq \fd$.
	Thus we have $\varphi^{-1}(V_\fc)=V_{\iota_{S,*}(\fc)}$.
\end{proof}

\begin{df}
 Let $X$ be a monoid scheme with congruence space $\til X$. A \emph{vanishing set of $X$} is a subset $Z$ of $\til X$ such that for every affine open $U$ of $X$ with congruence space $\til U$, the intersection $Z\cap \til U$ is of the form $V_\fc$ for some congruence $\fc$ on $\Gamma U$.
\end{df}
 
\begin{lemma}\label{lemma: vanishing sets are detectable by open affine covers}
 Let $X$ be a monoid scheme with congruence space $\til X$ and $\cU$ an affine open covering of $X$. Then a subset $Z$ of $\til X$ is a vanishing set if for every $U\in\cU$ with congruence space $\til U$, the intersection $Z\cap \til U$ is of the form $V_\fc$ for some congruence $\fc$ on $\Gamma U$.
\end{lemma}

\begin{proof}
 Consider an affine open $U$ of $X$. Then the unique maximal ideal $M_{\Gamma V}$ of $\Gamma V$ is its unique closed point, which is contained in some $W\in\cU$. Since $W$ is open, it contains the whole subset $U$. Thus $\Gamma U=S^{-1}\Gamma W$ for $S=\{a\in\Gamma W\mid \res_{W,U}(a)\notin M_{\Gamma U}\}$. If $Z\cap \til W=V_\fc$, as required by the hypothesis of the lemma, then $Z\cap \til U=V_{S^{-1}\fc}$ by Proposition \ref{prop: radical congruences in localizations}. This shows that $Z$ is a vanishing set.
\end{proof}

\subsection{Supports of closed subschemes}
\label{subsection: Supports of closed subschemes}

\begin{df}
 Let $X$ be a monoid scheme with congruence space $\til X$ and $\fC$ a quasi-coherent congruence sheaf on $X$. A \emph{vanishing set of $\fC$} is a vanishing set $Z$ of $\til X$ such that $Z\cap \til U=V_{\fC(U)}$ for every affine open subset $U$ of $X$ with congruence space $\til U$.
\end{df}

\begin{thm}\label{thm: characterization of vanishing sets}
 Let $X$ be a monoid scheme with congruence space $\til X$ and $Z$ a subset of $\til X$. Then the following are equivalent:
 \begin{enumerate}
  \item\label{char1} $Z$ is the support of a closed subscheme of $X$.
  \item\label{char2} $Z$ is the vanishing set of a quasi-coherent congruence sheaf on $X$.
  \item\label{char3} $Z$ is a vanishing set.
 \end{enumerate}
\end{thm}

\begin{proof}
 Let $Z$ be the support of a closed subscheme of $X$ represented by a closed immersion $\varphi:Y\to X$, i.e., $Z=\til{\varphi} (\til Y)$. By Lemma \ref{qcoh-cong}, we know that the $\Congker(\varphi)$ is a quasi-coherent congruence sheaf of $X$. We claim that $Z$ is the vanishing set of the quasi-coherent congruence sheaf $\fC:=\Congker(\varphi)$. Let $U=\MSpec(A)$ be an affine open of $X$. Then $\varphi^{-1}(U) = \MSpec (A/\fC(U))$ and $Z\cap \til U=\til {\varphi}( \Cong (A/\fC(U)))= \{\fp \in \Cong(A)~|~\fC(U)\subseteq \fp\} =V_{\fC(U)}$. This proves our claim and shows that $\eqref{char1}\implies \eqref{char2}\implies \eqref{char3}$.
 
 Let $Z$ be a vanishing set. For any affine open subset $U$ of $X$, we consider the strongly reduced congruence $\fc_U$ of $\Gamma U$ with $V_{\fc_U}= Z\cap \til U$. It follows from Lemma \ref{loc-van} that we have a quasi-coherent congruence sheaf $\fC$ given by $\fC(U)= \fc_U$. Therefore by Lemma \ref{qcoh-cong}, we have a closed immersion $\varphi: Y\to X$ with $\varphi|_{\varphi^{-1}(U)}: \MSpec (A/\fc_U)\to U=\MSpec(A)$ for every affine $U$ of $X$. We have $(\varphi|_{\varphi^{-1}(U)})^\congr(\Cong (A/\fc_U)) = V_{\fc_U}= Z\cap \til U$. Thus it follows $\til{\varphi}(\til Y)=Z$. This shows  $\eqref{char3}\implies \eqref{char1}$.
\end{proof}

\begin{cor}
 Every quasi-coherent congruence sheaf has a vanishing set.
\end{cor}
\begin{proof}
 Any quasi-coherent congruence sheaf on $X$ determines a closed immersion $\varphi:Y\to X$ by Lemma \ref{qcoh-cong} and the support of the closed subscheme $[\varphi:Y\to X]$ is a vanishing set by the implication \eqref{char1}$\Rightarrow$\eqref{char2} of Theorem \ref{thm: characterization of vanishing sets}. 
\end{proof}

We denote the vanishing set of a quasi-coherent congruence sheaf $\fC$ on $X$ by $V_\fC$ and we note that for every affine $U$ of $X$, we have $V_\fC \cap \til U = V_{\fC(U)}$.

\begin{thm}\label{thm: equivalence between closed subschemes and congruence scheaves}
 Let $X$ be a monoid scheme. Then there are mutually inverse bijections
 \[
  \begin{tikzcd}[column sep=60]
    \big\{\text{closed subschemes of $X$}\big\} \ar[r,shift left=1,"\Phi"] & 
    \Big\{\begin{array}{c}\text{quasi-coherent}\\ \text{congruence sheaves on $X$}\end{array}\Big\} \ar[l,shift left=1,"\Psi"]
  \end{tikzcd}
 \]
 that are given by $\Phi([\varphi:Y\to X])=\Congker(\varphi)$ and $\Psi(\fC)=[\varphi:Y\to X]$ where $Y=\pi_X(V_\fC)$, embedded via $\varphi$ as a subspace of $X$, and $\cO_Y=(\cO_X/\fC)|_Y$.
\end{thm}

\begin{proof}
 As the first part of the proof, we very that $\Phi$ and $\Psi$ are well-defined. If $[\varphi:Y\to X]$ is a closed subscheme, then we know by Lemma \ref{qcoh-cong} that $\Congker(\phi)$ is a quasi-coherent congruence sheaf. Thus $\Phi$ is well defined. 
 
 Let $\mathfrak{C}$ be any quasi-coherent congruence sheaf on $X$. Then by Lemma \ref{qcoh-cong}, we have a closed immersion $\varphi: Y\longrightarrow X$ with $\varphi|_{\varphi^{-1}(U)}: \MSpec (A/\fC(U))\to U=\MSpec(A)$. As a subspace of $X$, we indeed have $Y = \pi_X(V_\mathfrak{C})$ since for any affine open $U=\MSpec(A)$ of $X$ we have 
 \begin{multline*}
  \pi_X(V_\mathfrak{C})\cap U \ = \ \pi_X(V_\mathfrak{C}\cap\widetilde{U}) \ = \ \pi_X(V_{\mathfrak{C}(U)}) \\
  = \ \pi_X(\Cong(A/\mathfrak{C}(U))) \ = \ \MSpec (A/\mathfrak{C}(U)).  
 \end{multline*}
 In addition, $(\cO_X/\fC)|_U$ is the quasi-coherent sheaf of monoids $({ A/\mathfrak{C}(U)})^\sim$ associated with the $A$-module $A/\mathfrak{C}(U)$.  Therefore 
 \[
  (\cO_X/\fC)|_{\MSpec(A/\mathfrak{C}(U))}=({ A/\mathfrak{C}(U)\otimes_A A/\mathfrak{C}(U)})^\sim= ({ A/\mathfrak{C}(U)})^\sim = \cO_{\MSpec(A/\mathfrak{C}(U))}
 \]
 where $A/\mathfrak{C}(U)$ is viewed as an $A/\mathfrak{C}(U)$-module. Since an affine cover of $Y$ is given by $(\MSpec(A/\mathfrak{C}(U)), \cO_{\MSpec(A/\mathfrak{C}(U))})$, it follows that $\cO_Y=(\cO_X/\fC)|_Y$. Therefore $\Psi$ is also well defined.
 
 We turn to the proof that $\Phi$ and $\Psi$ are mutually inverse bijections. Let $\mathfrak{C}$ be a quasi-coherent congruence sheaf on $X$. For any affine open $U=\MSpec(A)$ of $X$, we have $\Phi(\Psi(\mathfrak{C}))|_{U}= \Congker(\varphi)|_{U}$ where $[\varphi:\pi_X(V_\mathfrak{C})\longrightarrow X ]$ denotes the closed subscheme associated with $\mathfrak{C}$. Since 
 \[
  \Congker(\varphi)|_{U} \ = \ {\congker (\varphi^\#(U))}^\sim \ = \ {\mathfrak{C}(U)}^\sim \ = \ \mathfrak{C}|_U, 
 \]
 it follows that $\Phi(\Psi(\mathfrak{C})) = \mathfrak{C}$. 
 
 Conversely, let $[\varphi:Y\to X]$  be a closed subscheme of $X$. By Theorem \ref{thm: characterization of vanishing sets}$(1)$, we know that $\widetilde{\varphi}(\widetilde{Y})=V_{\Congker(\varphi)}$. Therefore $\pi_X(V_{\Congker(\varphi)})= \pi_X(\widetilde{\varphi}(\widetilde{Y}))= \varphi(\pi_Y(\widetilde{Y}))=\varphi(Y)$. Thus $\Psi(\Phi([\varphi: Y\longrightarrow X]))= [\varphi:Y\longrightarrow X]$, which completes the proof.  
\end{proof}

\begin{rem}\label{vanish-c.i-qcoh}
 The results of this section can be summarized in the following diagram:
 \[
  \begin{tikzcd}[column sep=60]
   \big\{\text{vanishing sets of $X$}\big\} \ar[<->,r,"\text{bijective}"] \ar[>->,d,shift right=2] & \Big\{\begin{array}{c}\text{strongly reduced}\\ \text{closed subschemes of $X$}\end{array}\Big\} \ar[>->,d,shift right=2] \\
   \Big\{\begin{array}{c}\text{quasi-coherent} \\ \text{congruence sheaves on $X$}\end{array}\Big\} \ar[->>,u,shift right=2]\ar[<->,r,"\text{bijective}"] & \big\{\text{closed subschemes of $X$}\big\} \ar[->>,u,shift right=2]
  \end{tikzcd}
 \]
 The upper horizontal arrow identifies a vanishing set with the unique strongly reduced closed subscheme supported on it and the lower horizontal arrow stands for the bijection from Theorem \ref{thm: equivalence between closed subschemes and congruence scheaves}. The downwards arrow on the left hand side sends a vanishing set to the corresponding reduced quasi-coherent congruence sheaf, the upwards arrow sends a quasi-coherent congruence sheaf to its vanishing set. The downwards arrow on the right hand side is the natural inclusion, the upwards arrow sends a closed subscheme to its strong reduction. 
 
 All of these arrows are functorial in, respectively, inclusions of subspaces, surjections of sheaves and closed immersions of closed subschemes. The horizontal arrows become equivalences of categories and the vertical arrows are mutual adjoints (one is an embedding, the other a retract of categories). The horizontal arrows commute with the downwards arrows and with the upwards arrows.
\end{rem}

\subsection{Quasi-compact morphisms}
\label{subsection: quasi-compact morphisms}

\begin{defn}
 A morphism $\phi:Y\to X$ of monoid schemes is \emph{quasi-compact} if $\phi^{-1}(U)$ is quasi-compact for all affine open subsets $U$ of $X$.
\end{defn}

Note that an affine monoid scheme $Y=\MSpec(A)$ is always quasi-compact since it has a unique closed point. Therefore affine morphisms are quasi-compact. Since finite unions of affine opens are also quasi-compact, morphisms of finite type are quasi-compact. In general, quasi-compactness of morphisms can be verified on a chosen open affine covering $\cU$ of $X$. That is, $\varphi:Y\to X$ is quasi-compact if $\phi_i^{-1}(U)$ is quasi-compact for all $U\in\cU$.

\subsection{Vanishing closure}
\label{subsection: Vanishing closure}

\begin{defn}
 Let $X$ be a monoid scheme with congruence space $\til X$. Let $Z$ be a subset of $\til X$. The \emph{vanishing closure of $Z$} is the vanishing set
 \[
  Z^\vcl \ = \ \bigcap_{\substack{Z\subset V\\ \text{$V$ vanishing set}}} V.
 \]
\end{defn}

Note that $Z^\vcl$ is indeed a vanishing set due to its affine nature and the relation $V_\fc\cap V_{\fc'}=V_{\gen{\fc\cup\fc'}}$ given by Lemma \ref{lemma: vanishing sets of affine monoid schemes}.

\begin{lemma}\label{lemma: the congrunce kernel of a quasi-compact morphism is quasi-coherent}
 The congruence kernel of a quasi-compact morphism of monoid schemes $\varphi:Y\to X$ is quasi-coherent.
\end{lemma}

\begin{proof}
 Since the question is local in $X$, we can assume that $X=\MSpec(B)$ is affine. Consider a principal affine open $U=\MSpec(B[h^{-1}]) \subseteq X$ and let $V$ denote the open subset $\varphi^{-1}(U)= U\times_X Y$ of $Y$. Let 
 \begin{equation*}
  \mathfrak{c}= \congker(\varphi^\#(X):B\longrightarrow \Gamma Y)\quad \text{and} \quad \mathfrak{c}_h= \congker(\varphi^\#(U):B[h^{-1}]\longrightarrow \Gamma V).
 \end{equation*}
 We need to show $\mathfrak{c}_h = S^{-1}\mathfrak{c}$ for $S=\{h^i\}_{i\geq 0}$. Since $\varphi$ is quasi-compact, we have $Y=\bigcup_{i\in I} W_i$ for affine open $W_i= \MSpec(C_i)$ where $I$ is a finite set. Therefore the corresponding morphism of global sections $\Gamma Y\to \prod_i C_i$ is injective. If $\bar{h}=\varphi^\#(X)(h)\in \Gamma Y$ and $\bar{h_i}= \res_{Y,W_i}(\bar{h})\in C_i$ then $C_i\otimes_B B[h^{-1}]=C[\bar{h_i}^{-1}]$. Thus $W_i\times_X U= \MSpec (C_i[\bar{h_i}^{-1}])$ and we have the following commutative diagram
 \begin{equation*}
	\xymatrix@!0@R=1.5cm@C=3cm{
		B \ar@<-2pt>[d]_{\lambda_h} \ar[r]^{\varphi^\#(X)}& \Gamma Y\ar[r]^{\zeta} \ar@<-2pt>[d]_{\res_{Y,V}} &  \prod_{i\in I} C_i\ar@<-2pt>[d]_{\prod\lambda_{\bar{h_i}}}\ar@/^3.0pc/@[black][dd]^{\lambda_{\bar{h}}}\\
		B[h^{-1}]\ar[r]_{\varphi^\#(U)} &	\Gamma V\ar[r]_{\eta} & \prod_{i\in I} C_i[\bar{h_i}^{-1}]\ar[d]_{\iota}\\
		& & (\prod_{i\in I} C_i)[\bar{h}^{-1}]
	}
 \end{equation*}
 where $\iota$ is an isomorphism since $I$ is finite. Both $\eta$ and $\zeta$ are injections. Thus if $\gamma=\zeta\circ \varphi^\#(X)$ and $\delta=\iota \circ\eta\circ \varphi^\#(U)$, then we have $\mathfrak{c}=\congker(\gamma)$ and $\mathfrak{c}_h= \congker(\delta)$. 
 
 For to show that $S^{-1}\mathfrak{c}= \mathfrak{c}_h$, it is enough to check that any element of the form $(\frac{b}{1}, \frac{b'}{1})$  in $S^{-1}\mathfrak{c}$ also belongs to $\mathfrak{c}_h$ and vice versa.  If $(\frac{b}{1}, \frac{b'}{1})\in S^{-1}\mathfrak{c}$, then there exists $t\in S$ such that $(bt, b't)\in \mathfrak{c}$ and therefore we have $(\lambda_{\bar{h}}\circ\gamma)(bt)=(\lambda_{\bar{h}}\circ\gamma)(b't)$. It follows from the commutativity of the above diagram that $\delta(\frac{bt}{1})= \delta(\frac{b't}{1})$. This implies that $(\frac{bt}{1}, \frac{b't}{1})\in \mathfrak{c}_h$ or, equivalently, $(\frac{b}{1}, \frac{b'}{1})\in \mathfrak{c}_h$. 
 
 Conversely, consider any $(\frac{b}{1}, \frac{b'}{1})$  in $ \mathfrak{c}_h$. Then we have $(\delta\circ \lambda_h)(b) = (\delta\circ \lambda_h)(b')$.  It follows from the commutativity of the above diagram that $\frac{\gamma(b)}{1} = \frac{\gamma(b')}{1}$, i.e., there exists some $i\in \mathbb{Z}$ such that $\gamma(b)\bar{h}^i = \gamma(b')\bar{h}^i$. In other words, we have $\gamma(bh^i) = \gamma(b'h^i)$, i.e., $(bh^i, b'h^i)\in \mathfrak{c}$. Thus $(\frac{b}{1}, \frac{b'}{1})$  in $S^{-1}\mathfrak{c}$.
\end{proof}

\begin{lemma}\label{lemma: vanishing set of the congruence kernel is the vanishing closure of the image}
 Let $\varphi:Y\to X$ be a quasi-compact morphism of monoid schemes and $\tilde\varphi:\til Y\to\til X$ the associated map of congruence spaces. Then
 \[
  V_{\Congker(\varphi)} \ = \ \im(\tilde\varphi)^\vcl.
 \]
\end{lemma}

\begin{proof}
 Since $\varphi$ is quasi-compact, the congruence kernel $\Congker(\varphi)$ defines a quasi-coherent congruence sheaf by \autoref{lemma: the congrunce kernel of a quasi-compact morphism is quasi-coherent}, which means that it defines a closed subscheme $Z$ of $X$ with vanishing set $V_{\Congker(\varphi)}$. 
 
 The closed subscheme $Z$ is, by definition of the congruence kernel and its quasi-coherence, the smallest closed subscheme of $X$ such that $\til Z$ contains the image of $\tilde\varphi$. Since every vanishing set is the support of the congruence space of some closed subscheme of $X$, this establishes the claim of the lemma.
\end{proof}

\subsection{A topological characterization of closed immersions}
\label{subsection: A topological characterization of closed immersions}

\begin{thm}\label{thm: topological characterization of closed immersion}
 Let $\varphi:Y\to X$ be a morphism of monoid scheme and $\tilde\varphi:\til Y\to \til X$ the associated map of congruence spaces. Then the following are equivalent.
 \begin{enumerate}
  \item\label{closed1} $\varphi:Y\to X$ is a closed immersion.
  \item\label{closed2} $\varphi:Y\to X$ is a quasi-compact topological embedding, $\varphi^\#:\cO_X\to\varphi_\ast\cO_Y$ is surjective and the image of $\tilde\varphi$ is a vanishing set.
 \end{enumerate}
\end{thm}

\begin{proof}
 By Proposition \ref{prop: closed immersions of congruence spaces}, \eqref{closed1} implies \eqref{closed2}. To prove the converse implication, assume \eqref{closed2} and let $\fC$ be the congruence kernel of $\varphi$. By Lemma \ref{lemma: vanishing set of the congruence kernel is the vanishing closure of the image}, the image of $\tilde\varphi$ is equal to the vanishing set $V_\fC$ of $\fC$. Thus $\varphi$ factors through the closed immersion $\psi:Z\to X$ corresponding to $\fC$. After substituting $X$ by $Z$, we can thus assume that $\varphi^\#:\cO_X\to\varphi_\ast\cO_Y$ is an isomorphism of sheaves and that $\tilde\varphi$ is a homeomorphism. 
 
 Since the problem is local in $X$, we can assume that $X$ is affine, i.e.\ $X=\MSpec(\Gamma X)\simeq\MSpec(\Gamma Y)$. This reduces us to the situation that $\varphi$ is the canonical morphism $\varphi:Y\to \MSpec(\Gamma Y)=X$. Under these assumptions, it suffices to show that $Y$ is affine, since this implies that $Y=\MSpec(\Gamma Y)=X$ and that $\varphi:Y\to X$ is the identity.
 
 Let $M$ be the maximal prime ideal of $B=\Gamma X$. Then its fibre $\pi_X^{-1}(M)$ in $\til X$ has a unique closed point $\fm$, which is the congruence kernel of the composition $B\to M/B\to \Fun$ of the quotient map $B\to B/M$ with the unique map $B/M\to\Fun$ that maps all units of the pointed group $B/M$ to $1$. Let $Z=\varphi^{-1}(\fm)$ and $\fn\in \til Y$ be the unique point that maps to $\fm$. Since $\pi_X\circ\tilde\varphi=\varphi\circ\pi_Y$ and since $\tilde\varphi$ is a homeomorphism, $\fn$ is the unique closed point of $\pi_Y^{-1}(Z)$. Let $N=\pi_Y(\fn)$.

 Let $y$ be any closed point of $Y$. Since $\varphi$ is a topological embedding, $\varphi(y)=M$ and thus $y\in Z$. Since $y$ is closed, $\pi_Y^{-1}(y)$ is closed in $\til Y$ and thus contains $\fn$. This shows that $y=N$. We conclude that $Y$ has a unique closed point and is therefore affine, which concludes the proof.
\end{proof}

\subsection{Dominant morphisms}
\label{subsection: Dominant morphisms}

\begin{defn}
 Let $X$ be a monoid scheme and $\til X$ its congruence space. A subset $Z$ of $\til X$ is \emph{strictly dense} if its vanishing closure is $\til X$. A morphism $\varphi:Y\to X$ of monoid schemes is \emph{dominant} if the associated map $\tilde\varphi:\til Y\to \til X$ of congruence spaces has a strictly dense image.
\end{defn}

\begin{lem}
If $\phi: Y\to X$ is a dominant morphism of monoid schemes, then $\phi(Y)$ is dense in $X$.
\end{lem}

\begin{proof}
Since the morphisms $\pi_Y: \widetilde{Y} \to Y$ are functorial and $\pi_Y\circ \sigma_Y= \id_Y$, the diagram 
\[ \begin{tikzcd}
	& \widetilde{Y}\arrow{r}{{\tilde\phi}} & \widetilde{X}\arrow{d}{\pi_X}\\
	& Y \arrow{u}{\sigma_Y} \arrow{r}{\phi} & X
\end{tikzcd}
\]
commutes. By definition, $\sigma_Y(p)$ is minimal in $\pi_Y^{-1}(p)$ for all $p\in Y$, i.e., $\pi_Y^{-1}(p)\subseteq \overline{\{\sigma_Y(p)\}}$. We consider any $\mathfrak{c}\in {\tilde\phi}(\widetilde{Y})$ and let $ \mathfrak{c}={\tilde\phi}(\mathfrak{d})$. Let $V$ be any open in $\widetilde{X}$ such that $\mathfrak{c}\in V$. If $q= \pi_Y(\mathfrak{d})$, then we have $	\mathfrak{d}\in \pi_Y^{-1}(q)\subseteq \overline{\{\sigma_Y(q)\}}$ which implies $\sigma_Y(q)\in {\tilde\phi}^{-1}(V)$. Thus ${\tilde\phi}(\sigma_Y(q))\in V$ and so $V\cap {\tilde\phi}(\sigma_Y(Y))\neq \emptyset$. Therefore it follows that $\overline{{\tilde\phi}(\widetilde{Y})}=\overline{{\tilde\phi}\circ\sigma_Y(Y)}$. Hence $\widetilde{X}=\overline{{\tilde\phi}(\widetilde{Y})}=\overline{{\tilde\phi}\circ\sigma_Y(Y)} \subseteq \pi_X^{-1}(\overline{\phi(Y)})$. This shows that $\overline{\phi(Y)}=X$.
\end{proof}

\begin{rem}
 Note that $\varphi:Y\to X$ does not need to be dominant if $\varphi(Y)$ is dense in $X$. For example, the diagonal $\Delta:\A^1_\Fun\to\A^2_\Fun$ has a dense image, but is not dominant. 
 
 On the other hand, a strictly dense image of the associated map $\tilde\varphi:\til Y\to\til X$ of congruence spaces does not imply that the image of $\tilde\varphi$ is dense, as attested by the morphism $\varphi:\MSpec(\Fun)\amalg\MSpec(\Fun)\to\A^1_\Fun$ that stems from the two different maps $\Fun[t]\to\Fun$, one which sends $t$ to $0$, the other which sends $t$ to $1$. The image of $\tilde\varphi$ consists of the two prime congruences $\gen{(t,0)}$ and $\gen{(t,1)}$ on $\A^1_\Fun$, which form a proper closed subset of $\Cong(\A^1_\Fun)$. Thus the image is \textit{not} dense, but its vanishing closure is all of $\Cong(\A^1_\Fun)$. 
\end{rem}

\begin{lem}
 Let $\varphi:Y\to X$ be a dominant morphism of monoid schemes and $X$ be strongly reduced. Then $\varphi^\#:\Gamma X\to\Gamma Y$ is injective.
\end{lem}

\begin{proof}
 Since the image of the associated map $\tilde\varphi:\til Y\to\til X$ has a strictly dense image and since $X$ is strictly reduced, it follows from Lemma \ref{lemma: vanishing set of the congruence kernel is the vanishing closure of the image} and Proposition \ref{sred-c.i} that the congruence kernel of $\varphi$ is trivial. Thus $\varphi^\#:\Gamma X\to\Gamma Y$ is injective.
\end{proof}

\section{Valuation monoids}
\label{section: Valuation monoids}

The notion of a valuation monoid was first introduced in \cite{CHWW15}. 

\begin{defn}
 A \emph{valuation monoid} is an integral monoid $A$ such that for every non-zero element $\alpha\in \Frac A$, either $\alpha\in A$ or $\alpha^{-1}\in A$. 
\end{defn}

For our study, we require an existence result of valuation monoids that dominate a given integral monoid. All of this notions are introduced in the following.

\subsection{Dominant maps of monoids}
\label{subsection: Dominant maps of monoids}

\begin{defn}
 A morphism $f: A\to B$ of pointed monoids is \emph{dominant}, and $B$ \emph{dominates $A$}, if $f^{-1}(M_B) = M_A$ where $M_A$ is the maximal ideal of $A$ and $M_B$ is the maximal ideal of $B$.
\end{defn}

Note that for any morphism $f: A\to B$ of pointed monoids, it is always true that $f(A^\times)\subseteq B^\times$ and so $f^{-1}(M_B)\subseteq M_A$. Thus the morphism $f$ is dominant if and only if $f^{-1}(B^\times)\subseteq A^\times$. Also note that composition of dominant morphisms is dominant. 

\begin{defn}
 Let $B$ be a pointed monoid and let $A$ be an integral submonoid of $B$.  An element $b$ of $B$ is \emph{algebraic over $A$} if there is an $a\in A$ and an integer $n\geq 1$ such that $ab^n\in A$. Elements of $B$ that are not algebraic over $A$ are called \emph{transcendental over $A$}.
 
 An \emph{integral extension of monoids} is an injective morphism $f:A\to B$ of pointed monoids such that for any $b\in B$ there is an integer $ n\geq 1$ such that $b^n\in A$.
\end{defn}

\begin{lem}\label{intextdom}
 An integral extension $f:A\to B$ is dominant.
\end{lem}

\begin{proof}
 We want to show $B^\times \cap A\subseteq A^\times$ (we identify $A$ with $f(A)\subseteq B$). We consider $a\in B^\times \cap A$, i.e., $ab=1$ for some $b\in B$. Since $A\hookrightarrow B$ is an integral extension, there exists some integer $n\geq 1$ such that $c:=b^n \in A$. Thus $a(a^{n-1}c)= a(a^{n-1}b^n)= (ab)^n =1 $ which shows $a\in A^\times$.
\end{proof}

\subsection{Characterization and existence of valuation monoids}
\label{subsection: Characterization and existence of valuation monoids}

\begin{prop}\label{val-dom}
 An integral monoid $A$ is a valuation monoid if and only if $A$ is a maximal submonoid of $\Frac A$ with respect to domination.
\end{prop}

\begin{proof}
	Suppose $A$ is maximal in $\Frac A$ with respect to domination and consider $a\in \Frac A\setminus A$. We define $B:=A[a]=\{ca^i\in \Frac A\mid c\in A, i\geq 0\}$. Since $A$ is maximal with respect to domination and $A\neq B$, the inclusion $A\hookrightarrow B$ is not a dominant morphism, i.e., $M_A\nsubseteq M_B\cap A$. Therefore there exists some $d \in M_A\setminus M_B = M_A\cap B^\times$. Let $d= ca^i$ for some $c\in A$ and some integer $i\geq0$. If $i\geq 1$, we have $a\in B^\times$. If $i=0$, then $c=d\in M_A\cap B^\times$ and so there exists $c^{-1} = c'a^j \in B^\times$ for some $c'\in A$ and integer $j\geq 1$ (since $c\notin A^\times$). This shows that $a\in B^\times$ also in the case $i=0$. Let $a^{-1}= \tilde{c}a^k$ for some $\tilde{c}\in A$ and integer $k\geq 0$. Thus $(a^{-1})^{k+1}=\tilde{c}\in A$ and $A[a^{-1}]= \{ca^{-i}\in \Frac A\mid c\in A, i\geq 0\}$ is an integral extension of $A$. Therefore by Lemma \ref{intextdom}, the inclusion $A\hookrightarrow A[a^{-1}]$ is a dominant morphism. Since $A$ is maximal with respect to domination, we must have $A\Cong(A)[a^{-1}]$ which implies $a\in A^\times$ as desired.
	
	Conversely, suppose $A$ is a valuation monoid. Let $B$ be a pointed submonoid of $\Frac A$ and $f: A\to B$ is a dominant morphism such that the canonical injection $A\hookrightarrow \Frac A$ factorizes through $f$. We want to show that $f$ is an isomorphism. We consider any $b\in B\setminus \{0\}\subseteq \Frac A$. Since $A$ is a valuation monoid, either $b\in A$ or  $b^{-1}\in A$ (we identify $A$ with $f(A)\subseteq B$). If $b^{-1}\in A$, we have $b^{-1}\in B^\times \cap A= A^\times$ since $B$ dominates $A$. Thus $b= (b^{-1})^{-1}\in A$.
\end{proof}
 
\begin{prop}\label{val-dom2}
 Let $K$ be a pointed group and let $A\subseteq K$ be a pointed submonoid. Then there exists a valuation monoid $A$ dominating $A$ such that $K$ is the pointed group completion of $A$. 
\end{prop}

\begin{proof}
	We consider the collection $\mathscr{S}$ of all pointed submonoids of $K$ dominating $A$ as a partially ordered set using the relation of domination, i.e., for any $B,C\in \mathscr{S}$, we say $B\leq C$ if and only if $B\subseteq C$ and $M_B=M_C\cap B$. Suppose that $\{A_i\}_{i\in I}$ is a totally ordered chain in $\mathscr{S}$. Then $D=\bigcup A_i$ is clearly a pointed submonoid of $K$ which dominates all of the $A_i$ and hence $D$ dominates $A$ as well. Therefore by Zorn's Lemma, $\mathscr{S}$ has a maximal element, say $\overline{A}$. Since $\overline{A}$ is a maximal element in $\mathscr{S}$, by Proposition \ref{val-dom}, $A_\max$ is a valuation monoid. We are left to check that $K$ is the pointed group completion of $\overline{A}$. Suppose, if possible, $\Frac \overline{A}\neq K$. We consider $t\in K$ such that $t\notin \Frac \overline{A}$. If $t$ is algebraic over $\overline{A}$, there exists an integer $n\geq 1 $ such that $t^n\in \overline{A}$. Thus $\overline{A}\to \overline{A}[t]$ is an integral extension and therefore it is dominant by Lemma \ref{intextdom}. This contradicts that $\overline{A}$ is a maximal element in $\mathscr{S}$. On the other hand, if $t$ is transcendental over $\overline{A}$, it follows that $C:=\overline{A}[t]=\{kt^i\in K\mid k\in \overline{A}, i\geq 0\}\subseteq K$ is a pointed submonoid of $K$ with $M_C=\langle M_{\overline{A}}, t\rangle$ and that $C$ dominates $\overline{A}$, i.e., $M_C\cap \overline{A}= M_{\overline{A}}$. It follows that $C$ dominates $A$. Hence $C$ is a pointed submonoid of $K$ dominating $A$ and strictly containing $\overline{A}$. This gives the required contradiction.
\end{proof}

\section{Universally closed morphisms}
\label{section: Universally closed morphisms}

In this section, we introduce closed morphisms and universally closed morphisms. Under the assumption of quasi-compactness, we characterize universally closed morphisms in terms of a valuative criterion.

\subsection{Closed images}
\label{subsection: Closed morphisms}

\begin{prop}\label{specl}
 Let $\phi:Y\to X$ be a quasi-compact morphism of monoid schemes and $\til\phi:\til Y\to \til X$ be the associated map between congruence spaces. Then ${\tilde\phi}(\widetilde{Y})$ is closed in $\widetilde{X}$ if and only if ${\tilde\phi}(\widetilde{Y})$ is closed under specialization.
\end{prop}

\begin{proof}
 Since every closed set is closed under specialization, we are left with the converse implication. Assume that ${\tilde\phi}(\widetilde{Y})$ is closed under specialization. Without loss of generality, we can assume that $X=\MSpec(B)$ is affine and that both $Y$ and $X$ are strongly reduced. Consider $\mathfrak{d}\in \overline{{\tilde\phi}(\widetilde{Y})}$. We want to show  $\mathfrak{d}\in {\tilde\phi}(\widetilde{Y})$. Since $\phi$ is quasi-compact, $Y=\bigcup_{i=1}^{n}  Y_i$ for some affine opens $Y_i$. Thus $\overline{{\tilde\phi}(\widetilde{Y})}= \bigcup_{i=1}^{n} \overline{{\tilde\phi}(\widetilde{Y_i})}$ and $\mathfrak{d}\in  \overline{{\tilde\phi}(\widetilde{Y_i})}$ for some $i$. Hence we can assume that $Y=Y_i=\MSpec(A)$ is affine. Let $f=\Gamma_\phi: B\to A$ and let $\mathfrak{c}_0= \congker(f)$. The surjective morphism $p:B\to B/\mathfrak{c}_0$ induces a closed immersion $p^*: \MSpec(B/\mathfrak{c}_0)\hookrightarrow \MSpec(B)$. By Proposition \ref{imm-emb}, we know that $\widetilde{p^*}$ is a closed topological embedding. Since the diagram 
	\begin{equation*}
	\xymatrix{
	\Cong (A) \ar[r]^{{\tilde\phi} }\ar@<-2pt>[d]_{} & \Cong(B) \\		\Cong(B/\mathfrak{c}_0)\ar[ur]_{\widetilde{p^*}} & 
	}
\end{equation*}
 commutes, we can assume that $p:B\xrightarrow{\cong}B/\mathfrak{c}_0$ is an isomorphism and $f:=\Gamma_\phi: B\to A$ is an injection. Let $\mathfrak{d'}\subseteq \mathfrak{d}$ be a minimal prime congruence of $B$ (which exists by a standard application of Zorn's lemma). Let $S_{\mathfrak{d'}}= B\setminus \mathfrak{d}'$. By Lemma \ref{lemma: kernels of localization},
 $$f_{\mathfrak{d}'}: \ B_{\mathfrak{d}'} \ = \ S_{\mathfrak{d}'}^{-1}B \ \longrightarrow \ f(S_{\mathfrak{d}'})^{-1}A \ = \ A_{\mathfrak{d}'}$$ 
 is an injection. Since $\mathfrak{d}'$ is minimal, $B_{\mathfrak{d}'}$ is a pointed group and $\langle \mathfrak{d}'\rangle_{B_{\mathfrak{d}'}}$ is trivial. Thus the natural composition 
 $$\bar{f_{\mathfrak{d}'}}: B_{\mathfrak{d}'}\overset{f_{\mathfrak{d}'}}\longrightarrow A_{\mathfrak{d}'}\overset{q}\longrightarrow A_{\mathfrak{d}'}/M_{A_{\mathfrak{d}'}}= \{0\}\cup A_{\mathfrak{d}'}^\times$$
 is injective and 
 \begin{equation*}
 	 \langle \mathfrak{d}'\rangle_{B_{\mathfrak{d}'}}=\widetilde{{\bar{f_{\mathfrak{d}'}}}^*}(\mathfrak{c}_\triv)=\widetilde{f_{\mathfrak{d}'}^*}(\widetilde{q^*}(\mathfrak{c}_\triv)).
 \end{equation*}
Let $g: A\to A_{\mathfrak{d}'}$ and $h: B\to B_{\mathfrak{d}'}$ denote the canonical localization maps. This and the commutative diagram 
	\begin{equation*}
	\xymatrix{
	A\ar[r]^{g} & A_{\mathfrak{d}'}\ar[r]^{q} & A_{\mathfrak{d}'}/M_{\mathfrak{d}'}\\
		B\ar@<-2pt>[u]^{f} \ar[r]_{h} & B_{\mathfrak{d}'}\ar[u]_{f_{\mathfrak{d}'}}  &
	}
\end{equation*}
imply that $\mathfrak{d}'=\widetilde{h^*}( \langle \mathfrak{d}'\rangle_{B_{\mathfrak{d}'}})=\widetilde{h^*}(\widetilde{f_{\mathfrak{d}'}^*}(\widetilde{q^*}(\mathfrak{c}_\triv)))=\widetilde{f^*}(\widetilde{g^*}(\widetilde{q^*}(\mathfrak{c}_\triv))$ and thus $\mathfrak{d}'\in {\tilde\phi}(\widetilde{Y})$. Since $ {\tilde\phi}(\widetilde{X})$ is closed under specialization, we have $\mathfrak{d}\in  {\tilde\phi}(\widetilde{Y})$.
\end{proof}

\subsection{The valuative criterion for universally closed morphisms}
\label{subsection: The valuative criterion for universally closed morphisms}

\begin{defn}
 A morphism $\varphi:Y\to X$ is \emph{closed} if the associated map $\tilde\varphi:\til Y\to\til X$ between the respective congruence spaces is a closed continuous map. A morphism $\varphi:Y\to X$ of monoid schemes is \emph{universally closed} if for every morphism $\psi:X'\to X$, the base change $\varphi':Y'=X'\times_XY\to X'$ of $\varphi$ along $\psi$ is a closed morphism.
\end{defn}

As a first example, we have the following. 

\begin{lemma}\label{lemma: closes immersions are universally closed}
 Closed immersions are universally closed.
\end{lemma}

\begin{proof}
 Since surjective monoid morphisms are stable under cobase change, it follows that closed immersions are stable under base change. Therefore the result follows from Proposition \ref{prop: closed immersions of congruence spaces}.
\end{proof}

\begin{defn}
 Let $\phi:Y\to X$ be a morphism of monoid schemes. A \emph{test diagram for $\varphi$} is a commutative diagram
 \begin{equation*}
  \begin{tikzcd}[column sep=60]
   U=\MSpec(K)\ar[r,"\eta"] \ar[d,"\iota"'] & Y\ar[d,"\phi"] \\
   T=\MSpec(A)  \ar[r,"\nu"] & X
  \end{tikzcd}
 \end{equation*}
 where $A$ is a valuation monoid, $K=\Frac A$ and $\iota=i^*$ is the inclusion $i:A\hookrightarrow K$. A \emph{lift for the test diagram} is a morphism $\hat{\nu}:T\to Y$ such that 
 \begin{equation*}
  \begin{tikzcd}[column sep=60]
   U\ar[r,"\eta"] \ar[d,"\iota"'] & Y\ar[d,"\phi"] \\
   T\ar[r,"\nu"] \ar[ur,dashed,"\hat\nu"] & X
  \end{tikzcd}
 \end{equation*}
 commutes.
\end{defn}

\begin{rem}\label{aff-pointed-grp}
 The prime congruences $\mathfrak{c}$ of a pointed group $K$ correspond bijectively to the subgroups $H_\mathfrak{c}:=\{a\in K\mid (a,1)\in \mathfrak{c}\}$ of $K^\times$. Given a morphism of pointed groups $f: K_1\to K_2$, we have an induced morphism $\widetilde{f}_*: {\Cong(K_1)}\to {\Cong(K_2)}$ that sends a prime congruence $ \mathfrak{c}$ on $K_1$ to its pushforward
 \begin{equation*}\label{ext}
  \widetilde f_\ast(\fc) \ = \ \big\{ \big(cf(a),cf(b)\big)\in K_1\times K_1 \,\big|\, c\in K_2, (a,b)\in\fc \big\},
 \end{equation*}
 which is the prime congruence associated with $f(H_\mathfrak{c})= H_{\widetilde{f}_*(\mathfrak{c})}$. 
\end{rem}

In the proof of the next result, we use the following shorthand notation. Given a monoid scheme $X$ with congruence space $\til X$ and $\fp,\fq\in\til X$ such that $\fq$ is a specialization of $\fp$, we denote by $\cO_{\til X,\fq}/\fp$ the quotient of the stalk $\cO_{\til X,\fq}$ by the congruence generated by $\fp$ where we interpret $\fp$ as a congruence on a monoid $A=\Gamma U$ for an affine open neighbourhood $U$ of $\fq$ in $X$. In other words, $\cO_{\til X,\fq}/\fp=S^{-1}_\fq A/S^{-1}_\fq\fp$ for $S_\fq=A\setminus \fq$.

\begin{thm}\label{uni-closed}
 Let $\phi:X\to Y$ be a quasi-compact morphism of monoid schemes. Then $\phi:Y\to X$ is universally closed if and only if every test diagram 
 \[
  \begin{tikzcd}[column sep=40, row sep=15]
   U\ar[r,"\eta"] \ar[d,"\iota"'] & Y\ar[d,"\phi"] \\
   T\ar[r,"\nu"] & X
  \end{tikzcd}
 \]
 has a lift $\hat\nu:T\to Y$.
\end{thm}
\begin{proof}
 Assume $\phi:Y\to X$ is universally closed and consider a test diagram as in the theorem with $A=\Gamma T$ and $K=\Frac A$. This diagram extends to
	\begin{equation*}\label{testdiag1}
	\xymatrix{
		U\ar[r]^{\lambda}\ar[dr]_{\iota}\ar@/^2.0pc/@[black][rr]^{\eta}& T\times_X Y\ar[r]^{~~~\nu'} \ar@<-2pt>[d]_{\phi'} & Y\ar@<-2pt>[d]^{\phi} \\
		&	T\ar[r]_{\nu} & X
	}
\end{equation*}
Let $u_0=\langle0\rangle\in U$ be the generic point of $U$, $y_0=\lambda(u_0)\in Y_T= T\times_X Y$ and $t_0=\iota(u_0)=\phi'(y_0)\in T$ the generic point of $T$. Then we obtain the following commutative diagram of residue fields: 
	\begin{equation*}
	\xymatrix{
		K=k(t_0)\ar[r]^{\overline{\phi'_{T,t_0}}} \ar@<-2pt>[dr]_{id} & L=k(y_0)\ar@<-2pt>[d]^{p=\overline{\lambda_{Y_T,y_0}}} \\
	 & K=k(u_0)
	}
\end{equation*}
Thus $p$ is surjective and $K\cong L/\mathfrak{c}_0$ for $\mathfrak{c}_0=\congker(L\xrightarrow{p} K)$. Consider
\[
 \begin{tikzcd}[column sep=40pt]
   \psi: \quad {\Cong(L)} \ar[r,"\overline{\phi'_{T,t_0}}"] & {\Cong(K)}=\til U  \ar[r,"\tilde\lambda"] & \widetilde{Y_T}
 \end{tikzcd}
\]
and define $\mathfrak{q}_0=\psi(\fc_0)\in \widetilde{Y_T}$. Since the diagram of maps
\begin{equation*}
 \begin{tikzcd}[column sep=40pt]
	                     & {\Cong(L/\mathfrak{c}_0)}  \ar[r,"\tilde\iota"] &\widetilde{T} \\
	{\Cong(L)}\ar[r,"\overline{\phi'_{T,t_0}}"] \ar[ur] & \til U \ar[r,"\tilde\lambda"] \ar[d,"\pi_{U}"] \ar[u,"\simeq"'] & \widetilde{Y_T}\ar[d,"\pi_{Y_T}"] \ar[u,"{\tilde\phi'}"'] \\
	&	U\ar[r,"\lambda"] & Y_T & 
 \end{tikzcd}
\end{equation*}
commutes, we have $\pi_{Y_T}(\mathfrak{q}_0)= y_0$ in $Y_T$ and $\mathfrak{z}_0:={\tilde\phi'}(\mathfrak{q}_0)=\mathfrak{c}_\triv=\sigma_T(t_0)$ in $\widetilde{T}$. Let $Z=\overline{\{\mathfrak{q}_0\}}\subseteq \widetilde{Y_T}$. Then ${\tilde\phi'}(Z)\subseteq \widetilde{T}$ is closed since ${\tilde\phi}$ is universally closed. Since $\mathfrak{z}_0 \in {\tilde\phi'}(Z)$, we have ${\tilde\phi'}(Z)=\widetilde{T}$. Consider $\mathfrak{z}_1:= \sigma_T(M_A)\in \widetilde{T}$. Then there exists some $\mathfrak{q}_1\in Z$ such that $\mathfrak{z}_1= {\tilde\phi'}(\mathfrak{q}_1)$. Thus we have the induced map on stalks ${\tilde\phi'_{\mathfrak{q}_1}}:\mathcal{O}_{\widetilde{T},\mathfrak{z}_1}\to \mathcal{O}_{\widetilde{Y_T},\mathfrak{q}_1}$. Let $B:=  \mathcal{O}_{\widetilde{Y_T},\mathfrak{q}_1}/\mathfrak{q}_0$. Then $\mathfrak{q}:=S_{\mathfrak{q}_0}^{-1}\mathfrak{q}_0\in{\Cong(B)}$ and $S_\mathfrak{q}^{-1}B=\mathcal{O}_{\widetilde{Y_T},\mathfrak{q}_0}/{\mathfrak{q}_0}$. Let $s: B\to S_\mathfrak{q}^{-1}B$ denote the localization map. We obtain the commutative diagram
	\begin{equation*}
	\xymatrix{
&A=\mathcal{O}_{T,M_A}=\mathcal{O}_{\widetilde{T},\mathfrak{z}_1}\ar[drr]_{\Gamma_{\iota}} \ar[r]^{~~~~~~{\tilde\phi'_{\mathfrak{q}_1}}} & \mathcal{O}_{\widetilde{Y_T},\mathfrak{q}_1}\ar[r]^{q} &\mathcal{O}_{\widetilde{Y_T},\mathfrak{q}_1}/\mathfrak{q}_0=B\ar[d]^{s}\\
		&\qquad\qquad  & &K\cong L/\mathfrak{c}_0=\mathcal{O}_{\widetilde{Y_T},\mathfrak{q}_0}/{\mathfrak{q}_0}
	}
\end{equation*}
Since $s$ is a localization map of the integral monoid $B$, it must be injective. Therefore the above diagram shows that $\Frac B\cong K$. Let $f:= q\circ {\tilde\phi'_{\mathfrak{q}_1}}:A\to B$. Then $f^{-1}(M_B)=f^{-1}(\pi_{Y_T}(\mathfrak{q}_1))=\pi_T(\mathfrak{z}_1)=M_A$, which shows $f:A\to B$ is dominant. Thus by Proposition \ref{val-dom}, $f$ is an isomorphism. Let $g: B\to A$ be its inverse. Then
$$\hat{\nu}:T=\MSpec(A)\overset{g^*}\longrightarrow \MSpec(B)\hookrightarrow Y_T\overset{\nu'}\longrightarrow Y.$$ 
is the required lift, which completes one direction of the proof.

Conversely suppose every test diagram for the morphism $\phi:Y\to X$ has a lift. We consider a base change
	\begin{equation*}
	\xymatrix{
		Y'=X'\times_X Y\ar[r]^{~~~~~~~\psi'} \ar@<-2pt>[d]_{\phi'} & Y\ar@<-2pt>[d]^{\phi} \\
		X'\ar[r]_{\psi} & X
	}
\end{equation*}
Let $Z\subseteq \widetilde{Y'}$ be closed. We want to show that ${\tilde\phi'}(Z)\subseteq \widetilde{X'}$ is closed. Since $\phi $ is quasi-compact, so is $\phi'$. By Proposition \ref{specl}, it suffices to show that ${\tilde\phi'}(Z)$ is closed under specialization. We consider any $\mathfrak{z}_1\in Z$ and let $\mathfrak{p}_1= \tilde\phi'(\mathfrak{z}_1)\in \widetilde{X'}$. For any $\mathfrak{p}_0\in \overline{\{\mathfrak{p}_1\}}$, we need to show that $\mathfrak{p}_0\in {\tilde\phi'}(Z)$. As explained in section \ref{subsection: Stalks and residue fields}, we obtain an injection
\begin{equation*}
	A:= \mathcal{O}_{\widetilde{X'},\mathfrak{p}_0}/\mathfrak{p}_1\overset{i}\longrightarrow  \mathcal{O}_{\widetilde{X'},\mathfrak{p}_1}/\mathfrak{p}_1=k(\mathfrak{p}_1) \overset{{\tilde\phi'}_{\mathfrak{z}_1}}\longrightarrow k(\mathfrak{z}_1)=:L.
\end{equation*}
By Proposition \ref{val-dom2} there exists a valuation monoid $\overline{A}$ dominating $A$ such that $\Frac \overline{A}\cong L$. In other words, there are morphisms $h:A\to \overline{A}$ and $j: \overline{A}\hookrightarrow \Frac \overline{A}\cong L$ such that $h$ is dominant, $\overline{A}$ is a valuation monoid and the diagram 
\begin{equation*}
		\begin{tikzcd}
			A\arrow[hookrightarrow]{r}{i} \arrow{d}{h} & K\arrow[hookrightarrow]{d}{{\tilde\phi'}_{\mathfrak{z}_1}} \\
			\overline{A}\arrow{r}{j} & L
		\end{tikzcd}
\end{equation*}
commutes. We obtain the commutative diagram
\begin{equation*}
\begin{tikzcd}
  U:=\MSpec(L) \ar[rrr,bend left=20pt,"\eta"]\arrow{d}{\iota=j^*}	\ar[rr, hook,"\eta'"']& &	Y'\arrow[r,"\psi'"'] \arrow{d}{\phi'} & Y\arrow{d}{\phi} \\
  T:=\MSpec(\overline{A}) \ar[rrr,bend right=20pt,"\nu"'] \ar[rr,bend right=12pt,pos=0.62,"\nu'"'] \arrow{r}{~~h^*}&	\MSpec(A)\ar[r, hook,"\alpha"]&	X'\arrow{r}{\psi} & X
\end{tikzcd}
\end{equation*}
Its outer square is a test diagram, which has a lift $\hat{\nu}:T \to Y$ by assumption. This lift defines a unique morphism $\hat\nu':T\to Y'$ such that the diagram 
\begin{equation*}
	\xymatrix{
	&		T\ar[dr]^{\hat\nu'}\ar@/_1.0pc/@[black][ddr]_{\nu'}\ar@/^1.0pc/@[black][drr]^{\hat{\nu}} & &\\
&	& Y'\ar[r]^{\psi'} \ar@<-2pt>[d]_{\phi'} & Y\ar@<-2pt>[d]^{\phi} \\
&		&	X'\ar[r]_{\psi} & X
	}
\end{equation*}
commutes. By construction, we have 
\[
 \mathfrak{z}_1=\widetilde{\hat\nu'}(\mathfrak{c}_\triv), \qquad \widetilde{\nu'}(\mathfrak{c}_\triv)=\mathfrak{p}_1={\tilde\phi'}(\mathfrak{z}_1), \qquad  \nu'(M_{\overline{A}})=h^*(M_{\overline{A}})=M_A=\pi_{X'}(\mathfrak{p}_0). 
\]
Since $Z$ is closed and $\mathfrak{c}_\triv\in \widetilde{T}$ is generic, $\widetilde{\hat\nu'}(\widetilde{T})\subseteq Z$. Since $A\overset{h}\to {\overline{A}}$ is an injection and since $h^{-1}(M_{\overline{A}})= M_A$, the induced map $\overline{h}: A/M_A\to {\overline{A}}/M_{\overline{A}}$ is an injection. Consider the map $\widetilde{\overline{h}}_*: { \Cong(A/M_A)}\to { \Cong({\overline{A}}/M_{\overline{A}})}$, as defined in section \ref{subsection: Stalks and residue fields}. Then $\widetilde{\overline{h}^*}(\widetilde{\overline{h}_*}(\mathfrak{p}_0))= \mathfrak{p}_0$. Then ${\tilde\phi'}\big(\widetilde{\hat\nu'}(\widetilde{\overline{h}_*}(\mathfrak{p}_0)\big)={\tilde\nu'}\big(\widetilde{\overline{h}_*}(\mathfrak{p}_0)\big)=\mathfrak{p}_0$, which shows that $\mathfrak{p}_0\in {\tilde\phi'}(Z)$.
\end{proof}

\section{Separated morphisms}
\label{section: Separated morphisms}

In this section, we develop a topological characterization of separated morphisms and extend the characterization in terms of the valuative criterion from finite type morphisms, as treated in \cite{CHWW15}, to arbitrary morphisms.

\begin{defn}
 A morphism $\varphi:Y\to X$ of monoid schemes is \emph{separated} if the diagonal $\Delta_\varphi\colon Y\to Y\times_XY$ is a closed immersion.  A morphism $\varphi:Y\to X$ of monoid schemes is \emph{quasi-separated} if the diagonal $\Delta_\varphi:Y\to Y\times_XY$ is quasi-compact.
\end{defn}

\subsection{Locally closed immersions}
Our topological characterization of separated morphisms relies on some general facts for locally closed immersions.

\begin{defn}
 A morphism $\varphi:Y\to X$ is a \emph{locally closed immersion} if it factors into a closed immersion $Y\to U$ followed by an open immersion $U\to X$.
\end{defn}

\begin{lemma}\label{lemma: closure of diagonal is a vanishing set}
 Let $\varphi:Y\to X$ be a quasi-separated morphism of monoid schemes with diagonal $\Delta_\varphi:Y\to Y\times_XY$ and associated map $\tilde\Delta_\varphi:\til Y\to (Y\times_XY)^\congr$ of congruence spaces. Then $\Delta_\varphi$ is a locally closed immersion and the topological closure of $\im(\tilde\Delta_\varphi)$ is a vanishing set of $(Y\times_XY)^\congr$. 
\end{lemma}

\begin{proof}
 If $U=\MSpec(A)$ is affine open of $Y$, then $\varphi(U)$ is contained in an affine open $V=\MSpec(B)$ of $X$ and the restriction of the diagonal morphism to $\Delta_\varphi:U\to U\times_VU$ corresponds to the surjective monoid morphism $A\otimes_{B}A\to A$. This shows that $\Delta_\varphi:U\to U\times_VU=U\times_XU$ is a closed immersion. 
 
 Note that $U\times_VU=U\times_XU$ is an open subset of $Y\times_XY$. If $\cU$ is an affine open covering of $Y$, then $W=\bigcup_{U\in\cU}U\times_XU$ is open in $Y\times_XY$ and $\Delta_\varphi:Y\to W$ is a closed  immersion. Thus $\Delta_\varphi:Y\to Y\times_XY$ is a locally closed immersion, which establishes the first claim of the lemma.
 
 We turn to the second claim. Since $\Delta_\varphi$ is quasi-compact, Lemma \ref{lemma: vanishing set of the congruence kernel is the vanishing closure of the image} shows that the vanishing closure of $Z=\im(\tilde\Delta_\varphi)$ is the vanishing set of the congruence kernel $\fD=\Congker(\Delta_\varphi)$ of $\Delta_\varphi$. We need to show that every point $\fq\in Z$ is in the image of $\tilde\Delta_\varphi$.
 
 The fibre product $Y\times_XY$ is covered by affine opens of the form $U_1\times_V U_2$ where $U_1$ and $U_2$ are affine opens of $Y$ and $V$ is an affine open of $X$ that contains both $\varphi(U_1)$ and $\varphi(U_2)$. If $\fq\in U_1\times_VU_2$, then $U_1\times_VU_2$ is non-trivial and thus $W=\varphi^{-1}(U_1\times_V U_2)=U_1\cap U_2$ is a nonempty open subset of $Y$. 
 
 The congruence kernel of $\Gamma (U_1\times_V U_2)\to\Gamma W$ is
 \[
  \fd \ = \ \fD(U_1\times_V U_2) \ = \ \{ (a_1\otimes a_2,\ b_1\otimes b_2) \mid a_1a_2=b_1b_2 \text{ in } \Gamma W\}
 \]
 where $a_1\otimes a_2$ and $b_1\otimes b_2$ are elements of $\Gamma (U_1\times_VU_2)=\Gamma U_1\otimes_{\Gamma V}\Gamma U_2$ and where $a_1a_2$ is a shorthand notation for $\res_{U_1,W}(a_1)\cdot \res_{U_2,W}(a_2)$, and similar for $b_1b_2$. By our assumptions, $\fq\in V_\fd$, i.e., $\fd\subseteq\fq$ as congruence on $\Gamma U_1\otimes_{\Gamma V}\Gamma U_2$.
 
 In the next step, we find a prime congruence $\fq'\in \im(\til\Delta_\varphi)$ with $\fq'\subseteq \fq$. Let $Q=I_\fq$ be the vanishing ideal of $\fq$, which is a point of $U_1\times_VU_2$. Let $\pr_1:Y\times_XY\to Y$ be the canonical projection onto the first factor. Then $Q_1=\pr_1(Q)\in U_1$. Let $Q_1'\subseteq Q_1$ be a prime ideal that is in $W$ and $\fq_1'=\sigma_W(Q_1')$ the minimal prime congruence in $\til W$ with vanishing ideal $Q_1'$. Then $\fq'=\til\Delta_\varphi(\fq'_1)\in W\times_VW$ is the minimal prime congruence in $\im(\til\Delta_\varphi)$ with vanishing ideal 
 \[
  I_{\fq'} \ =  \ \{a_1\otimes a_2\mid a_1a_2\in Q_1'\text{ in }\Gamma W\}.
 \]
 We claim that $\fq'\subseteq\fq$ as congruences on $\Gamma U_1\otimes_{\Gamma V}\Gamma U_2$. By the minimality of $\fq'$, it suffices to show that $Q'=I_{\fq'}$ is contained in $Q=I_\fq$. Consider $a\otimes b\in Q'$, i.e., 
 \[
  \res_{U_1,W}(a) \cdot \res_{U_2,W}(b) \ \in \ Q_1' \ \subseteq \ Q_1. 
 \]
 The restriction $\res_{U_2,W}(b)$ is the multiple of an element of the form $\res_{U_1,W}(b')$ by an invertible element in $\Gamma W$. Since $Q_1$ is invariant under the multiplication by units, we conclude that $\res_{U_1,W}(ab')\in Q_1$ and thus $ab\otimes 1\in Q$, as desired. This shows that $\fq'\subseteq\fq$.
 
 By our assumptions $\im(\tilde\Delta_\varphi)$ is closed and thus contains every specialization of $\fq'$. This shows that $\fq$ is in the image of $\tilde\Delta_\varphi$, which concludes the proof.
\end{proof}

\begin{lemma}\label{lemma: locally closed immersion whose image is a vanishing set is closed}
 Let $\varphi:Y\to X$ be a locally closed immersion and $\tilde\varphi:\til Y\to \til X$ the associated map of congruence spaces. If $\im(\tilde\varphi)$ is a vanishing set of $X$, then $\varphi$ is a closed immersion. 
\end{lemma}

\begin{proof}
 Since $\varphi:Y\to X$ is a locally closed immersion, there exists an open subset $U$ of $X$ such that $\varphi$ factors into
 $$\varphi: Y\xrightarrow{\varphi_U}U\xrightarrow{\iota_U} X.$$
 where $\varphi_U$ is a closed immersion and $\iota_U$ is an open immersion. Let $Z:= \im(\tilde\varphi)$. Since $Z$ is a vanishing set of $X$, Theorem \ref {thm: characterization of vanishing sets} implies that there exists a strongly reduced closed subscheme $Y'\overset{j}\hookrightarrow X$ such that $Z=\im(\widetilde{j}:\widetilde{Y'}\to \widetilde{X})$. Since $\im(\tilde\varphi)\subset\til U$, $\widetilde{j}$ factors into
 \begin{equation*}
\widetilde{j}: \widetilde{Y'}\overset{\widetilde{j_U}}\longrightarrow \widetilde{U}\overset{\widetilde{\iota_U}}\longrightarrow \widetilde{X}.
 \end{equation*}
Thus 
\[
 j \ = \ \pi_{X}\circ \widetilde{j}\circ \sigma_{Y'} \ = \ \pi_{X}\circ \widetilde{\iota_U}\circ \widetilde{j_U}\circ \sigma_{Y'} \ = \ \iota_U\circ \pi_U \circ \widetilde{j_U}\circ \sigma_{Y'},
\]
which shows that $j$ factors through $U$ via the map $j_U:=  \pi_U \circ \widetilde{j_U}\circ \sigma_{Y'}:Y'\to U$. Since $Y'$ is strongly reduced (and therefore final among all closed subschemes with support $Z$), $\varphi_U$ factors through $Y'$, i.e., there exists some morphism $\iota': Y\to Y'$ such that the diagram 
\begin{equation*}
	\begin{tikzcd}
	Y\arrow{d}{\iota'}	\arrow{r}{\varphi_U}&  U\arrow{d}{\iota_U} \\
    Y'\arrow{r}{~~j}\arrow{ru}{j_U}&	X
	\end{tikzcd}
\end{equation*}
commutes. Since $j= \iota_U\circ j_U$ is a closed immersion and $\iota_U$ is separated as an open immersion, $j_U$ is a closed immersion. Therefore $j_U$ is separated. Since $\varphi_U=j_U\circ i'$ is a closed immersion, it follows that $\iota'$ is a closed immersion. Since composition of closed immersions is a closed immersion, we conclude that $\varphi = j\circ \iota'$ is a closed immersion.
\end{proof}

\begin{rem}
 It is not enough to ask for $\im(\tilde\varphi)$ to be closed for Lemma \ref{lemma: locally closed immersion whose image is a vanishing set is closed} to hold. For example, let $X=\MSpec(\Fun[x,y]/{\langle xy\rangle})$ and consider the locally closed immersion 
 \[
  \varphi: \ Y = \MSpec(\Fun) \amalg \MSpec(\Fun) \stackrel{\iota_Y}\longrightarrow \mathbb{G}_{m,\Fun} \amalg \mathbb{G}_{m,\Fun} \stackrel{\iota_U}\longrightarrow \{\langle x\rangle, \langle y\rangle\} = U \subseteq  X
 \]
  where the map $\iota_Y$ is induced by maps of the form $\Fun[t^{\pm1}]\to \Fun$ given by $t\mapsto 1$ and the map $\iota_U$ is induced by isomorphisms of the form $\Fun[s,t^{\pm1}]/\langle s\rangle \to \Fun[t^{\pm1}]$. It follows that $\iota_Y$ is a closed immersion and $\iota_U$ is an open immersion and hence $\varphi$ is a locally closed immersion. The image $\widetilde{\varphi}(\widetilde{Y})= V_{x,1}\cup V_{y,1}$ is a closed subset of $\widetilde{X}$, but $\varphi$ is not a closed immersion since it is not affine.
\end{rem}

\subsection{Topological characterization of separated morphisms}
\label{subsection: Topological characterization of separated morphisms}

\begin{thm}\label{thm: topological characterization of separated morphisms}
 Let $\varphi:Y\to X$ be a morphism of monoid schemes with diagonal $\Delta_\varphi:Y\to Y\times_XY$ and associated map $\tilde\Delta_\varphi:\til Y\to (Y\times_XY)^\congr$ of congruence spaces. Then the following are equivalent:
 \begin{enumerate}
  \item\label{sep1} $\varphi$ is separated;
  \item\label{sep2} $\varphi$ is quasi-separated and the image of $\tilde\Delta_\varphi$ is a closed subset of $(Y\times_XY)^\congr$.
 \end{enumerate}
\end{thm}

\begin{proof}
 Suppose $\varphi$ is separated. Then $\tilde\Delta_\varphi$ is a closed immersion and therefore quasi-compact and by Proposition \ref{prop: closed immersions of congruence spaces} the image of $\tilde\Delta_\varphi$ is a closed subset of $(Y\times_XY)^\congr$. This shows that \eqref{sep1} implies \eqref{sep2}.
 
 Assume \eqref{sep2}. Then by Lemma \ref{lemma: closure of diagonal is a vanishing set}, we know that $\Delta_\varphi:Y\to Y\times_XY$ is a locally closed immersion and $\im(\tilde\Delta_\varphi)$ is a vanishing set of $(Y\times_XY)^\congr$. Therefore $\Delta_\varphi$ is a closed immersion by Lemma \ref{lemma: locally closed immersion whose image is a vanishing set is closed}. This shows that \eqref{sep2} implies \eqref{sep1} and completes the proof.
\end{proof}

\subsection{The valuative criterion for separated morphisms}

\begin{thm}\label{thm: valuative criterion of separatedness}
 A morphism $\varphi:Y\to X$ of monoid schemes is separated if and only if
 \begin{enumerate}
 	\item\label{valsep1} the morphism $\varphi$ is quasi-separated, and 
 	\item\label{valsep2} every test diagram 
    \begin{equation*}
     \begin{tikzcd}[column sep=60]
      U=\MSpec(K)\ar[r,"\eta"] \ar[d,"\iota"'] & Y\ar[d,"\phi"] \\
      T=\MSpec(A)  \ar[r,"\nu"] & X
     \end{tikzcd}
    \end{equation*}
 	has at most one lift.
 \end{enumerate}
\end{thm}

\begin{proof}
 Suppose $\varphi:Y\to X$ satisfies \eqref{valsep1} and \eqref{valsep2}. Then $\Delta_\varphi:Y\to Y\times_XY$ is quasi-compact. By Theorems \ref{uni-closed} and \ref{thm: topological characterization of separated morphisms}, it suffices to show that the morphism $\Delta_\varphi:Y\to Y\times_XY$ satisfies the existence part of the valuative criterion. Consider a test diagram 
 	\begin{equation*}
 		 \begin{tikzcd}[column sep=40, row sep=15]
 			U=\MSpec(K)\ar[r,"\eta"] \ar[d,"\iota"'] & Y\ar[d,"\Delta_\varphi"] \\
 			T= \MSpec(A)\ar[r,"\nu"] & Y\times_X Y
 		\end{tikzcd}
 	\end{equation*}
 for $\Delta_\varphi$. Let $p_1, p_2: Y\times_X Y\longrightarrow Y$ be the canonical projections. Then both maps $a= p_1 \circ \nu$ and $b=p_2\circ \nu$ are lifts for the test diagram 
 \[
 \begin{tikzcd}[column sep=40, row sep=15]
 	U\ar[r,"\eta"] \ar[d,"\iota"'] & Y\ar[d,"\varphi"] \\
 	T\ar[r,"\varphi\circ a= \varphi\circ b"] & X
 \end{tikzcd}
 \]
 and consequently equal by \eqref{valsep2}. Hence $a=b:\MSpec(A)\longrightarrow Y$ is a lift for the the test for $\Delta_\varphi$. This shows that $\varphi$ is separated.
 
 Conversely, let $\varphi:Y\to X$ be a separated morphism. By Theorem \ref{thm: topological characterization of separated morphisms}, $\varphi$ is quasi-separated, which establishes \eqref{valsep1}. Consider two lifts $a,b:\MSpec(A)\longrightarrow Y$ for the test diagram in \eqref{valsep2}. Let $Z\subseteq \MSpec(A)$ be the equalizer of $a$ and $b$. In other words, $Z$ is the fibre product of $Y$ and $T$ over $Y\times_XY$. Thus the projection $p_2:Z=Y\times_{Y\times_XY}T\to T$ is the base change of the closed immersion $\Delta_\varphi$ along $\nu:T\to Y\times_XY$ and therefore itself a closed immersion. By Proposition \ref{prop: closed immersions of congruence spaces}, $\widetilde{Z}$ is homeomorphic to a closed subset of $\Cong(A)$. 
 
 Since $a\circ\iota=\eta=b\circ\iota$, the image of $\tilde\iota:\til U\to \til T$ is contained in $\til Z$. Thus the generic point of $\til T$, which is in $\til U$, is in $\til Z$. We conclude that $\widetilde{Z} = \til T$ and $Z = \pi_Z(\widetilde{Z})= \pi_{T}(\til T)= T$. Hence $a=b$, which establishes \eqref{valsep2}.
\end{proof}

\section{Proper morphisms}
\label{section: Proper morphisms}

In this section, we introduce proper morphisms of monoid schemes and characterize them in terms of the valuative criterion for properness in the finite type case. This shows, in particular, that our notion of proper morphisms extends that of Corti\~naz, Haesemeyer, Walker and Weibel; cf.\ \cite[Definition 8.4]{CHWW15}.

\begin{defn}
 A morphism $\varphi:Y\to X$ of monoid schemes is \emph{proper} if it is separated, of finite type and universally closed.
\end{defn}

\subsection{The valuative criterion of properness}
\label{subsection: The valuative criterion of properness}

\begin{thm}\label{prop-intrinsic}
 Let $\phi:Y\to X$  be a finite type morphism of monoid schemes. Then $\phi$ is proper if and only if every test diagram 
 \[
  \begin{tikzcd}[column sep=40, row sep=15]
   U\ar[r,"\eta"] \ar[d,"\iota"'] & Y\ar[d,"\phi"] \\
   T\ar[r,"\nu"] & X
  \end{tikzcd}
 \]
 has a unique lift.
\end{thm}

\begin{proof}
 By Theorem \ref{thm: valuative criterion of separatedness}, $\phi$ is separated if and only if every test diagram has at most one lift. Since $\varphi$ is of finite type, it is quasi-compact. Thus we can apply Theorem \ref{uni-closed}, ${\tilde\phi}:  \widetilde{Y}\to \widetilde{X}$, which shows that $\varphi$ is universally closed if and only if every test diagram has a lift.
\end{proof}

\begin{cor}\label{complete}
 Let $X$ be a monoid scheme and let $k$ be an algebraically closed field. Then $X\to \MSpec(\mathbb{F}_1)$ is proper if and only if $X_k$ is a complete $k$-variety.
\end{cor}

\begin{proof}
 By Theorem \ref{prop-intrinsic}, the morphism $X\to \MSpec(\mathbb{F}_1)$ is proper if and only if every test diagram for $X\to \MSpec(\mathbb{F}_1)$ has a unique lift and if $X$ is of finite type. By  \cite[Theorem 8.9]{CHWW15}, every test diagram for $X\to \MSpec(\mathbb{F}_1)$ has a unique lift if and only if every test diagram for $X_k\to  \Spec(k)$ has a unique lift. By \cite[Lemma 3.5]{Chu-Lorscheid-Santhanam12}, $X$ is of finite type if and only if $X_k$ is of finite type. By the valuative criterion for properness for usual schemes, this is, in turn, equivalent to $X_k\to \Spec(k)$ being proper, i.e.\ $X_k$ being a complete $k$-variety.
\end{proof}

\begin{lem}\label{pro-prop} \ 
\begin{enumerate}
	\item\label{proper1} Closed immersions are proper.
	\item\label{proper2} Proper morphisms are closed under composition and base change.
	\item\label{proper3} Properness is local on the base.
\end{enumerate}
\end{lem}
\begin{proof}
We begin with \eqref{proper1}. By definition, closed immersions are of finite type. By Lemma \ref{lemma: closes immersions are universally closed}, closed immersions are universally closed. Closed immersions are affine and thus separated. Thus closed immersions are proper.

We continue with \eqref{proper2}. Finite type morphisms and universally closed morphisms are clearly stable under composition. If $Y\to X$ and $X\to Z$ are separated, then the composition 
\[
 Y\xrightarrow{\Delta_{Y/X}}  Y\times_X Y\longrightarrow Y\times_Z Y
\]
is a closed immersion because the composition of affine morphisms is affine and the composition of surjections of sheaves is surjective. This shows that proper morphisms are closed under composition. Proper morphisms are also closed under base change because the property of being finite type, separated and universally closed are all stable under base change. 

Finally we turn to \eqref{proper3}. All defining properties of proper morphisms, finite type, separated and universally closed, are local on the base and therefore so is properness.
\end{proof}

We recall from \cite[\S 7]{CHWW15} that a morphism $Y\to X$ of monoid schemes is \emph{projective}, if locally on $X$, it factors as a closed immersion $Y\to {\mathbb{P}^n_X}$ for some $n\geq 0$ followed by the projection $ {\mathbb{P}^n_X}\to X$.
\begin{cor}
	Projective morphisms are proper.
\end{cor}
 \begin{proof}
 	Since closed immersions are proper by Lemma \ref{pro-prop}, it is enough to show that morphisms ${\mathbb{P}^n_X}\to X$ are proper. Since $\mathbb{P}^n_X\cong \mathbb{P}^n_{\mathbb{F}_1}\times_{\mathbb{F}_1}X$, it also suffices to show that morphisms ${\mathbb{P}^n_{\mathbb{F}_1}}\to {\mathbb{F}_1}$ are proper. We know that  ${\mathbb{P}^n_{\mathbb{C}}}\to {\mathbb{C}}$ are proper morphisms. Therefore the result follows from Corollary \ref{complete}.
\end{proof}

\appendix

\section{Alternatives, ramifications, and why not}

Certain concepts from ring theory and algebraic geometry allow for different generalizations to monoids and monoid schemes. From the outset, it is not clear which choices to make---and this is one of the struggles in the development of $\Fun$-geometry.

We comment in this appendix on alternative perspectives on the congruence space and why we choose to dismiss them.

\subsection{Prime congruences}

There are alternative concepts for prime congruences, which alters the notion of points of the congruence space. 

A more restrictive choice of prime congruences is considered by Manoel Jarra in \cite{Jarra23b} and \cite{Jarra23a}, which leads to subspaces of the congruence space $\til X$ of a monoid scheme $X$. This theory has different goals in mind than the present text: Jarra shows that his subspaces reflect suitably the dimension of monoid schemes and link to Smirnov's approach to the ABC-conjecture. 

There is also a more lax notion of prime congruence. To facilitate the language in our discussion, we say that a congruence $\fc$ on a monoid $A$ is \emph{weak prime} if $A/\fc$ is without zero divisors. As a first observation, note that every prime congruence is weakly prime, but not vice versa. For example, the trivial congruence on $A=\{0,e,1\}$ with idempotent $e=e^2$ is weak prime, but not prime.

We define the \emph{weak congruence space of $A$} as the topological space $\Cong^w(A)$ of all weak prime congruence of $A$ together with the topology generated by the open subsets of the form
\[
 U_{a,b} \ = \ \{ \fp\in\Cong^w(A)\mid (a,b)\notin\fp\}
\]
with $a,b\in A$. It comes with a topological embedding $\iota_A:\Cong(A)\hookrightarrow\Cong^w(A)$.

A large part of the theory of this paper transfer \textit{mutatis mutandis} to weak congruence spaces: the weak congruence space generalizes to monoid schemes $X$ in terms of open coverings, which yields a topological space $\til X^w$ and a commutative diagram of continuous maps
 \[
  \begin{tikzcd}[column sep=60pt, row sep=15pt]
   {\til X} \ar[r,>->,"\iota_X"] \ar[dr,->>,"\pi_X"'] & {\til X^w} \ar[d,->>,"\pi_X^w"] \\
                                                       & X
  \end{tikzcd}
 \]
Open (closed) immersions induce open (closed) topological embeddings of weak congruence spaces. The topological characterization of closed immersions and separated morphisms carries over to weak congruence spaces.

\begin{rem}[Uncertainty about closed morphisms and valuative criteria]
At the point of writing, we do not know, however, if the weak congruence space characterizes the same class of closed morphisms as the congruence space. It seems very possible that they do, and in this case, the weak congruence space would be a natural extension of the congruence space. 

If not, then this infers a discrepancy for (universally) closed and proper morphisms. The valuative criteria for these classes of morphisms serve as a strong argument that the definitions of this paper are satisfactory. This is our primary reason to work with the congruence space $\til X$ in this paper opposed to the weak congruence space $\til X^w$.

Zooming in into the proof of the valuative criteria of universally closed and proper morphisms, we see that a direct transfer of the arguments to the setting of weak congruence spaces faces the difficulty that the congruence kernel of morphisms into valuation monoids are prime congruences. This means that weak prime congruence that are not prime are not seen by test diagrams in the sense of this paper.

A natural variant that might come to mind is a larger class of test objects, such as pointed monoids without zero divisors whose integral closure is a valuation monoid. A simple example of this nature is the projection $f:A=\{0,e,1\}\to \Fun$, which sends the idempotent element $e=e^2$ of $A$ to $f(e)=1$. However, this particular choice of test object fails to satisfy the lifting property for the induced closed immersion $\varphi=f^\ast:U=\MSpec(\Fun)\to\MSpec(A)=T$ and the tautological test diagram
\[
 \begin{tikzcd}[column sep=80pt, row sep=20pt]
  U \ar[r,"\id_U"] \ar[d,"\varphi"'] & U \ar[d,"\varphi"] \\
  T \ar[r,"\id_T"] & T
 \end{tikzcd}
\]
since $\varphi:U\to T$ is not an isomorphism.
\end{rem}

To conclude, it might be possible that the weak congruence space leans itself for a characterization of closed morphisms (in the sense of this paper) and therefore is \emph{a posteriori} a suitable extension of the notion of congruence space. \emph{A priori}, we do not know whether this holds and therefore refrain from a consideration of weak congruence spaces in the main body of this paper.

\subsection{Structure sheaf for the congruence space}
\label{subsection: structure sheaf for the congruence space}

There are approaches to endow the congruence space $\til X=\Cong A$ of a monoid $A$ with a structure sheaf of local coordinate monoids. In contrast to the structure sheaf of the prime spectrum $X=\MSpec A$, this is more subtle since it is not clear what the coordinate monoids for the open subsets of the form $U_{a,b}=\{\fp\mid (a,b)\notin\fp\}$ should be. 

There are different solutions to this, which we discuss in the following. However, in every case the conclusion is that a refinement of the structure sheaf of $\cO_X$ to a sheaf $\cO_{\til X}$ on $\til X$ results in the fact that for some pointed monoids $A$, the canonical map $A\to\cO_{\til X}(\til X)$ is not an isomorphism.

Topos theory sheds light on this phenomenon: if we endow the category $\cA$ of affine monoid schemes with the Grothendieck topology of all covering families $\{U_i\to X\}_{i\in I}$ by principal opens that appear in the canonical topology for $\cA$, then we find that each such covering has an $i\in I$ for which $U_i\to X$ is an isomorphism. As a consequence, the locale of open subschemes of $X$ is the locale of the prime spectrum of $A=\Gamma X$; in essence this follows from Marty's work in \cite{Marty07}. This means that any refinement of the topology of $\MSpec(A)$ towards $\Cong(A)$ destroys the property of $A\to\cO_X(X)$ being an isomorphism.

This insight is the reason for us to consider the congruence space merely as a topological space and to dismiss the notion of a structure sheaf. Nevertheless, we describe some natural sheaves on the congruence space in the following. The main reason for the existence of different natural choices is that there are several reasonable alternatives for coordinate monoids of open subsets of the form $U_{a,b}$ of the congruence space $\Cong(A)$ of a pointed monoid.

\subsubsection{Pullback sheaf}
\label{subsubsection: pullback sheaf}

The first candidate for a structure sheaf for the congruence space $\til X$ of a monoid scheme $X$ that might come to mind is the pullback $\cO_{\til X}=\pi_X^\ast\cO_X$ of the structure sheaf of $X$. Explicitly $\cO_{\til X}$ is the sheafification of the presheaf with values
\[
 \cO_{\til X}(V) \ = \ \colim \cO_X(U)
\]
where $U$ ranges over all open subsets of $X$ that contain $\pi_X(V)$, modulo the identifications coming from the restriction maps. 

For general $X$, a sheafification of the presheaf pullback is necessary. This means, in particular, that for affine $X=\MSpec(A)$, the canonical morphism $A=\cO_X(X)\to \cO_{\til X}(\til X)$ is in general not an isomorphism, as the following example shows.

\begin{ex}\label{ex: sheaf for Cong A part 1}
 Consider $A=\{0,e,1\}$ with idempotent $e^2=e$. Then $X=\MSpec(A)$ has two prime ideals $\gen 0=\{0\}$ and $\gen e=\{0,e\}$ and $\til X$ has two prime congruences $\fp_0=\gen{(e,1)}$ and $\fp_e=\gen{(e,0)}$. The map $\pi_X:\til X\to X$ is a bijection with $\pi_X(\fp_0)=\gen 0$ and $\pi_X(\fp_e)=\gen e$, but not a homeomorphism since $U_{e,1}=\{\fp_e\}$ is open in $\til X$, but $\pi_X(U_{e,1})=\{\gen e\}$ is not open in $X$. 
 
 In particular, $\til X=U_{e,0}\amalg U_{e,1}=\{\fp_e\}\amalg\{\fp_0\}$ is discrete and thus $\cO_{\til X}(\til X)=\cO_{\til X}(U_{e,0})\times\cO_{\til X}(Y_{e,1})$. We find that
 \[
  \cO_{\til X}(U_{e,0}) \ = \ \cO_{X}(\{\gen 0\}) \ = \ \Fun \qquad \text{and} \qquad \cO_{\til X}(U_{e,1}) \ = \ \cO_{X}(X) \ = \ A.
 \]
 We see that in this case the canonical map $A=\cO_X(X)\to\cO_{\til X}(\til X)=\Fun\times A$ is injective, but not surjective.
\end{ex}

\subsubsection{Restriction to a coarser topology}
\label{subsubsection: restriction to coarses topology}

The failure of the presheaf pullback of $\cO_X$ to $\til X$ from section \ref{subsubsection: pullback sheaf} is caused by open subsets of $\til X$ that are not inverse images of open subsets of $X$. 

There is a trivial solution to fix this failure: we consider $\til X$ together with the topology generated by the inverse image of open subsets of $X$. In the affine case $X=\MSpec(A)$, these are unions of open subsets of $\til X$ of the form $U_{a,0}$ with $a\in A$.

For this coarser topology on $\til X$, the presheaf pullback $\cO_{\til X}'$ of $\cO_X$ along $\pi_X$ is a sheaf and satisfies 
\[
 \cO_{\til X}' (U_{a,0}) \ = \ \cO_X(U_a) \ = \ A[a^{-1}].
\]
In particular, we have $\cO_{\til X}'(\til X)=A$ in the affine case $\til X=\Cong(A)$. This point of view does, however, not provide any additional insight into the nature of monoid schemes and is therefore not of interest.

\subsubsection{Berkovich's approach: open subsets with coordinate monoid}
\label{subsubsection: Berkovich}

In contrast to the sheaf $\cO_{\til X}$ from section \ref{subsubsection: pullback sheaf}, which endows open subsets of $\til X$ with (colimits of) local coordinate monoids for $X$, Berkovich takes another route in \cite{Berkovich12a}.

Namely, there are more open subsets of $\til X$ than the inverse images of opens in $X$ (as considered in section \ref{subsubsection: restriction to coarses topology}) that come with a canonical coordinate monoid. Let us call an open subset $U$ of $\til X$ \emph{affine-like} if there is a monoid morphism $f:A\to A_U$ such that the induced morphism $\varphi=f^\ast:Y\to X$ with $Y=\MSpec A_U$ satisfies the following properties:
\begin{enumerate}
 \item The image of the associated map $\tilde\varphi:\til Y\to\til X$ between congruence spaces has image $U$.
 \item Every morphism $\psi:Z\to X$ with $\im(\tilde\psi)\subset U$ factors uniquely through $\varphi:Y\to X$.
 \item For every $y\in \til Y$ and $x=\varphi(y)$, the induced morphism $k(y)\to k(x)$ be the respective residue fields is an isomorphism.
\end{enumerate}

For such open subsets $U$ of $\til U$, the monoid $A_U$ is a canonical choice of coordinate monoid. In fact, affine-like open subsets of $\til X$ are closed under finite intersections and form a basis for a topology on $\til X$ that is in general strictly coarser than the topology of all open subsets and strictly finer than the topology induced from $X$ (as considered in section \ref{subsubsection: restriction to coarses topology}).

If we attempt to extend the sheaf $\cO_{\til X}$ to this finer topology, then we gain only a presheaf that fails in general to satisfy the sheaf axiom. Sheafification defines a sheaf of monoids on this finer topology, whose global sections do not necessarily agree with $A$, as the following example shows. 

\begin{ex}\label{ex: sheaf for Cong A part 2}
 We reconsider the monoid $A=\{0,e,1\}$ with idempotent $e=e^2$ from Example \ref{ex: sheaf for Cong A part 1}. The only opens of $\til X=\Cong A$ that stem from $X$ are $\emptyset$, $\til X$ and $U_{e,0}=\{\gen{e\sim 1}\}$. The open subset $U_{e,1}=\{\gen{e\sim0}\}$ is affine-like as witnessed by the morphism $A\to A/\gen{e\sim 0}\simeq\Fun$. Thus the coordinate monoid of $U_{e,1}$ is $\Fun$, equally to $\cO_{\til X}(U_{e,0})\simeq\Fun$. The monoid of global sections of the sheaf on affine-like opens is $\Fun\times\Fun$, which contains $A$ as a proper submonoid.
\end{ex}

\subsubsection{Deitmar: sheaf for all opens via enriched structures}

Deitmar uses in \cite{Deitmar13} a richer structure, called sesquiads, that incorporate additive relations, which allows him to define localizations $A[\frac{1}{a-b}]$ for difference of elements $a$ and $b$ in $A$. This generalized localization $A[\frac{1}{a-b}]$ serves as the coordinate object for $U_{a,b}$. Also this theory fails to be subcanonical; in particular, the global sections of the resulting structure sheaf on $\Cong A$ do in general not equal $A$.


\begin{small}
 \bibliographystyle{plain}
 \bibliography{f1}

\begin{thebibliography}{10}

\bibitem{Berkovich12a}
Vladimir~G. Berkovich.
\newblock Algebraic geometry over $\mathbb{F}_1$.
\newblock Unfinished project. Online available at
  \url{https://www.wisdom.weizmann.ac.il/~vova/Algeom_2012.pdf}, 2012.

\bibitem{Berkovich12b}
Vladimir~G. Berkovich.
\newblock Analytic geometry over $\mathbb{F}_1$.
\newblock Unfinished project. Online available at
  \url{https://www.wisdom.weizmann.ac.il/~vova/Angeom_2012.pdf}, 2012.

\bibitem{Chu-Lorscheid-Santhanam12}
Chenghao Chu, Oliver Lorscheid, and Rekha Santhanam.
\newblock Sheaves and {$K$}-theory for {$\mathbb F_1$}-schemes.
\newblock {\em Adv. Math.}, 229(4):2239--2286, 2012.

\bibitem{Cohn04}
Henry Cohn.
\newblock Projective geometry over {$\mathbb F_1$} and the {G}aussian binomial
  coefficients.
\newblock {\em Amer. Math. Monthly}, 111(6):487--495, 2004.

\bibitem{Connes-Consani10a}
Alain Connes and Caterina Consani.
\newblock Schemes over {$\mathbb F_1$} and zeta functions.
\newblock {\em Compos. Math.}, 146(6):1383--1415, 2010.

\bibitem{CHWW14}
Guillermo Corti{\~n}as, Christian Haesemeyer, Mark~E. Walker, and Charles
  Weibel.
\newblock The {$K$}-theory of toric varieties in positive characteristic.
\newblock {\em J. Topol.}, 7(1):247--286, 2014.

\bibitem{CHWW15}
Guillermo Corti{\~n}as, Christian Haesemeyer, Mark~E. Walker, and Charles
  Weibel.
\newblock Toric varieties, monoid schemes and cdh descent.
\newblock {\em J. Reine Angew. Math.}, 698:1--54, 2015.

\bibitem{CHWW18}
Guillermo Corti{\~n}as, Christian Haesemeyer, Mark~E. Walker, and Charles
  Weibel.
\newblock The {$K$}-theory of toric schemes over regular rings of mixed
  characteristic.
\newblock In {\em Singularities, algebraic geometry, commutative algebra, and
  related topics}, pages 455--479. Springer, Cham, 2018.

\bibitem{Deitmar05}
Anton Deitmar.
\newblock Schemes over {$\mathbb F\sb 1$}.
\newblock In {\em Number fields and function fields---two parallel worlds},
  volume 239 of {\em Progr. Math.}, pages 87--100. Birkh\"auser Boston, Boston,
  MA, 2005.

\bibitem{Deitmar07}
Anton Deitmar.
\newblock Remarks on zeta functions and {$K$}-theory over {${\mathbb F}\sb 1$}.
\newblock {\em Proc. Japan Acad. Ser. A Math. Sci.}, 82(8):141--146, 2006.

\bibitem{Deitmar08}
Anton Deitmar.
\newblock {$\mathbb F_1$}-schemes and toric varieties.
\newblock {\em Beitr\"age Algebra Geom.}, 49(2):517--525, 2008.

\bibitem{Deitmar13}
Anton Deitmar.
\newblock Congruence schemes.
\newblock {\em Internat. J. Math.}, 24(2):1350009, 46, 2013.

\bibitem{Eberhardt-Lorscheid-Young22}
Jens~Niklas Eberhardt, Oliver Lorscheid, and Matthew~B. Young.
\newblock Algebraic {$K$}-theory and {G}rothendieck-{W}itt theory of monoid
  schemes.
\newblock {\em Math. Z.}, 301(2):1407--1445, 2022.

\bibitem{Flores-Lorscheid-Szczesny17}
Jaret Flores, Oliver Lorscheid, and Matt Szczesny.
\newblock \v {C}ech cohomology over {$\mathbb F_{1^2}$}.
\newblock {\em J. Algebra}, 485:269--287, 2017.

\bibitem{Giansiracusa-Giansiracusa16}
Jeffrey Giansiracusa and Noah Giansiracusa.
\newblock Equations of tropical varieties.
\newblock {\em Duke Math. J.}, 165(18):3379--3433, 2016.

\bibitem{Giansiracusa-Giansiracusa22}
Jeffrey Giansiracusa and Noah Giansiracusa.
\newblock The universal tropicalization and the {B}erkovich analytification.
\newblock {\em Kybernetika (Prague)}, 58(5):790--815, 2022.

\bibitem{Bothmer-Hinsch-Stuhler11}
H.-C. Graf~von Bothmer, L.~Hinsch, and U.~Stuhler.
\newblock Vector bundles over projective spaces. {T}he case {$\mathbb F_1$}.
\newblock {\em Arch. Math. (Basel)}, 96(3):227--234, 2011.

\bibitem{Jarra23b}
Manoel Jarra.
\newblock On smirnov's approach to the abc-conjecture.
\newblock Preprint, 2023.

\bibitem{Jarra23a}
Manoel Jarra.
\newblock Strong congruence spaces and dimension theory in {$\Fun$}-geometry.
\newblock Preprint, 2023.

\bibitem{LopezPena-Lorscheid11}
Javier L\'opez Pe\~na and Oliver Lorscheid.
\newblock Mapping {$\mathbb F_1$}-land: an overview of geometries over the
  field with one element.
\newblock In {\em Noncommutative geometry, arithmetic, and related topics},
  pages 241--265. Johns Hopkins Univ. Press, Baltimore, MD, 2011.

\bibitem{Lorscheid17}
Oliver Lorscheid.
\newblock Blue schemes, semiring schemes, and relative schemes after {T}o\"en
  and {V}aqui\'e.
\newblock {\em J. Algebra}, 482:264--302, 2017.

\bibitem{Lorscheid18}
Oliver Lorscheid.
\newblock Lecture notes on blueprints and tropical scheme theory.
\newblock Lecture notes, {\url
  https://oliver.impa.br/notes/2018-Blueprints/lecturenotes.pdf}, 2018.

\bibitem{Lorscheid23}
Oliver Lorscheid.
\newblock A unifying approach to tropicalization.
\newblock {\em Trans. Amer. Math. Soc.}, 376(5):3111--3189, 2023.

\bibitem{Lorscheid-Salgado16}
Oliver Lorscheid and Cec\'ilia Salgado.
\newblock A remark on topologies for rational point sets.
\newblock {\em J. Number Theory}, 159:193--201, 2016.

\bibitem{Lorscheid-Szczesny18}
Oliver Lorscheid and Matt Szczesny.
\newblock Quasicoherent sheaves on projective schemes over {$\mathbb F_1$}.
\newblock {\em J. Pure Appl. Algebra}, 222(6):1337--1354, 2018.

\bibitem{Maclagan-Rincon18}
Diane Maclagan and Felipe Rinc\'{o}n.
\newblock Tropical ideals.
\newblock {\em Compos. Math.}, 154(3):640--670, 2018.

\bibitem{Maclagan-Rincon20}
Diane Maclagan and Felipe Rinc\'{o}n.
\newblock Tropical schemes, tropical cycles, and valuated matroids.
\newblock {\em J. Eur. Math. Soc. (JEMS)}, 22(3):777--796, 2020.

\bibitem{Maclagan-Rincon22}
Diane Maclagan and Felipe Rinc\'{o}n.
\newblock Varieties of tropical ideals are balanced.
\newblock {\em Adv. Math.}, 410(part A):Paper No. 108713, 44, 2022.

\bibitem{Manin10}
Yuri~I. Manin.
\newblock Cyclotomy and analytic geometry over {$\mathbb F_1$}.
\newblock In {\em Quanta of maths}, volume~11 of {\em Clay Math. Proc.}, pages
  385--408. Amer. Math. Soc., Providence, RI, 2010.

\bibitem{Marty07}
Florian Marty.
\newblock Relative {Z}ariski open objects.
\newblock Preprint, {\tt arXiv:0712.3676}, 2007.

\bibitem{Merida-Angulo-Thas17a}
Manuel M\'{e}rida-Angulo and Koen Thas.
\newblock Deitmar schemes, graphs and zeta functions.
\newblock {\em J. Geom. Phys.}, 117:234--266, 2017.

\bibitem{Merida-Angulo-Thas17b}
Manuel M\'{e}rida-Angulo and Koen Thas.
\newblock The structure of {D}eitmar schemes, {II}. {Z}eta functions and
  automorphism groups.
\newblock {\em C. R. Math. Acad. Sci. Soc. R. Can.}, 39(4):143--153, 2017.

\bibitem{Pirashvili12}
Ilia Pirashvili.
\newblock On the spectrum of monoids and semilattices.
\newblock Preprint, {\tt arXiv:1112.0023}, 2012.

\bibitem{Ray20}
Samarpita Ray.
\newblock Closure operations, continuous valuations on monoids and spectral
  spaces.
\newblock {\em J. Algebra Appl.}, 19(1):2050006, 27, 2020.

\bibitem{Szczesny12b}
Matt Szczesny.
\newblock On the {H}all algebra of coherent sheaves on {$\mathbb P^1$} over
  {$\mathbb F_1$}.
\newblock {\em J. Pure Appl. Algebra}, 216(3):662--672, 2012.

\bibitem{Thas14}
Koen Thas.
\newblock The structure of {D}eitmar schemes, {I}.
\newblock {\em Proc. Japan Acad. Ser. A Math. Sci.}, 90(1):21--26, 2014.

\bibitem{Thas19}
Koen Thas.
\newblock Projective spaces over {$\mathbb F_{1^\ell}$}.
\newblock {\em J. Combin. Des.}, 27(2):55--74, 2019.

\bibitem{Toen-Vaquie09}
Bertrand To{\"e}n and Michel Vaqui{\'e}.
\newblock Au-dessous de {${\rm Spec}\,\mathbb Z$}.
\newblock {\em J. K-Theory}, 3(3):437--500, 2009.

\bibitem{Vezzani12}
Alberto Vezzani.
\newblock Deitmar's versus {T}o\"en-{V}aqui\'e's schemes over {$\mathbb{F}_1$}.
\newblock {\em Math. Z.}, 271(3-4):911--926, 2012.

\end{thebibliography}
\end{small}

\end{document}